\documentclass[11pt]{amsart}

\usepackage{
    amsmath,
    amsfonts,
    amssymb,
    amsthm,
    amscd,
    comment,
    enumitem,
    etoolbox,
    gensymb,    
    mathtools,
    mathdots,
    stmaryrd,
    booktabs,
}
\usepackage[dvipsnames]{xcolor}
\usepackage[all]{xy}

\usepackage[T1]{fontenc}
\usepackage{dsfont}                     
\usepackage[colorlinks=true, linkcolor=blue, citecolor=blue, urlcolor=blue, breaklinks=true]{hyperref}
\hypersetup{
       pdftitle={The quantum spin Brauer category},
       pdfauthor={Peter J. McNamara and Alistair Savage},
}

\DeclareFontFamily{OT1}{pzc}{}
\DeclareFontShape{OT1}{pzc}{m}{it}{<-> s * [1.10] pzcmi7t}{}
\DeclareMathAlphabet{\mathpzc}{OT1}{pzc}{m}{it}

\usepackage{zref-clever}

\zcsetup{
  cap,                      
  nameinlink=false,         
  lastsep = {, and }        
}

\zcRefTypeSetup{assumption}{
    Name-sg = Assumption ,
    name-sg = assumption ,
    Name-pl = Assumptions ,
    name-pl = assumptions ,
}
\zcRefTypeSetup{cor}{
    Name-sg = Corollary ,
    name-sg = corollary ,
    Name-pl = Corollaries ,
    name-pl = corollaries ,
}
\zcRefTypeSetup{defin}{
    Name-sg = Definition ,
    name-sg = definition ,
    Name-pl = Definitions ,
    name-pl = definitions ,
}
\zcRefTypeSetup{enumi}{
    Name-sg = {} ,
    name-sg = {} ,
    Name-pl = {} ,
    name-pl = {} ,
}
\zcRefTypeSetup{equation}{
    Name-sg = {} ,
    name-sg = {} ,
    Name-pl = {} ,
    name-pl = {} ,
}
\zcRefTypeSetup{eg}{
    Name-sg = Example ,
    name-sg = example ,
    Name-pl = Examples ,
    name-pl = examples ,
}
\zcRefTypeSetup{lem}{
    Name-sg = Lemma ,
    name-sg = lemma ,
    Name-pl = Lemmas ,
    name-pl = lemmas ,
}
\zcRefTypeSetup{prop}{
    Name-sg = Proposition ,
    name-sg = proposition ,
    Name-pl = Propositions ,
    name-pl = propositions ,
}
\zcRefTypeSetup{rem}{
    Name-sg = Remark ,
    name-sg = remark ,
    Name-pl = Remarks ,
    name-pl = remarks ,
}
\zcRefTypeSetup{theo}{
    Name-sg = Theorem ,
    name-sg = theorem ,
    Name-pl = Theorems ,
    name-pl = theorems ,
}

\zcsetup{
  countertype = {
    enumi=item, enumii=item, enumiii=item, enumiv=item
  },
  counterresetby = {
    enumii=enumi, enumiii=enumii, enumiv=enumiii
  }
}
\zcRefTypeSetup{item}{
    Name-sg = {} ,
    name-sg = {} ,
    Name-pl = {} ,
    name-pl = {} ,
}

\AddToHook{env/assumption/begin}{\zcsetup{countertype={theo=assumption}}}
\AddToHook{env/cor/begin}{\zcsetup{countertype={theo=cor}}}
\AddToHook{env/defin/begin}{\zcsetup{countertype={theo=defin}}}
\AddToHook{env/eg/begin}{\zcsetup{countertype={theo=eg}}}
\AddToHook{env/lem/begin}{\zcsetup{countertype={theo=lem}}}
\AddToHook{env/prop/begin}{\zcsetup{countertype={theo=prop}}}
\AddToHook{env/rem/begin}{\zcsetup{countertype={theo=rem}}}

\newcommand{\cref}[1]{\zcref{#1}}
\newcommand{\Cref}[1]{\zcref[S]{#1}}

\newcommand\la{\lambda}

\newcommand\C{\mathbb{C}}

\newcommand\N{\mathbb{N}}

\newcommand\Z{\mathbb{Z}}
\newcommand\kk{\Bbbk}
\newcommand\one{\mathds{1}}

\newcommand\bB{\mathbf{B}}

\newcommand\bF{\mathbf{F}}

\newcommand\fg{\mathfrak{g}}

\newcommand\fo{\mathfrak{o}}

\newcommand\fso{\mathfrak{so}}

\newcommand\Cl{\mathrm{Cl}}             

\newcommand\op{\mathrm{op}}

\newcommand\rO{\mathrm{O}}              

\newcommand\Spin{\mathrm{Spin}}         
\newcommand\Pin{\mathrm{Pin}}           

\newcommand\triv{\mathrm{triv}}

\newcommand\md{\textup{-mod}}           

\newcommand\Vset{\mathtt{I}}            

\newcommand\inv{^{-1}}
\newcommand\qbinom[2]{\begin{bmatrix} #1 \\ #2 \end{bmatrix}}

\newcommand\Brauer{\mathpzc{B}}         
\newcommand\cC{\mathcal{C}}

\newcommand\cN{\mathpzc{N}}
\newcommand\cK{\mathpzc{K}}             

\newcommand\QSB{\mathpzc{QSB}}          

\newcommand\Sgo{\mathsf{S}}             
\newcommand\Vgo{{\color{blue} \mathsf{V}}}             

\DeclareMathOperator{\End}{End}

\DeclareMathOperator{\flip}{flip}
\DeclareMathOperator{\Hom}{Hom}

\DeclareMathOperator{\id}{id}
\DeclareMathOperator{\Ind}{Ind}
\DeclareMathOperator{\Kar}{Kar}

\DeclareMathOperator{\Res}{Res}
\DeclareMathOperator{\Span}{span}

\DeclareMathOperator{\Tr}{Tr}
\DeclareMathOperator{\wt}{wt}

\usepackage{tikz}
\usetikzlibrary{arrows.meta}
\usetikzlibrary{decorations.markings,decorations.pathreplacing}
\usetikzlibrary{calc}
\usetikzlibrary{patterns}   
\usepackage{tikz-cd}

\tikzset{
    centerzero/.style={
        >=To,
        line cap = round, line join = round,
        baseline={([yshift=-0.5ex]#1)},
        },
    centerzero/.default={0,0},
    anchorbase/.style={centerzero=current bounding box.center},
}
\tikzset{
    wipe/.style={
        white,
        line width=3pt,
        line cap = butt,    
    }
}
\tikzset{
    over/.style={
        preaction={
            draw=white,
            line cap = butt,    
            line width=4pt,
            -{},                
        }
    },
}
\tikzset{vec/.style={thick,blue}}
\tikzset{spin/.style={black}}
\tikzset{multispin/.style={very thick,black}}
\tikzset{any/.style={thick,red}}

\tikzset{->-/.style={decoration={
  markings,
  mark=at position #1 with {\arrow{>}}},postaction={decorate}}
}
\tikzset{-<-/.style={decoration={
  markings,
  mark=at position #1 with {\arrow{<}}},postaction={decorate}}
}

\newcommand\braidup{to[out=up,in=down]}
\newcommand\braiddown{to[out=down,in=up]}

\newcommand\strandlabel[1]{$\scriptstyle{#1}$}
\newcommand\botlabel[1]{node[anchor=north] {\strandlabel{#1}}}
\newcommand\toplabel[1]{node[anchor=south] {\strandlabel{#1}}}

\newcommand{\altbox}[3]{
    \filldraw[fill=black!20!white] (#1) rectangle (#2) node[midway] {\strandlabel{#3}}
}

\newcommand\bub[2]{
    \draw[#1] (#2)++(0,0.2) arc(90:-270:0.2)
}

\newcommand\bubble[1]{
    \begin{tikzpicture}[centerzero]
        \bub{#1}{0,0};
    \end{tikzpicture}
}
\newcommand\idstrand[1]{
    \begin{tikzpicture}[centerzero]
        \draw[#1] (0,-0.2) -- (0,0.2);
    \end{tikzpicture}
}

\newcommand\poscross[2]{
    \begin{tikzpicture}[centerzero]
        \draw[#2] (0.2,-0.2) -- (-0.2,0.2);
        \draw[wipe] (-0.2,-0.2) -- (0.2,0.2);
        \draw[#1] (-0.2,-0.2) -- (0.2,0.2);
    \end{tikzpicture}
}
\newcommand\negcross[2]{
    \begin{tikzpicture}[centerzero]
        \draw[#1] (-0.2,-0.2) -- (0.2,0.2);
        \draw[wipe] (0.2,-0.2) -- (-0.2,0.2);
        \draw[#2] (0.2,-0.2) -- (-0.2,0.2);
    \end{tikzpicture}
}

\newcommand{\cupmor}[1]{
    \begin{tikzpicture}[centerzero]
        \draw[#1] (-0.15,0.15) -- (-0.15,0) arc(180:360:0.15) -- (0.15,0.15);
    \end{tikzpicture}
}
\newcommand{\capmor}[1]{
    \begin{tikzpicture}[centerzero]
        \draw[#1] (-0.15,-0.15) -- (-0.15,0) arc(180:0:0.15) -- (0.15,-0.15);
    \end{tikzpicture}
}
\newcommand{\mergemor}[3]{
    \begin{tikzpicture}[anchorbase]
        \draw[#1] (-0.197,-0.197) -- (0,0);
        \draw[#2] (0.197,-0.197) -- (0,0);
        \draw[#3] (0,0) -- (0,0.229);
    \end{tikzpicture}
}
\newcommand{\splitmor}[3]{
    \begin{tikzpicture}[anchorbase]
        \draw[#2] (-0.197,0.197) -- (0,0);
        \draw[#3] (0.197,0.197) -- (0,0);
        \draw[#1] (0,0) -- (0,-0.229);
    \end{tikzpicture}
}

\newcommand\jail[2]{
    \begin{tikzpicture}[centerzero]
        \draw[#1] (-0.1,-0.2) -- (-0.1,0.2);
        \draw[#2] (0.1,-0.2) -- (0.1,0.2);
    \end{tikzpicture}
}

\newtheorem{theo}{Theorem}[section]

\newtheorem{prop}[theo]{Proposition}
\newtheorem{lem}[theo]{Lemma}
\newtheorem{cor}[theo]{Corollary}

\theoremstyle{definition}
\newtheorem{defin}[theo]{Definition}
\newtheorem{rem}[theo]{Remark}

\numberwithin{equation}{section}
\allowdisplaybreaks

\setenumerate[1]{label=(\alph*)}          

\newtoggle{comments}
\newtoggle{details}
\newtoggle{detailsnote}

\iftoggle{comments}{%
    \usepackage[notref,notcite]{showkeys}   
    \newcommand{\acomments}[1]{
        \ \\
        {\color{red}
            \textbf{AS:} #1
        }
        \ \\
    }
    \newcommand{\pcomments}[1]{
        \ \\
        {\color{blue}
            \textbf{PM:} #1
        }
        \ \\
    }
}{%
    \newcommand{\acomments}[1]{\ignorespaces}
    \newcommand{\pcomments}[1]{\ignorespaces}
}

\iftoggle{details}{%
    \newcommand{\details}[1]{
        \ \\
        {\color{OliveGreen}
            \textbf{Details:} #1
        }
        \\
    }
}{%
    \newcommand{\details}[1]{\ignorespaces}
}

\begin{document}

\title{The quantum spin Brauer category}

\author{Peter J. McNamara}
\address[P.M.]{
    School of Mathematics and Statistics \\
    University of Melbourne \\
    Parkville, VIC, 3010, Australia
}
\urladdr{\href{http://petermc.net/maths}{petermc.net/maths},
\textrm{\textit{ORCiD}:} \href{https://orcid.org/0000-0001-6111-1511}{orcid.org/0000-0001-6111-1511}}
\email{maths@petermc.net}

\author{Alistair Savage}
\address[A.S.]{
    Department of Mathematics and Statistics \\
    University of Ottawa \\
    Ottawa, ON, K1N 6N5, Canada
}
\urladdr{\href{https://alistairsavage.ca}{alistairsavage.ca}, \textrm{\textit{ORCiD}:} \href{https://orcid.org/0000-0002-2859-0239}{orcid.org/0000-0002-2859-0239}}
\email{alistair.savage@uottawa.ca}

\begin{abstract}
    We introduce a diagrammatic braided monoidal category, the \emph{quantum spin Brauer category}, together with a full functor to the category of finite-dimensional type-$1$ modules for $U_q(\mathfrak{so}(N))$ or $U_q(\mathfrak{o}(N))$.  This functor becomes essentially surjective after passing to the idempotent completion.  The quantum spin Brauer category can be thought of as a quantum version of the spin Brauer category introduced in \cite{MS24}.  Alternatively, it is an enlargement of the Kauffman category, obtained by adding a generating object corresponding to the quantum spin module.
\end{abstract}

\subjclass[2020]{18M15, 18M30, 17B37}

\keywords{Quantum group, quantized enveloping algebra, special orthogonal Lie algebra, spin group, orthogonal group, monoidal category, string diagram, graphical calculus, Deligne category, interpolating category}

\ifboolexpr{togl{comments} or togl{details}}{%
  {\color{magenta}DETAILS OR COMMENTS ON}
}{%
}

\maketitle
\thispagestyle{empty}


\section{Introduction}

For the orthogonal groups, the analogue of Schur--Weyl duality involves the statement that there is a surjective algebra homomorphism
\[
    \mathrm{Br}_r \to \End_{\rO(V)}(V^{\otimes r}),
\]
where $\mathrm{Br}_r$ is the Brauer algebra on $r$ strands.  The quantum version of this statement is that there is a surjective algebra homomorphism
\begin{equation} \label{BMW}
    \mathrm{BMW}_r \to \End_{U_q(\fo(N))}(V^{\otimes r}),
\end{equation}
where $\mathrm{BMW}_r$ is the Birman--Murakami--Wenzl (BMW) algebra on $r$ strands, $U_q(\fo(N)) = U_q(\fso(N)) \rtimes (\Z/2\Z)$ is the quantized enveloping algebra, and $V$ is the quantum analogue of the natural module for the orthogonal group $\rO(V)$, with $N = \dim V$.

More recently, these statements have been incorporated into a more comprehensive approach, where one considers morphisms between \emph{different} powers of the natural module $V$ and rephrases the results in terms of monoidal categories.  More precisely, there is a full and essentially surjective functor
\begin{equation} \label{toomy}
    \Brauer(N) \to \rO(V)\md,
\end{equation}
where $\Brauer(N)$ is the \emph{Brauer category}.  See, for example, \cite[Th.~4.8]{LZ15}.  The category $\Brauer(N)$ is defined for \emph{any} choice of the parameter $N \in \C$.  Its additive Karoubi envelope is Deligne's interpolating category for the orthogonal groups \cite{Del07}.  In the quantum setting, there is a full functor
\begin{equation} \label{Kauffman}
    \cK(N) \to U_q(\fso(N))\md,
\end{equation}
where $\cK(N)$ is the \emph{Kauffman category}.  See \cite[\S7.7]{Tur89} for a definition of the Kauffman category and \cite[Prop.~4.1]{GRS22} for the definition of the above functor.  Fullness can be deduced from the surjectivity of \cref{BMW}.

The functor \cref{Kauffman} is \emph{not} essentially surjective because its image does not contain the quantum analogue of the spin module.  The same issue occurs if we replace the orthogonal group in \cref{toomy} by its double cover, the pin group $\Pin(V)$, or the identity component, the spin group $\Spin(V)$.  The goal of the current paper is to resolve this problem by enlarging the category $\cK(N)$ to take the quantum spin module $S$ into account.  A first step in this direction was taken by Wenzl \cite{Wen12,Wen20}, who described the endomorphism algebra of $S^{\otimes r}$ in terms of coideal subalgebras, also known as iquantum groups; see \cref{iquantum}.  The endomorphism algebras of $S \otimes V^{\otimes r}$ were also studied in \cite{OW02}.  Other partial results were obtained in \cite{Abo22}, and the special case $N=7$ was treated in \cite{Wes08}.

The non-quantum setting was considered in \cite{MS24}, where the Brauer category was enlarged to the \emph{spin Brauer category}.  One then has a full and essentially surjective functor from the spin Brauer category to the category of modules of the spin or pin groups.  The results of the current paper can be viewed as quantum analogues of those results.  We introduce the \emph{quantum spin Brauer category} $\QSB(q,t,\kappa,d_\Sgo)$, where the parameters $q,t,\kappa,d_\Sgo$ are elements of the ground ring.  The definition of this strict monoidal category is diagrammatic, given via a presentation in terms of generators and relations.  While the Kauffman category has one generating object, which should be thought of as a formal version of the quantum analogue of the natural module, the quantum spin Brauer category (which might also be called the \emph{spin Kauffman category}) has an additional generating object corresponding to the quantum spin module.  The parameter $d_\Sgo$ corresponds to the quantum dimension of this additional object.  In \cref{incarnation}, we define a functor
\begin{equation} \label{functorintro}
    \bF \colon \QSB(N) \to U_q(N)\md,
\end{equation}
where $\QSB(N)$ is the quantum spin Brauer category at certain values of the parameters (see \cref{crabby}) and
\[
    U_q(N)
    =
    \begin{cases}
        U_q(\fso(N)) & \text{if $N$ is odd}, \\
        U_q(\fso(N)) \rtimes (\Z/2\Z) & \text{if $N$ is even}.
    \end{cases}
\]
In the case where $N$ is even (type $D$), the nontrivial element of $\Z/2\Z$ acts on $U_q(\fso(N))$ by the nontrivial Dynkin diagram automorphism.  This construction is a quantum version of the pin group, for which there is only one spin module.
We show that this functor $\bF$ is full (\cref{full}). The functor $\bF$ factors through a quotient $\overline{\QSB}$ given by adjoining an extra relation \cref{extra} when $N$ is odd. We show that the induced functor from $\overline{\QSB}$ is essentially surjective after passing to the additive Karoubi envelope (\cref{essential}). In addition we show that the subcategory of $\overline{\QSB}$ in which we omit the braiding morphisms for two copies of the spin module is an interpolating category for the categories of $U_q(N)$-modules (\cref{negligible}).

The constructions of the current paper open some natural directions for future research.  Since $\QSB$ is a braided monoidal category, one can apply the general affinization procedure of \cite{MS21} to define a \emph{quantum affine spin Brauer category}.  Roughly speaking, this corresponds to considering string diagrams in the quantum spin Brauer category on the cylinder.  This affine category acts on the category of $U_q(N)$-modules and provides tools for investigating the translation functors given by tensoring with the quantum spin and vector modules.  In the non-quantum setting, the \emph{affine spin Brauer category} was studied in \cite[\S9]{MS24}.

It would also be interesting to explore the precise connection between the quantum spin Brauer category and various web categories in types $B$ and $D$ that have recently appeared in the literature.  The goal of web categories is closely related to the goals of the current paper.  Namely, one aims to give a presentation of a module category by generators and relations.  The main difference is that web categories typically contain more generating objects (e.g., a generating object for each exterior power of the natural module) and generating morphisms, and one is often able to give a complete presentation, including a complete set of relations.  Thus, the categories are more complicated but give a more complete description.  In contrast, the quantum spin Brauer category is quite simple, but one does not have a complete presentation of $U_q(N)\md$ until one explicitly describes the kernel of the functor \cref{functorintro}.  Webs for the quantum analogues of orthogonal groups were described in \cite{BW23}, but the spin module was excluded there (as in \cref{toomy}).  In type $B$, the theory of webs, including the spin module, was partially developed in \cite{BER24}.  It would be interesting to precisely describe the connection between the quantum spin Brauer category and the web categories developed in those papers.  Roughly speaking, this should involve passing to a partial idempotent completion of the quantum spin Brauer category, where one adds in an object corresponding to each of the quantum antisymmetrizer idempotents given in \cref{sec:antisymmetrizer}.

Other work that is somewhat related to ours is \cite{ST19}, which develops diagrammatics for the representation categories of various coideal subalgebras by enlarging the known web categories in type $A$, developed in \cite{CKM14}.  In some sense, \cite{ST19} goes in the opposite direction as the current paper, seeing usual quantized enveloping algebras appear in diagrammatic endomorphism algebras of modules over coideal subalgebras, whereas we see coideal subalgebras appear as diagrammatic endomorphism algebras of modules over usual quantized enveloping algebras (see \cref{iquantum}).

\subsection*{Acknowledgements}

The research of A.S.\ was supported by Discovery Grant RGPIN-2023-03842 from the Natural Sciences and Engineering Research Council of Canada.  The authors would like to thank Elijah Bodish and Weiqiang Wang for helpful conversations.  Many algebraic computations and verifications were done using SageMath \cite{sagemath}.  The authors used OpenAI’s ChatGPT-5.5 to assist with proofreading the manuscript and to identify possible gaps and suggest approaches to certain intermediate arguments. All such suggestions were independently checked and revised by the authors, who assume full responsibility for the correctness and content of the paper.

\section{The quantum spin Brauer category\label{sec:qspinBrauer}}

In this section, we introduce our main category of interest.   We work over an arbitrary commutative ring $\kk$ and use the usual string-diagram calculus for strict monoidal categories.  Throughout this paper, $N$ denotes a natural number and $n = \lfloor \frac{N}{2} \rfloor$, so that $N=2n$ (type $D$) or $N=2n+1$ (type $B$).

\begin{defin} \label{QSBdef}
    Suppose $d_\Sgo \in \kk$ and $q, t, \kappa \in \kk^\times$, such that $q-q^{-1} \in \kk^\times$.  The \emph{quantum spin Brauer category}, or \emph{spin Kauffman category}, $\QSB = \QSB(q,t,\kappa,d_\Sgo)$ is the strict $\kk$-linear monoidal category presented as follows.  The generating objects are $\Sgo$ and $\Vgo$, whose identity morphisms we depict by a thin black strand and a thick blue strand:
    \[
        \idstrand{spin} := 1_\Sgo,\qquad
        \idstrand{vec} := 1_\Vgo.
    \]
    The generating morphisms are
    \begin{gather*}
        \capmor{spin} \colon \Sgo \otimes \Sgo \to \one,\qquad
        \cupmor{spin} \colon \one \to \Sgo \otimes \Sgo,\qquad
        \capmor{vec} \colon \Vgo \otimes \Vgo \to \one,\qquad
        \cupmor{vec} \colon \one \to \Vgo \otimes \Vgo,
        \\
        \poscross{spin}{spin}, \negcross{spin}{spin} \colon \Sgo \otimes \Sgo \to \Sgo \otimes \Sgo,\qquad
        \poscross{vec}{vec}, \negcross{vec}{vec} \colon \Vgo \otimes \Vgo \to \Vgo \otimes \Vgo,
        \\
        \poscross{vec}{spin}, \negcross{vec}{spin} \colon \Vgo \otimes \Sgo \to \Sgo \otimes \Vgo,\qquad
        \poscross{spin}{vec}, \negcross{spin}{vec} \colon \Sgo \otimes \Vgo \to \Vgo \otimes \Sgo,
        \\
        \mergemor{vec}{spin}{spin} \colon \Vgo \otimes \Sgo \to \Sgo.
    \end{gather*}
    To state the defining relations, we will use the convention that a relation involving $r \ge 1$ red strands (as in \cref{braid,yank}) means we impose the $2^r$ relations obtained from replacing each such strand with either a black strand or a blue strand.  The defining relations on morphisms are then as follows:
    \begin{gather} \label{braid}
        \begin{tikzpicture}[centerzero]
            \draw[any] (0.2,-0.4) to[out=135,in=down] (-0.15,0) to[out=up,in=225] (0.2,0.4);
            \draw[wipe] (-0.2,-0.4) to[out=45,in=down] (0.15,0) to[out=up,in=-45] (-0.2,0.4);
            \draw[any] (-0.2,-0.4) to[out=45,in=down] (0.15,0) to[out=up,in=-45] (-0.2,0.4);
        \end{tikzpicture}
        =
        \begin{tikzpicture}[centerzero]
            \draw[any] (-0.2,-0.4) -- (-0.2,0.4);
            \draw[any] (0.2,-0.4) -- (0.2,0.4);
        \end{tikzpicture}
        =
        \begin{tikzpicture}[centerzero]
            \draw[any] (-0.2,-0.4) to[out=45,in=down] (0.15,0) to[out=up,in=-45] (-0.2,0.4);
            \draw[wipe] (0.2,-0.4) to[out=135,in=down] (-0.15,0) to[out=up,in=225] (0.2,0.4);
            \draw[any] (0.2,-0.4) to[out=135,in=down] (-0.15,0) to[out=up,in=225] (0.2,0.4);
        \end{tikzpicture}
        \ ,\quad
        \begin{tikzpicture}[centerzero]
            \draw[any] (0.4,-0.4) -- (-0.4,0.4);
            \draw[wipe] (0,-0.4) to[out=135,in=down] (-0.32,0) to[out=up,in=225] (0,0.4);
            \draw[any] (0,-0.4) to[out=135,in=down] (-0.32,0) to[out=up,in=225] (0,0.4);
            \draw[wipe] (-0.4,-0.4) -- (0.4,0.4);
            \draw[any] (-0.4,-0.4) -- (0.4,0.4);
        \end{tikzpicture}
        =
        \begin{tikzpicture}[centerzero]
            \draw[any] (0.4,-0.4) -- (-0.4,0.4);
            \draw[wipe] (0,-0.4) to[out=45,in=down] (0.32,0) to[out=up,in=-45] (0,0.4);
            \draw[any] (0,-0.4) to[out=45,in=down] (0.32,0) to[out=up,in=-45] (0,0.4);
            \draw[wipe] (-0.4,-0.4) -- (0.4,0.4);
            \draw[any] (-0.4,-0.4) -- (0.4,0.4);
        \end{tikzpicture}
        \ ,\quad
        \begin{tikzpicture}[centerzero]
            \draw[any] (-0.2,-0.3) -- (-0.2,-0.1) arc(180:0:0.2) -- (0.2,-0.3);
            \draw[wipe] (-0.3,0.3) \braiddown (0,-0.3);
            \draw[any] (-0.3,0.3) \braiddown (0,-0.3);
        \end{tikzpicture}
        =
        \begin{tikzpicture}[centerzero]
            \draw[any] (-0.2,-0.3) -- (-0.2,-0.1) arc(180:0:0.2) -- (0.2,-0.3);
            \draw[wipe] (0.3,0.3) \braiddown (0,-0.3);
            \draw[any] (0.3,0.3) \braiddown (0,-0.3);
        \end{tikzpicture}
        \ ,\quad
        \begin{tikzpicture}[centerzero]
            \draw[any] (-0.3,0.3) \braiddown (0,-0.3);
            \draw[wipe] (-0.2,-0.3) -- (-0.2,-0.1) arc(180:0:0.2) -- (0.2,-0.3);
            \draw[any] (-0.2,-0.3) -- (-0.2,-0.1) arc(180:0:0.2) -- (0.2,-0.3);
        \end{tikzpicture}
        =
        \begin{tikzpicture}[centerzero]
            \draw[any] (0.3,0.3) \braiddown (0,-0.3);
            \draw[wipe] (-0.2,-0.3) -- (-0.2,-0.1) arc(180:0:0.2) -- (0.2,-0.3);
            \draw[any] (-0.2,-0.3) -- (-0.2,-0.1) arc(180:0:0.2) -- (0.2,-0.3);
        \end{tikzpicture}
        \ ,
        \\ \label{yank}
        \begin{tikzpicture}[centerzero]
            \draw[any] (-0.3,0.4) -- (-0.3,0) arc(180:360:0.15) arc(180:0:0.15) -- (0.3,-0.4);
        \end{tikzpicture}
        =
        \begin{tikzpicture}[centerzero]
            \draw[any] (0,-0.4) -- (0,0.4);
        \end{tikzpicture}
        = \
        \begin{tikzpicture}[centerzero]
            \draw[any] (-0.3,-0.4) -- (-0.3,0) arc(180:0:0.15) arc(180:360:0.15) -- (0.3,0.4);
        \end{tikzpicture}
        \ ,
        \\ \label{skein}
        \begin{tikzpicture}[centerzero]
            \draw[vec] (-0.2,-0.4) -- (0.2,0.4);
            \draw[wipe] (0.2,-0.4) -- (-0.2,0.4);
            \draw[vec] (0.2,-0.4) -- (-0.2,0.4);
        \end{tikzpicture}
        -
        \begin{tikzpicture}[centerzero]
            \draw[vec] (0.2,-0.4) -- (-0.2,0.4);
            \draw[wipe] (-0.2,-0.4) -- (0.2,0.4);
            \draw[vec] (-0.2,-0.4) -- (0.2,0.4);
        \end{tikzpicture}
        = (q-q^{-1})
        \left(
            \begin{tikzpicture}[centerzero]
                \draw[vec] (-0.2,-0.4) -- (-0.2,0.4);
                \draw[vec] (0.2,-0.4) -- (0.2,0.4);
            \end{tikzpicture}
            -
            \begin{tikzpicture}[centerzero]
                \draw[vec] (-0.2,-0.4) -- (-0.2,-0.3) arc(180:0:0.2) -- (0.2,-0.4);
                \draw[vec] (-0.2,0.4) -- (-0.2,0.3) arc(180:360:0.2) -- (0.2,0.4);
            \end{tikzpicture}
        \right)
        ,
        \\ \label{deloop}
        \begin{tikzpicture}[anchorbase]
            \draw[spin] (0.15,0) arc(0:180:0.15) to[out=-90,in=120] (0.15,-0.4);
            \draw[wipe] (-0.15,-0.4) to[out=60,in=-90] (0.15,0);
            \draw[spin] (-0.15,-0.4) to[out=60,in=-90] (0.15,0);
        \end{tikzpicture}
        = t \
        \capmor{spin}
        \ ,\quad
        \begin{tikzpicture}[anchorbase]
            \draw[vec] (0.15,0) arc(0:180:0.15) to[out=-90,in=120] (0.15,-0.4);
            \draw[wipe] (-0.15,-0.4) to[out=60,in=-90] (0.15,0);
            \draw[vec] (-0.15,-0.4) to[out=60,in=-90] (0.15,0);
        \end{tikzpicture}
        = \kappa^2 \
        \capmor{vec}
        \ ,
        \\ \label{typhoon}
        \begin{tikzpicture}[centerzero]
            \draw[spin] (0.4,-0.4) to[out=110,in=-20] (-0.4,0.4);
            \draw[vec] (-0.4,-0.4) -- (-0.1,-0.1);
            \draw[wipe] (-0.1,-0.1) -- (0.4,0.4);
            \draw[spin] (0,-0.4) -- (-0.1,-0.1) -- (0.4,0.4);
        \end{tikzpicture}
        =
        \begin{tikzpicture}[centerzero]
            \draw[spin] (0.4,-0.4) to[out=150,in=-60] (-0.4,0.4);
            \draw[wipe] (-0.4,-0.4) -- (0.2,0.2);
            \draw[wipe] (0,-0.4) to[out=45,in=-90] (0.2,0.2);
            \draw[vec] (-0.4,-0.4) -- (0.2,0.2);
            \draw[spin] (0,-0.4) to[out=45,in=-90] (0.2,0.2) -- (0.4,0.4);
        \end{tikzpicture}
         \ ,\qquad
        \begin{tikzpicture}[centerzero]
            \draw[vec] (0.4,-0.4) to[out=110,in=-20] (-0.4,0.4);
            \draw[vec] (-0.4,-0.4) -- (-0.1,-0.1);
            \draw[wipe] (-0.1,-0.1) -- (0.4,0.4);
            \draw[spin] (0,-0.4) -- (-0.1,-0.1) -- (0.4,0.4);
        \end{tikzpicture}
        =
        \begin{tikzpicture}[centerzero]
            \draw[vec] (0.4,-0.4) to[out=150,in=-60] (-0.4,0.4);
            \draw[wipe] (-0.4,-0.4) -- (0.2,0.2);
            \draw[wipe] (0,-0.4) to[out=45,in=-90] (0.2,0.2);
            \draw[vec] (-0.4,-0.4) -- (0.2,0.2);
            \draw[spin] (0,-0.4) to[out=45,in=-90] (0.2,0.2) -- (0.4,0.4);
        \end{tikzpicture}
         \ ,\qquad
        \begin{tikzpicture}[centerzero]
            \draw[vec] (-0.4,-0.4) -- (-0.1,-0.1);
            \draw[spin] (0,-0.4) -- (-0.1,-0.1) -- (0.4,0.4);
            \draw[wipe] (0.4,-0.4) to[out=110,in=-20] (-0.4,0.4);
            \draw[spin] (0.4,-0.4) to[out=110,in=-20] (-0.4,0.4);
        \end{tikzpicture}
        =
        \begin{tikzpicture}[centerzero]
            \draw[vec] (-0.4,-0.4) -- (0.2,0.2);
            \draw[spin] (0,-0.4) to[out=45,in=-90] (0.2,0.2) -- (0.4,0.4);
            \draw[wipe] (0.4,-0.4) to[out=150,in=-60] (-0.4,0.4);
            \draw[spin] (0.4,-0.4) to[out=150,in=-60] (-0.4,0.4);
        \end{tikzpicture}
         \ ,\qquad
        \begin{tikzpicture}[centerzero]
            \draw[vec] (-0.4,-0.4) -- (-0.1,-0.1);
            \draw[spin] (0,-0.4) -- (-0.1,-0.1) -- (0.4,0.4);
            \draw[wipe] (0.4,-0.4) to[out=110,in=-20] (-0.4,0.4);
            \draw[vec] (0.4,-0.4) to[out=110,in=-20] (-0.4,0.4);
        \end{tikzpicture}
        =
        \begin{tikzpicture}[centerzero]
            \draw[vec] (-0.4,-0.4) -- (0.2,0.2);
            \draw[spin] (0,-0.4) to[out=45,in=-90] (0.2,0.2) -- (0.4,0.4);
            \draw[wipe] (0.4,-0.4) to[out=150,in=-60] (-0.4,0.4);
            \draw[vec] (0.4,-0.4) to[out=150,in=-60] (-0.4,0.4);
        \end{tikzpicture}
        \ ,
        \\ \label{swishy}
        \begin{tikzpicture}[anchorbase]
            \draw[spin] (0,0) -- (0,0.1) arc (0:180:0.2) -- (-0.4,-0.6);
            \draw[spin] (0,0) to[out=-45, in=left] (0.2,-0.2) arc(-90:0:0.15) -- (0.35,0.4);
            \draw[vec] (0,0) to[out=-135,in=up] (-0.2,-0.3) to[out=down,in=down] (0.6,-0.3) -- (0.6,0.4);
        \end{tikzpicture}
        =
        \begin{tikzpicture}[anchorbase,xscale=-1]
            \draw[spin] (0,0) -- (0,0.1) arc (0:180:0.2) -- (-0.4,-0.6);
            \draw[vec] (0,0) to[out=-45, in=left] (0.2,-0.2) arc(-90:0:0.15) -- (0.35,0.4);
            \draw[spin] (0,0) to[out=-135,in=up] (-0.2,-0.3) to[out=down,in=down] (0.6,-0.3) -- (0.6,0.4);
        \end{tikzpicture}
        \ ,
        \\ \label{fishy}
        \begin{tikzpicture}[anchorbase]
            \draw[vec] (0.2,-0.5) to [out=135,in=down] (-0.15,-0.2) to[out=up,in=-135] (0,0);
            \draw[wipe] (-0.2,-0.5) to[out=45,in=down] (0.15,-0.2);
            \draw[spin] (-0.2,-0.5) to[out=45,in=down] (0.15,-0.2) to[out=up,in=-45] (0,0) -- (0,0.2);
        \end{tikzpicture}
        = \kappa\
        \begin{tikzpicture}[anchorbase]
            \draw[spin] (0,0) -- (0,0.1) arc (0:180:0.2) -- (-0.4,-0.4);
            \draw[spin] (0,0) to[out=-45, in=left] (0.2,-0.2) arc(-90:0:0.15) -- (0.35,0.4);
            \draw[vec] (0,0) to[out=-135,in=up] (-0.2,-0.4);
        \end{tikzpicture}
        \ ,
        \\ \label{oist}
        \begin{tikzpicture}[centerzero]
            \draw[vec] (-0.4,-0.4) -- (0,0.1);
            \draw[vec] (0,-0.4) -- (0.2,-0.15);
            \draw[spin] (0.4,-0.4) --(0,0.1) -- (0,0.4);
        \end{tikzpicture}
        + q
        \begin{tikzpicture}[centerzero]
            \draw[vec] (0,-0.4) to[out=135,in=-135] (0,0.1);
            \draw[wipe] (-0.4,-0.4) -- (0.2,-0.15);
            \draw[vec] (-0.4,-0.4) -- (0.2,-0.15);
            \draw[spin] (0.4,-0.4) -- (0,0.1) -- (0,0.4);
        \end{tikzpicture}
        = \left( q \kappa^2 + 1 \right)\
        \begin{tikzpicture}[centerzero]
            \draw[vec] (-0.4,-0.4) -- (-0.4,-0.3) arc(180:0:0.2) -- (0,-0.4);
            \draw[spin] (0.4,-0.4) \braidup (0,0.4);
        \end{tikzpicture}
        \ ,
        \\ \label{dimrel}
        \bubble{spin} = d_\Sgo 1_\one,\qquad
        \bubble{vec} = d_\Vgo 1_\one,\quad
        \text{where }
        d_\Vgo = \frac{\kappa^{-2}-\kappa^2}{q-q^{-1}} + 1 = \frac{(q\kappa^2+1)(\kappa^{-2}-q^{-1})}{q-q^{-1}}.
    \end{gather}
    This concludes the definition of $\QSB$.
\end{defin}

\begin{rem}
    The assumption that $q-q^{-1}$ be invertible is needed in \cref{dimrel}.  In \cref{sec:incarnation}, when we define an incarnation functor to a category of modules, we will specialize $\kappa$ to be a power of $q$.  Then $d_\Vgo$ is a Laurent polynomial in $q$ and the definition of $\QSB$ makes sense without the assumption that $q-q^{-1}$ is invertible.  In particular, we could then specialize $q=1$.
\end{rem}

\begin{rem}
    Note the asymmetry in \cref{QSBdef} with regard to the roles played by $d_\Sgo$ and $d_\Vgo$.  While $d_\Sgo$ is left as a parameter, $d_\Vgo$ is written in terms of the other parameters in \cref{dimrel}.  This is because adding a cap to the top of the \emph{skein relation} \cref{skein}, then using \cref{deloop} and the second relation in \cref{dimrel} shows that, in order for $\QSB$ to not collapse to the zero category, we need
    \[
        \kappa^2 - \kappa^{-2} = (q-q^{-1})(1 - d_\Vgo).
    \]
    Since we have no skein relation for the object $\Sgo$, we do not get any analogous restriction on $d_\Sgo$.
\end{rem}

\begin{rem}
    One can justify the presence of the factor of $q\kappa^2+1$ in \cref{oist} in two ways.  First note that, using \cref{deloop}, one can rewrite \cref{oist} as
    \[
        \left(
            \begin{tikzpicture}[centerzero]
                \draw[vec] (-0.4,-0.4) -- (0,0.1);
                \draw[vec] (0,-0.4) -- (0.2,-0.15);
                \draw[spin] (0.4,-0.4) --(0,0.1) -- (0,0.4);
            \end{tikzpicture}
            -
            \begin{tikzpicture}[centerzero]
                \draw[vec] (-0.4,-0.4) -- (-0.4,-0.3) arc(180:0:0.2) -- (0,-0.4);
                \draw[spin] (0.4,-0.4) \braidup (0,0.4);
            \end{tikzpicture}
        \right)
        \circ
        \left(
            \jail{vec}{vec} + q \poscross{vec}{vec}
        \right)
        = 0,
    \]
    and
    \(
        \jail{vec}{vec} + q \poscross{vec}{vec}
    \)
    is the quantum symmetrizer (up to a scalar multiple).  The second justification is that the factor of $q\kappa^2+1$ is needed to have a natural bar involution on the quantum spin Brauer category; see \cref{barinv}.
\end{rem}

\begin{rem}
    At least one of the relations in \cref{braid} is redundant.  Using \cref{skein} and the fourth equality in \cref{braid} for blue strands, we have
    \[
        \begin{tikzpicture}[centerzero]
            \draw[vec] (-0.3,0.3) \braiddown (0,-0.3);
            \draw[wipe] (-0.2,-0.3) -- (-0.2,-0.1) arc(180:0:0.2) -- (0.2,-0.3);
            \draw[vec] (-0.2,-0.3) -- (-0.2,-0.1) arc(180:0:0.2) -- (0.2,-0.3);
        \end{tikzpicture}
        \overset{\cref{skein}}{\underset{\cref{yank}}{=}}
        \begin{tikzpicture}[centerzero]
            \draw[vec] (-0.2,-0.3) -- (-0.2,-0.1) arc(180:0:0.2) -- (0.2,-0.3);
            \draw[wipe] (-0.3,0.3) \braiddown (0,-0.3);
            \draw[vec] (-0.3,0.3) \braiddown (0,-0.3);
        \end{tikzpicture}
        + (q-q^{-1})
        \left(\,
            \begin{tikzpicture}[centerzero]
                \draw[vec] (-0.2,-0.3) -- (-0.2,-0.1) arc(180:0:0.2) -- (0.2,-0.3);
                \draw[vec] (-0.4,0.3) -- (-0.4,-0.3);
            \end{tikzpicture}
            -
            \begin{tikzpicture}[centerzero]
                \draw[vec] (-0.2,-0.3) -- (-0.2,-0.1) arc(180:0:0.2) -- (0.2,-0.3);
                \draw[vec] (0.4,0.3) -- (0.4,-0.3);
            \end{tikzpicture}
        \, \right)
        \overset{\cref{braid}}{=}
        \begin{tikzpicture}[centerzero]
            \draw[vec] (-0.2,-0.3) -- (-0.2,-0.1) arc(180:0:0.2) -- (0.2,-0.3);
            \draw[wipe] (0.3,0.3) \braiddown (0,-0.3);
            \draw[vec] (0.3,0.3) \braiddown (0,-0.3);
        \end{tikzpicture}
        + (q-q^{-1})
        \left(\,
            \begin{tikzpicture}[centerzero]
                \draw[vec] (-0.2,-0.3) -- (-0.2,-0.1) arc(180:0:0.2) -- (0.2,-0.3);
                \draw[vec] (-0.4,0.3) -- (-0.4,-0.3);
            \end{tikzpicture}
            -
            \begin{tikzpicture}[centerzero]
                \draw[vec] (-0.2,-0.3) -- (-0.2,-0.1) arc(180:0:0.2) -- (0.2,-0.3);
                \draw[vec] (0.4,0.3) -- (0.4,-0.3);
            \end{tikzpicture}
        \, \right)
        \overset{\cref{skein}}{\underset{\cref{yank}}{=}}
        \begin{tikzpicture}[centerzero]
            \draw[vec] (0.3,0.3) \braiddown (0,-0.3);
            \draw[wipe] (-0.2,-0.3) -- (-0.2,-0.1) arc(180:0:0.2) -- (0.2,-0.3);
            \draw[vec] (-0.2,-0.3) -- (-0.2,-0.1) arc(180:0:0.2) -- (0.2,-0.3);
        \end{tikzpicture}
        \ .\qedhere
    \]
    implying the last relation in \cref{braid} for blue strands.  Note that a similar argument does not work for black strands, since we do not have an analogue of the skein relation \cref{skein} in that case.  In \cref{QSBdef}, we include the fourth and fifth relations in \cref{braid} for all colours of the strands since, in practice, these relations are easy to verify---they follow immediately from the fact that we work in a braided monoidal category.  See the proof of \cref{incarnation}.
\end{rem}

The relations \cref{yank} imply that $\QSB$ is a rigid monoidal category, with the objects $\Sgo$ and $\Vgo$ being self-dual.  The relations \cref{braid,typhoon} imply that $\QSB$ is braided monoidal, with braiding given by the crossings.  Then \cref{swishy} implies that $\QSB$ is strict pivotal, with duality given by rotating diagrams through $180\degree$.  This means that diagrams are isotopy invariant, and so rotated versions of all the defining relations hold.  For example, we have
\[

    \ .
\]
Throughout this document, we will refer to a relation by its equation number even when we are, in fact, using a rotated version of that relation.  Composing the relations in \cref{deloop} on the bottom with $\negcross{spin}{spin}$ and $\negcross{vec}{vec}$ respectively, then using \cref{braid}, we get
\begin{equation} \label{reloop}
    %
    = \kappa^{-2} \
    \capmor{vec}
\end{equation}

We introduce other trivalent morphisms by successive clockwise rotation:
\begin{equation} \label{vortex}
    \splitmor{spin}{vec}{spin}
    :=
    %
    \ .
\end{equation}
Since $\QSB$ is strict pivotal, the trivalent morphisms are also related in the natural way by counterclockwise rotation:
\begin{equation} \label{vortex2}
    %
        \colon X \otimes Y \to Z
    \]
    is any morphism in a ribbon category, then
    \begin{equation} \label{canal}
        %
        \ .
    \end{equation}
\end{lem}

\begin{proof}
    To simplify diagrams, we will omit the object labels and colour the strands.  (The colours here do not have the same meaning as in $\QSB$.)  We have
    \[
        %
        \ .
    \]
    Now compose on the bottom with $\poscross{vec}{any}$ to get
    \[
        %
        \ .
    \]
    Finally, compose on the bottom with $\poscross{any}{vec}$ and on the top with
    $
        %
        = t^{-1} \kappa\,
        \mergemor{spin}{spin}{vec}
        \ .
    \end{gather}
\end{lem}

\begin{proof}
    The first relation in \cref{lobster1} is simply a rewriting of \cref{fishy}, using \cref{vortex2}.  Then, composing on the bottom of the first relation in \cref{lobster1} with $\negcross{vec}{spin}$ and using the second relation in \cref{braid} gives the second relation in \cref{lobster1}.  Next, we have
    \[
        \kappa
        %
        \overset{\cref{deloop}}{\underset{\cref{reloop}}{=}} \kappa^2 \mergemor{vec}{spin}{spin}
        \ ,
    \]
    where, in the first equality above, we used the first relation in \cref{lobster1}.  This proves the third relation in \cref{lobster1}.  The fourth relation in \cref{lobster1} then follows after composing on the bottom with $\negcross{spin}{vec}$ and using \cref{braid}.

    For the first relation in \cref{lobster2}, we compute
    \[
        %
        \overset{\cref{fishy}}{=} t \kappa^{-1}\
        \mergemor{spin}{spin}{vec}
        \ .
    \]
    The second relation in \cref{lobster2} then follows after composing on the bottom with $\negcross{spin}{spin}$ and using the second relation in \cref{braid}.
\end{proof}

Next we explore some symmetries of the quantum spin Brauer category.

\begin{prop} \label{horizon}
    There is an isomorphism of $\kk$-linear monoidal categories
    \[
        \Omega_{\updownarrow} \colon \QSB(q,t,\kappa,d_\Sgo) \to \QSB(q,t,\kappa,d_\Sgo)^\op
    \]
    given on objects by $\Vgo \mapsto \Vgo$, $\Sgo \mapsto \Sgo$, and on the generating morphisms by
    \begin{gather*}
        \capmor{spin} \mapsto \cupmor{spin}\, ,\quad
        \cupmor{spin} \mapsto \capmor{spin}\, ,\quad
        \capmor{vec} \mapsto \cupmor{vec}\, ,\quad
        \cupmor{vec} \mapsto \capmor{vec}\, ,\quad
        \mergemor{vec}{spin}{spin} \mapsto \splitmor{spin}{vec}{spin},
        \\
        \poscross{spin}{spin} \mapsto \poscross{spin}{spin},\quad
        \negcross{spin}{spin} \mapsto \negcross{spin}{spin},\quad
        \poscross{vec}{vec} \mapsto \poscross{vec}{vec},\quad
        \negcross{vec}{vec} \mapsto \negcross{vec}{vec},\quad
        \\
        \poscross{vec}{spin} \mapsto \poscross{spin}{vec},\quad
        \negcross{vec}{spin} \mapsto \negcross{spin}{vec},\quad
        \poscross{spin}{vec} \mapsto \poscross{vec}{spin},\quad
        \negcross{spin}{vec} \mapsto \negcross{vec}{spin},\quad
    \end{gather*}
\end{prop}

\begin{proof}
    It is a straightforward computation to show that $\Omega_\updownarrow$ preserves the defining relations of $\QSB(q,t,\kappa,d_\Sgo)$.  For example,
    \[
        \Omega_\updownarrow
        \left(
            \begin{tikzpicture}[anchorbase]
                \draw[vec] (0.2,-0.5) to [out=135,in=down] (-0.15,-0.2) to[out=up,in=-135] (0,0);
                \draw[wipe] (-0.2,-0.5) to[out=45,in=down] (0.15,-0.2);
                \draw[spin] (-0.2,-0.5) to[out=45,in=down] (0.15,-0.2) to[out=up,in=-45] (0,0) -- (0,0.2);
            \end{tikzpicture}
        \right)
        =
        \begin{tikzpicture}[anchorbase,yscale=-1]
            \draw[spin] (-0.2,-0.5) to[out=45,in=down] (0.15,-0.2) to[out=up,in=-45] (0,0);
            \draw[wipe] (0.2,-0.5) to [out=135,in=down] (-0.15,-0.2) to[out=up,in=-135] (0,0);
            \draw[vec] (0.2,-0.5) to [out=135,in=down] (-0.15,-0.2) to[out=up,in=-135] (0,0);
            \draw[spin] (0,0) -- (0,0.2);
        \end{tikzpicture}
        \overset{\cref{lobster1}}{=}
        \kappa \splitmor{spin}{spin}{vec}
        = \Omega_\updownarrow \left( \mergemor{spin}{vec}{spin} \right),
    \]
    showing that $\Omega_\updownarrow$ respects \cref{fishy}.  The verifications of the remaining relations in \cref{QSBdef} are analogous.
\end{proof}

\begin{prop} \label{barinv}
    There is an isomorphism of $\kk$-linear monoidal categories
    \[
        \QSB(q,t,\kappa,d_\Sgo) \xrightarrow{\cong} \QSB(q^{-1},t^{-1},\kappa^{-1},d_\Sgo),
    \]
     which we call the \emph{bar involution}, given by flipping all crossings in a diagram.  More precisely, the bar involution is given on objects by $\Sgo \mapsto \Sgo$, $\Vgo \mapsto \Vgo$, and on morphisms by
    \begin{gather*}
        \capmor{spin} \mapsto \capmor{spin}\, ,\quad
        \cupmor{spin} \mapsto \cupmor{spin}\, ,\quad
        \capmor{vec} \mapsto \capmor{vec}\, ,\quad
        \cupmor{vec} \mapsto \cupmor{vec}\, ,\quad
        \mergemor{vec}{spin}{spin} \mapsto \mergemor{vec}{spin}{spin}\, ,
        \\
        \poscross{spin}{spin} \mapsto \negcross{spin}{spin},\quad
        \negcross{spin}{spin} \mapsto \poscross{spin}{spin},\quad
        \poscross{vec}{vec} \mapsto \negcross{vec}{vec},\quad
        \negcross{vec}{vec} \mapsto \poscross{vec}{vec},
        \\
        \poscross{spin}{vec} \mapsto \negcross{spin}{vec},\quad
        \negcross{spin}{vec} \mapsto \poscross{spin}{vec},\quad
        \poscross{vec}{spin} \mapsto \negcross{vec}{spin},\quad
        \negcross{vec}{spin} \mapsto \poscross{vec}{spin}.
    \end{gather*}
\end{prop}

\begin{proof}
    We need to verify that the defining relations of \cref{QSBdef} hold with all crossings flipped and $q$, $t$, and $\kappa$ replaced by $q^{-1}$, $t^{-1}$, and $\kappa^{-1}$, respectively.  For most of the relations, these follow immediately from what we have already seen.  For example, for \cref{deloop} and \cref{fishy}, it follows from \cref{reloop} and \cref{lobster1}, respectively.  The most involved relation to check is \cref{oist}, where we need to verify that
    \begin{equation} \label{oister}
        \begin{tikzpicture}[centerzero]
            \draw[vec] (-0.4,-0.4) -- (0,0.1);
            \draw[vec] (0,-0.4) -- (0.2,-0.15);
            \draw[spin] (0.4,-0.4) --(0,0.1) -- (0,0.4);
        \end{tikzpicture}
        + q^{-1}
        \begin{tikzpicture}[centerzero]
            \draw[vec] (-0.4,-0.4) -- (0.2,-0.15);
            \draw[wipe] (0,-0.4) to[out=135,in=-135] (0,0.1);
            \draw[vec] (0,-0.4) to[out=135,in=-135] (0,0.1);
            \draw[spin] (0.4,-0.4) -- (0,0.1) -- (0,0.4);
        \end{tikzpicture}
        = \left( q^{-1} \kappa^{-2} + 1 \right)\
        \begin{tikzpicture}[centerzero]
            \draw[vec] (-0.4,-0.4) -- (-0.4,-0.3) arc(180:0:0.2) -- (0,-0.4);
            \draw[spin] (0.4,-0.4) \braidup (0,0.4);
        \end{tikzpicture}
        \ .
    \end{equation}
    To see this, we compose \cref{oist} on the bottom with $\negcross{vec}{vec}\ \idstrand{spin}\,$, multiply both sides by $q^{-1}$, and then use \cref{braid,reloop}.  We leave the verification of the remaining relations to the reader.
\end{proof}

It will sometimes be convenient to draw horizontal strands.  Since $\QSB$ is strict pivotal, the meaning of diagrams containing such strands is unambiguous.  For example,
\begin{equation} \label{barbell}
    \begin{tikzpicture}[centerzero]
        \draw[spin] (-0.2,-0.3) -- (-0.2,0.3);
        \draw[spin] (0.2,-0.3) -- (0.2,0.3);
        \draw[vec] (-0.2,0) -- (0.2,0);
    \end{tikzpicture}
    \ =\
    \begin{tikzpicture}[centerzero]
        \draw[spin] (-0.2,-0.3) -- (-0.2,0.3);
        \draw[spin] (0.2,-0.3) -- (0.2,0.3);
        \draw[vec] (-0.2,0) to[out=-45,in=-135,looseness=1.8] (0.2,0);
    \end{tikzpicture}
    \ =\
    \begin{tikzpicture}[centerzero]
        \draw[spin] (-0.2,-0.3) -- (-0.2,0.3);
        \draw[spin] (0.2,-0.3) -- (0.2,0.3);
        \draw[vec] (-0.2,0) to[out=45,in=135,looseness=1.8] (0.2,0);
    \end{tikzpicture}
    \ =\
    \begin{tikzpicture}[centerzero]
        \draw[spin] (-0.2,-0.3) -- (-0.2,0.3);
        \draw[spin] (0.2,-0.3) -- (0.2,0.3);
        \draw[vec] (-0.2,0.15) -- (0.2,-0.15);
    \end{tikzpicture}
    \ =\
    \begin{tikzpicture}[centerzero]
        \draw[spin] (-0.2,-0.3) -- (-0.2,0.3);
        \draw[spin] (0.2,-0.3) -- (0.2,0.3);
        \draw[vec] (-0.2,-0.15) -- (0.2,0.15);
    \end{tikzpicture}
    \ .
\end{equation}

\begin{lem}
    If $q \kappa^2 + 1$ is invertible, then
    \begin{equation} \label{bump}
        \begin{tikzpicture}[centerzero]
            \draw[spin] (0,-0.4) -- (0,0.4);
            \draw[vec] (0,-0.2) -- (-0.1,-0.2) arc(270:90:0.2) -- (0,0.2);
        \end{tikzpicture}
        = d_\Vgo \
        \begin{tikzpicture}[centerzero]
            \draw[spin] (0,-0.4) -- (0,0.4);
        \end{tikzpicture}
        \ .
    \end{equation}
\end{lem}

\begin{proof}
    We compute
    \[
        (q \kappa^2 +1)\
        \begin{tikzpicture}[centerzero]
            \draw[spin] (0,-0.4) -- (0,0.4);
            \draw[vec] (0,-0.2) -- (-0.1,-0.2) arc(270:90:0.2) -- (0,0.2);
        \end{tikzpicture}
        \overset{\cref{deloop}}{=}
        \begin{tikzpicture}[centerzero]
            \draw[spin] (0,-0.4) -- (0,0.4);
            \draw[vec] (0,-0.2) -- (-0.1,-0.2) arc(270:90:0.2) -- (0,0.2);
        \end{tikzpicture}
        + q\
        \begin{tikzpicture}[centerzero]
            \draw[vec] (-0.4,-0.2) to[out=right,in=left] (0,0.2);
            \draw[wipe] (0,-0.2) to[out=left,in=right] (-0.4,0.2);
            \draw[vec] (0,-0.2) to[out=left,in=right] (-0.4,0.2) arc(90:270:0.2);
            \draw[spin] (0,-0.4) -- (0,0.4);
        \end{tikzpicture}
        \overset{\cref{oist}}{=} \left( q \kappa^2 + 1 \right)\
        \begin{tikzpicture}[centerzero]
            \bub{vec}{-0.4,0};
            \draw[spin] (0,-0.4) -- (0,0.4);
        \end{tikzpicture}
        \overset{\cref{dimrel}}{=}
        \left( q \kappa^2 + 1 \right) d_\Vgo \
        \begin{tikzpicture}[centerzero]
            \draw[spin] (0,-0.4) -- (0,0.4);
        \end{tikzpicture}
        \ . \qedhere
    \]
\end{proof}

\begin{lem}
    If $q \kappa^2 + 1$ is invertible, then
    \begin{equation} \label{zombie}
        \begin{tikzpicture}[centerzero]
            \draw[vec] (-0.35,0.35) -- (0.35,-0.35);
            \draw[wipe] (-0.35,-0.35) -- (0.35,0.35);
            \draw[spin] (-0.35,-0.35) -- (0.35,0.35);
        \end{tikzpicture}
        = \frac{\kappa}{q \kappa^2 + 1}
        \left(\,
            \begin{tikzpicture}[centerzero]
                \draw[spin] (0.35,0.35) -- (0,0.15) -- (0,-0.15) -- (-0.35,-0.35);
                \draw[vec] (0,0.15) -- (-0.35,0.35);
                \draw[vec] (0,-0.15) -- (0.35,-0.35);
            \end{tikzpicture}
            + q\,
            \begin{tikzpicture}[centerzero]
                \draw[spin] (0.35,0.35) -- (0.15,0) -- (-0.15,0) -- (-0.35,-0.35);
                \draw[vec] (0.35,-0.35) -- (0.15,0);
                \draw[vec] (-0.35,0.35) -- (-0.15,0);
            \end{tikzpicture}
        \, \right)
        ,\qquad
        \begin{tikzpicture}[centerzero]
            \draw[spin] (-0.35,0.35) -- (0.35,-0.35);
            \draw[wipe] (-0.35,-0.35) -- (0.35,0.35);
            \draw[vec] (-0.35,-0.35) -- (0.35,0.35);
        \end{tikzpicture}
        = \frac{\kappa}{q \kappa^2 + 1}
        \left(\,
            \begin{tikzpicture}[centerzero]
                \draw[spin] (-0.35,0.35) -- (0,0.15) -- (0,-0.15) -- (0.35,-0.35);
                \draw[vec] (0,0.15) -- (0.35,0.35);
                \draw[vec] (0,-0.15) -- (-0.35,-0.35);
            \end{tikzpicture}
            + q\,
            \begin{tikzpicture}[centerzero]
                \draw[spin] (-0.35,0.35) -- (-0.15,0) -- (0.15,0) -- (0.35,-0.35);
                \draw[vec] (-0.35,-0.35) -- (-0.15,0);
                \draw[vec] (0.35,0.35) -- (0.15,0);
            \end{tikzpicture}
        \, \right)
        .
    \end{equation}
\end{lem}

\begin{proof}
    By \cref{oist}, we have
    \begin{equation} \label{ledge}
        (q \kappa^2 + 1)\
        \begin{tikzpicture}[centerzero]
            \draw[spin] (0,-0.4) -- (0,0.4);
            \draw[vec] (-0.3,-0.4) -- (-0.3,0.4);
        \end{tikzpicture}
        =
        \begin{tikzpicture}[centerzero]
            \draw[spin] (0.35,0.35) -- (0,0.15) -- (0,-0.15) -- (0.35,-0.35);
            \draw[vec] (0,0.15) -- (-0.35,0.35);
            \draw[vec] (0,-0.15) -- (-0.35,-0.35);
        \end{tikzpicture}
        + q\,
        \begin{tikzpicture}[centerzero]
            \draw[vec] (-0.3,-0.4) -- (0,0.25);
            \draw[wipe] (-0.3,0.4) -- (0,-0.25);
            \draw[vec] (-0.3,0.4) -- (0,-0.25);
            \draw[spin] (0,-0.4) -- (0,0.4);
        \end{tikzpicture}
        \ .
    \end{equation}
    Then the first equation in \cref{zombie} follows after composing on the bottom with $\poscross{spin}{vec}$ and using \cref{typhoon,braid,lobster1}.  Similarly, the second equation in \cref{zombie} follows after composing \cref{ledge} on the top with $\poscross{vec}{spin}$.
\end{proof}

As a corollary of \cref{zombie}, we have the skein relation
\begin{equation} \label{ash}
    \begin{tikzpicture}[centerzero]
        \draw[vec] (-0.35,0.35) -- (0.35,-0.35);
        \draw[wipe] (-0.35,-0.35) -- (0.35,0.35);
        \draw[spin] (-0.35,-0.35) -- (0.35,0.35);
    \end{tikzpicture}
    -
    \begin{tikzpicture}[centerzero]
        \draw[spin] (-0.35,-0.35) -- (0.35,0.35);
        \draw[wipe] (-0.35,0.35) -- (0.35,-0.35);
        \draw[vec] (-0.35,0.35) -- (0.35,-0.35);
    \end{tikzpicture}
    = \frac{\kappa(q-1)}{q \kappa^2 + 1}
    \left(\,
        \begin{tikzpicture}[centerzero]
            \draw[spin] (0.35,0.35) -- (0.15,0) -- (-0.15,0) -- (-0.35,-0.35);
            \draw[vec] (0.35,-0.35) -- (0.15,0);
            \draw[vec] (-0.35,0.35) -- (-0.15,0);
        \end{tikzpicture}
        -
        \begin{tikzpicture}[centerzero]
            \draw[spin] (0.35,0.35) -- (0,0.15) -- (0,-0.15) -- (-0.35,-0.35);
            \draw[vec] (0,0.15) -- (-0.35,0.35);
            \draw[vec] (0,-0.15) -- (0.35,-0.35);
        \end{tikzpicture}
    \, \right).
\end{equation}

\begin{lem}
    We have
    \begin{equation} \label{grapes}
        \begin{tikzpicture}[centerzero]
            \draw[spin] (-0.4,-0.4) -- (-0.4,0.4);
            \draw[spin] (0,-0.4) -- (0,0.4);
            \draw[spin] (0.4,-0.4) -- (0.4,0.4);
            \draw[vec] (-0.4,0.2) -- (0,0.2);
            \draw[vec] (-0.4,0) -- (0,0);
            \draw[vec] (0,-0.2) -- (0.4,-0.2);
        \end{tikzpicture}
        \ + (q+q\inv)\
        \begin{tikzpicture}[centerzero]
            \draw[spin] (-0.4,-0.4) -- (-0.4,0.4);
            \draw[spin] (0,-0.4) -- (0,0.4);
            \draw[spin] (0.4,-0.4) -- (0.4,0.4);
            \draw[vec] (-0.4,0.2) -- (0,0.2);
            \draw[vec] (-0.4,-0.2) -- (0,-0.2);
            \draw[vec] (0,0) -- (0.4,0);
        \end{tikzpicture}
        \ +\
        \begin{tikzpicture}[centerzero]
            \draw[spin] (-0.4,-0.4) -- (-0.4,0.4);
            \draw[spin] (0,-0.4) -- (0,0.4);
            \draw[spin] (0.4,-0.4) -- (0.4,0.4);
            \draw[vec] (-0.4,-0.2) -- (0,-0.2);
            \draw[vec] (-0.4,0) -- (0,0);
            \draw[vec] (0,0.2) -- (0.4,0.2);
        \end{tikzpicture}
        \ = (q\kappa^2+1)(q\inv \kappa^{-2}+1)\
        \begin{tikzpicture}[centerzero]
            \draw[spin] (-0.4,-0.4) -- (-0.4,0.4);
            \draw[spin] (0,-0.4) -- (0,0.4);
            \draw[spin] (0.4,-0.4) -- (0.4,0.4);
            \draw[vec] (0,0) -- (0.4,0);
        \end{tikzpicture}
        \ .
    \end{equation}
\end{lem}

\begin{proof}
    We have
    \begin{multline*}
        \begin{tikzpicture}[centerzero]
            \draw[spin] (-0.4,-0.4) -- (-0.4,0.4);
            \draw[spin] (0,-0.4) -- (0,0.4);
            \draw[spin] (0.4,-0.4) -- (0.4,0.4);
            \draw[vec] (-0.4,0.2) -- (0,0.2);
            \draw[vec] (-0.4,0) -- (0,0);
            \draw[vec] (0,-0.2) -- (0.4,-0.2);
        \end{tikzpicture}
        \ + (q+q^{-1})\
        \begin{tikzpicture}[centerzero]
            \draw[spin] (-0.4,-0.4) -- (-0.4,0.4);
            \draw[spin] (0,-0.4) -- (0,0.4);
            \draw[spin] (0.4,-0.4) -- (0.4,0.4);
            \draw[vec] (-0.4,0.2) -- (0,0.2);
            \draw[vec] (-0.4,-0.2) -- (0,-0.2);
            \draw[vec] (0,0) -- (0.4,0);
        \end{tikzpicture}
        \ +\
        \begin{tikzpicture}[centerzero]
            \draw[spin] (-0.4,-0.4) -- (-0.4,0.4);
            \draw[spin] (0,-0.4) -- (0,0.4);
            \draw[spin] (0.4,-0.4) -- (0.4,0.4);
            \draw[vec] (-0.4,-0.2) -- (0,-0.2);
            \draw[vec] (-0.4,0) -- (0,0);
            \draw[vec] (0,0.2) -- (0.4,0.2);
        \end{tikzpicture}
        \overset{\cref{zombie}}{=}
        (q^{-1} \kappa^{-1} + \kappa)
        \left(
            q\
            \begin{tikzpicture}[centerzero]
                \draw[spin] (-0.4,-0.4) -- (-0.4,0.4);
                \draw[spin] (0.4,-0.4) -- (0.4,0.4);
                \draw[vec] (-0.4,-0.15) -- (0.4,-0.15);
                \draw[wipe] (0,-0.4) -- (0,0.4);
                \draw[spin] (0,-0.4) -- (0,0.4);
                \draw[vec] (-0.4,0.15) -- (0,0.15);
            \end{tikzpicture}
            \ + \
            \begin{tikzpicture}[centerzero]
                \draw[spin] (-0.4,-0.4) -- (-0.4,0.4);
                \draw[spin] (0.4,-0.4) -- (0.4,0.4);
                \draw[vec] (-0.4,0.15) -- (0.4,0.15);
                \draw[wipe] (0,-0.4) -- (0,0.4);
                \draw[spin] (0,-0.4) -- (0,0.4);
                \draw[vec] (-0.4,-0.15) -- (0,-0.15);
            \end{tikzpicture}
        \right)
        \\
        \overset{\cref{typhoon}}{=}
        (q^{-1}\kappa^{-1}+\kappa)
        \left(
            q\
            \begin{tikzpicture}[centerzero]
                \draw[spin] (0.4,-0.4) -- (0.4,0.4);
                \draw[vec] (-0.4,0) -- (0.4,0);
                \draw[wipe] (0,-0.4) -- (0,0.4);
                \draw[spin] (0,-0.4) -- (0,0.4);
                \draw[wipe] (-0.4,0.2) to[out=right,in=left] (0,-0.2);
                \draw[vec] (-0.4,0.2) to[out=right,in=left] (0,-0.2);
                \draw[spin] (-0.4,-0.4) -- (-0.4,0.4);
            \end{tikzpicture}
            \ + \
            \begin{tikzpicture}[centerzero]
                \draw[spin] (-0.4,-0.4) -- (-0.4,0.4);
                \draw[spin] (0.4,-0.4) -- (0.4,0.4);
                \draw[vec] (-0.4,0.15) -- (0.4,0.15);
                \draw[wipe] (0,-0.4) -- (0,0.4);
                \draw[spin] (0,-0.4) -- (0,0.4);
                \draw[vec] (-0.4,-0.15) -- (0,-0.15);
            \end{tikzpicture}
        \right)
        \overset{\cref{oist}}{=}
        (q \kappa^2 + 1) (q^{-1}\kappa^{-1}+\kappa)\
        \begin{tikzpicture}[centerzero]
            \draw[spin] (-0.4,-0.4) -- (-0.4,0.4);
            \draw[spin] (0,-0.4) -- (0,0.4);
            \draw[spin] (0.4,-0.4) -- (0.4,0.4);
            \draw[vec] (0,-0.15) to[out=left,in=left,looseness=2] (0,0.15) -- (0.4,0.15);
            \draw[wipe] (0,0) -- (0,0.4);
            \draw[spin] (0,-0.4) -- (0,0.4);
        \end{tikzpicture}
        \\
        \overset{\cref{lobster1}}{=} (q\kappa^2+1)(q\inv \kappa^{-2}+1)\
        \begin{tikzpicture}[centerzero]
            \draw[spin] (-0.4,-0.4) -- (-0.4,0.4);
            \draw[spin] (0,-0.4) -- (0,0.4);
            \draw[spin] (0.4,-0.4) -- (0.4,0.4);
            \draw[vec] (0,0) -- (0.4,0);
        \end{tikzpicture}
        \ ,
    \end{multline*}
    where, in the second equality, we used the bar involution applied to \cref{zombie}.
\end{proof}

\begin{rem} \label{iquantum}
    The image of
    \(
        \begin{tikzpicture}[centerzero]
            \draw[spin] (-0.2,-0.3) -- (-0.2,0.3);
            \draw[spin] (0.2,-0.3) -- (0.2,0.3);
            \draw[vec] (-0.2,0) -- (0.2,0);
        \end{tikzpicture}
    \)
    under the incarnation functor to be defined in \cref{sec:incarnation} is given in \cref{rhino}.  The image of the relation \cref{grapes} under the incarnation functor corresponds to \cite[Prop.~4.2]{Wen20}.  As noted in \cite{Wen20}, this implies that one has an action on $S^{\otimes r}$ of the coideal subalgebra $U_q'(\fso_r)$ of $U_q(\mathfrak{gl}_r)$ introduced in \cite{GK91} and further studied in \cite{NS95,Let97}.
\end{rem}

\section{The quantum antisymmetrizer\label{sec:antisymmetrizer}}

In the classical and quantum representation theory of orthogonal groups, exterior powers of the natural module play a fundamental role, and their categorical avatars are given by antisymmetrizer idempotents.  In this section, we construct the corresponding quantum antisymmetrizers inside the quantum spin Brauer category.  Their diagrammatic properties provide key reduction tools used later in the paper: in \cref{sec:full}, they enter the analysis of the images of diagrams under the incarnation functor, while in \cref{sec:essential} their vanishing and trace identities are used to reduce closed diagrams and prove finiteness and interpolation results. They also indicate how $\QSB$ should be related to web-type categories, where one typically includes objects corresponding to the exterior powers of the natural module.

Throughout this section, we assume that $\kk$ is a field, and we fix $q,t,\kappa \in \kk^\times$ such that $q$ is not a root of unity.  We also assume that
\begin{equation} \label{lime}
    q^{2r-1} \kappa^2 + 1 \ne 0 \qquad \text{for all } r \in \N,
\end{equation}
since many of the constructions below (e.g., \cref{asymdef}) require division by these elements.  It will be convenient to introduce a shorthand for multiple strands:
\begin{equation} \label{multistrand}

        \ ,\quad r \ge 0
    \end{equation}
    (where we label the strands by $r$ when we wish to emphasize how many there are), recursively by
    \begin{gather} \label{asymbase}
        %
\
        \right),\quad r \ge 1.
    \end{gather}
\end{defin}

Since \cref{asymbase,asymdef} are invariant under the functor $\Omega_\updownarrow$ of \cref{horizon}, we have
\begin{equation} \label{eclipse}
    \Omega_\updownarrow
    \left(
        %
    \ .
\end{equation}
In order to keep diagrams as uncluttered as possible, we will often omit the label on a strand when this label is uniquely determined by the rest of the diagram.  For instance, in the diagram
\[
    %
    \, ,\qquad \text{where } 0 \le s \le r,
\]
we have omitted the label on the top-right strand, since this label is forced to be $r-s$.

\begin{prop} \label{sup}
    We have
    \begin{gather} \label{absorbcr}
        %
        \ ,\qquad u \ge 0,\ 0 \le s \le r-u.
    \end{gather}
\end{prop}

\begin{proof}
    First note that the second equality in \cref{absorbcr} follows from the first equality in \cref{absorbcr} after applying $\Omega_\updownarrow$ and using \cref{eclipse}.  Similarly, each of the second equalities in \cref{absorbcup,asymidem} follow from the first ones.  In addition, if \cref{absorbcr} holds, then we have
    \[
        %
    \]
    By our assumption \cref{lime}, $1 + q^{-1} \kappa^{-2}$ is invertible, and so the first equality in \cref{absorbcup} follows.  Thus, it suffices to prove the first equalities in \cref{absorbcr,asymidem}.

    For $m \ge 0$, let $P(m)$ be the statement that the first equalities in \cref{absorbcr,asymidem} hold for all $r \le m$ and all $s,u$ satisfying the given inequalities.  Using \cref{braid}, the equalities \cref{absorbcr} are equivalent to
    \begin{equation} \label{absorbneg}
        %
        \ .
    \end{equation}
    For $r \in \{0,1\}$, \cref{absorbcr} is vacuous, and \cref{asymidem} is trivially satisfied.  For $r=2$, \cref{absorbcr,asymidem} are easily verified using the explicit form
    \begin{equation} \label{ski}
        %
            \,
        \right).
    \end{equation}
    Thus, $P(2)$ is true.

    Now suppose that $m \ge 2$, and that $P(m)$ is true.  We have
    \begin{multline} \label{fool1}
        %
        \ .
    \end{multline}
    We also have
    \begin{equation} \label{fries}
        %
\
        \right),
    \end{equation}
    which implies, using $P(m)$, that
    \begin{equation} \label{fool2}
        %
        \ .
    \end{equation}

    We now prove the first equality in \cref{absorbcr} for $r=m+1$.  When $s < m-1$, this equality follows easily from \cref{asymdef} and the induction hypothesis.  To prove that it holds for $s=m-1$, we first use the $r=m-1$ case of \cref{asymdef} on the top antisymmetrizer in the middle term on the right-hand side of the $r=m$ case of \cref{asymdef} to see that
    \begin{multline} \label{bone}
        %
\
        \right).
    \end{multline}
    Thus, using the inductive hypothesis $P(m)$, we have
    \begin{align*}
        %
        \ .
    \end{multline}
    On the other hand, we have
    \begin{align*}
        [m]\
        %
        \ .
    \end{align*}
    Composing on the bottom with
    \[
        %
    \]
    and using \cref{asymidem} for $r=m$ then gives that the right-hand side of \cref{deity} is zero.  This completes the proof of the first equality in \cref{absorbcr} for $r=m+1$.

    Finally, we show that the first equality in \cref{asymidem} holds for $r=m+1$.  Since the relation is trivial when $s \le 1$, we assume that $s \ge 2$.  Then, using the inductive hypothesis $P(m)$, we have
    \begin{multline*}
        %
        \ ,
    \end{multline*}
    as desired.
\end{proof}

\begin{rem}
    It follows from \cref{asymidem} that the quantum antisymmetrizer is idempotent.  The analogous idempotent in the Birman--Murakami--Wenzl (BMW) algebra is well known.  See, for example, \cite[\S3]{DHS13} and \cite[\S4]{HS99}.  The recursion \cref{asymdef} resembles \cite[(7.12)]{TW05}, except that the expression there seems to have an error: the numerator of $[m]_q$ there should be $q^m-q^{-m}$.  \Cref{sup} could likely be deduced from existing results about the quantum symmetrizer in the BMW algebra.  However, we have opted for a self-contained direct proof.  Note that our use of the term \emph{quantum antisymmetrizer} does not agree with the use of this term in \cite{DF94}.
\end{rem}

\begin{lem}
    We have
    \begin{equation} \label{asymbend}
        \begin{tikzpicture}[centerzero]
            \draw[vec] (-0.5,0.6) -- (-0.5,-0.15) to[out=down,in=down,looseness=1.5] (0,-0.15) -- (0,0.15) to[out=up,in=up,looseness=1.5] (0.5,0.15) -- (0.5,-0.6);
            \altbox{-0.3,-0.15}{0.3,0.15}{r};
        \end{tikzpicture}
        =
        \begin{tikzpicture}[centerzero]
            \draw[vec] (0,-0.6) -- (0,0.6);
            \altbox{-0.3,-0.15}{0.3,0.15}{r};
        \end{tikzpicture}
        =
        \begin{tikzpicture}[centerzero,xscale=-1]
            \draw[vec] (-0.5,0.6) -- (-0.5,-0.15) to[out=down,in=down,looseness=1.5] (0,-0.15) -- (0,0.15) to[out=up,in=up,looseness=1.5] (0.5,0.15) -- (0.5,-0.6);
            \altbox{-0.3,-0.15}{0.3,0.15}{r};
        \end{tikzpicture}
        \ ,\qquad r \ge 0.
    \end{equation}
\end{lem}

\begin{proof}
    We prove the result by induction on $r$.  The result clearly holds for $r \in \{0,1\}$.  For the induction step, it suffices to show that the quantum antisymmetrizer satisfies the rotated version of the recursion \cref{asymdef}:
    \[
        \begin{tikzpicture}[centerzero]
            \draw[vec] (0,-0.6) -- (0,0.6);
            \altbox{-0.4,-0.15}{0.4,0.15}{r+1};
        \end{tikzpicture}
        = \frac{1}{[r+1]}
        \left(
            q^r\
            \begin{tikzpicture}[centerzero]
                \draw[vec] (0,-0.6) -- (0,0.6);
                \altbox{-0.3,-0.15}{0.3,0.15}{r};
                \draw[vec] (-0.5,-0.6) -- (-0.5,0.6);
            \end{tikzpicture}
            - [r]\
            \begin{tikzpicture}[centerzero]
                \draw[vec] (0.15,-0.2) -- (0.15,0.2) node[midway,anchor=west] {\strandlabel{r-1}};
                \draw[vec] (-0.15,-0.2) \braidup (-0.5,0.2) -- (-0.5,0.7);
                \draw[wipe] (-0.15,0.2) \braiddown (-0.5,-0.2);
                \draw[vec] (-0.15,0.2) \braiddown (-0.5,-0.2) -- (-0.5,-0.7);
                \draw[vec] (0,0.5) -- (0,0.7);
                \draw[vec] (0,-0.5) -- (0,-0.7);
                \altbox{0.3,-0.5}{-0.3,-0.2}{r};
                \altbox{0.3,0.2}{-0.3,0.5}{r};
            \end{tikzpicture}
            - \frac{q^r - q^{-r}}{1+q^{1-2r}\kappa^{-2}}\
            \begin{tikzpicture}[centerzero]
                \draw[vec] (0.15,-0.2) -- (0.15,0.2) node[midway,anchor=west] {\strandlabel{r-1}};
                \draw[vec] (-0.15,0.2) to[out=down,in=down,looseness=1.5] (-0.5,0.2) -- (-0.5,0.7);
                \draw[vec] (-0.15,-0.2) to[out=up,in=up,looseness=1.5] (-0.5,-0.2) -- (-0.5,-0.7);
                \draw[vec] (0,0.5) -- (0,0.7);
                \draw[vec] (0,-0.5) -- (0,-0.7);
                \altbox{0.3,-0.5}{-0.3,-0.2}{r};
                \altbox{0.3,0.2}{-0.3,0.5}{r};
            \end{tikzpicture}\
        \right),\quad r \ge 1.
    \]
    This follows from composing both sides of \cref{asymdef} on top and bottom by
    \[
        \begin{tikzpicture}[centerzero]
            \draw[vec] (-0.2,-0.4) -- (0.2,0.4);
            \draw[wipe] (0.2,-0.4) -- (-0.2,0.4);
            \draw[vec] (0.2,-0.4) -- (-0.2,0.4);
            \draw[vec] (0.1,0.2) node[anchor=west] {\strandlabel{r}};
        \end{tikzpicture}
        \qquad \text{and} \qquad
        \begin{tikzpicture}[centerzero]
            \draw[vec] (0.2,-0.4) -- (-0.2,0.4);
            \draw[wipe] (-0.2,-0.4) -- (0.2,0.4);
            \draw[vec] (-0.2,-0.4) -- (0.2,0.4);
            \draw[vec] (-0.1,0.2) node[anchor=east] {\strandlabel{r}};
        \end{tikzpicture}
        \ ,
    \]
    respectively, then using \cref{braid,deloop,reloop,absorbcr}.
\end{proof}

\begin{lem}
    We have
    \begin{equation} \label{monkey}
        \begin{tikzpicture}[anchorbase]
            \draw[vec] (0,-0.15) arc(180:360:0.2) -- (0.4,0.15) arc(0:180:0.2);
            \altbox{-0.2,-0.15}{0.2,0.15}{r};
        \end{tikzpicture}
        =
        \frac{\kappa^2 q^r  + q^{1-r}}{\kappa^2 + q}
        \prod_{s=1}^r \frac{\kappa^{-2} q^{2-s} - \kappa^2 q^{s-2}}{q^s - q^{-s}} 1_\one
        ,\qquad r \in \N,
    \end{equation}
    where we interpret the right-hand side as $1_\one$ when $r=0$.
\end{lem}

\begin{proof}
    This is a straightforward induction using \cref{fool1}.
\end{proof}

\begin{rem}
    When $\kappa^2 = q^{1-N}$, the right-hand side of \cref{monkey} becomes
    \begin{equation} \label{monkey2}
        \frac{q^{r-N}  + q^{-r}}{1 + q^{-N}} \frac{[N]!}{[r]![N-r]!} 1_\one
        = \left( \qbinom{N-1}{r} + \qbinom{N-1}{r-1} \right) 1_\one,
    \end{equation}
    where $[a]!$ and $\qbinom{a}{b}$ are quantum factorials and quantum binomial coefficients, respectively.  (See \cref{subsec:Uq}.)  This specializes, at $q=1$, to $\binom{N}{r} 1_\one$.
\end{rem}

\begin{prop} \label{Delprop}
    We have
    \begin{equation} \label{Deligne}
        \begin{tikzpicture}[anchorbase]
            \draw[spin] (-0.1,0) -- (0.1,0) arc(-90:90:0.15) -- (-0.1,0.3) arc(90:270:0.15);
            \draw[vec] (0,-1) -- (0,0);
            \altbox{-0.2,-0.65}{0.2,-0.35}{r};
        \end{tikzpicture}
        = 0
        \qquad \text{for all } r > 0 \text{ such that } \kappa^2 \ne (-q)^{1-r}.
    \end{equation}
    (Note that the condition $\kappa^2 \ne (-q)^{1-r}$ follows from assumption \cref{lime} when $r$ is even.)
\end{prop}

\begin{proof}
    We have
    \begin{align*}
        d_\Vgo\
        \begin{tikzpicture}[anchorbase]
            \draw[vec] (0,-1.2) -- (0,-0.8) node[midway,anchor=east] {\strandlabel{r}};
            \draw[vec] (0,-0.5) -- (0,0) node[midway,anchor=east] {\strandlabel{r}};
            \draw[spin] (-0.4,0) -- (0.4,0) arc(-90:90:0.15) -- (-0.4,0.3) arc(90:270:0.15);
            \altbox{-0.5,-0.8}{0.5,-0.5}{r};
        \end{tikzpicture}
        \ &\overset{\mathclap{\cref{bump}}}{=}\
        \begin{tikzpicture}[anchorbase]
            \draw[vec] (0,-1.2) -- (0,-0.8) node[midway,anchor=east] {\strandlabel{r}};
            \draw[vec] (0,-0.5) -- (0,0) node[midway,anchor=east] {\strandlabel{r}};
            \draw[vec] (0.2,0) -- (0.2,0.3);
            \draw[spin] (-0.4,0) -- (0.4,0) arc(-90:90:0.15) -- (-0.4,0.3) arc(90:270:0.15);
            \altbox{-0.5,-0.8}{0.5,-0.5}{r};
        \end{tikzpicture}
        \overset{\cref{zombie}}{=} - q
        \begin{tikzpicture}[anchorbase]
            \draw[vec] (0,-1.2) -- (0,-0.8) node[midway,anchor=east] {\strandlabel{r}};
            \draw[vec] (-0.2,-0.5) -- (-0.2,0) node[midway,anchor=east] {\strandlabel{r-1}};
            \draw[vec] (0.2,-0.5) -- (0.2,0);
            \draw[spin] (-0.4,0) -- (0.4,0) arc(-90:90:0.15) -- (-0.4,0.3) arc(90:270:0.15);
            \altbox{-0.5,-0.8}{0.5,-0.5}{r};
            \draw[vec] (0,0) -- (0,0.3);
        \end{tikzpicture}
        + \frac{q \kappa^2 + 1}{\kappa}
        \begin{tikzpicture}[anchorbase]
            \draw[vec] (0,-1.2) -- (0,-0.8) node[midway,anchor=east] {\strandlabel{r}};
            \draw[vec] (-0.2,-0.5) -- (-0.2,0) node[midway,anchor=east] {\strandlabel{r-1}};
            \draw[vec] (0.2,-0.5) -- (0.2,0.3);
            \draw[wipe] (0,0) -- (0.4,0);
            \draw[spin] (-0.4,0) -- (0.4,0) arc(-90:90:0.15) -- (-0.4,0.3) arc(90:270:0.15);
            \altbox{-0.5,-0.8}{0.5,-0.5}{r};
        \end{tikzpicture}
        \\
        &\overset{\mathclap{\cref{zombie}}}{=}\ q^2
        \begin{tikzpicture}[anchorbase]
            \draw[vec] (0,-1.2) -- (0,-0.8) node[midway,anchor=east] {\strandlabel{r}};
            \draw[vec] (-0.2,-0.5) -- (-0.2,0) node[midway,anchor=east] {\strandlabel{r-2}};
            \draw[vec] (0.2,-0.5) -- (0.2,0) node[midway,anchor=west] {\strandlabel{2}};
            \draw[spin] (-0.4,0) -- (0.4,0) arc(-90:90:0.15) -- (-0.4,0.3) arc(90:270:0.15);
            \altbox{-0.5,-0.8}{0.5,-0.5}{r};
            \draw[vec] (0,0) -- (0,0.3);
        \end{tikzpicture}
        - q\frac{q \kappa^2 + 1}{\kappa}
        \begin{tikzpicture}[anchorbase]
            \draw[vec] (0,-1.2) -- (0,-0.8) node[midway,anchor=east] {\strandlabel{r}};
            \draw[vec] (-0.2,-0.5) -- (-0.2,0) node[midway,anchor=east] {\strandlabel{r-2}};
            \draw[vec] (0,-0.5) -- (0,0.3);
            \draw[vec] (0.2,-0.5) -- (0.2,0);
            \draw[wipe] (-0.1,0) -- (0.1,0);
            \draw[spin] (-0.4,0) -- (0.4,0) arc(-90:90:0.15) -- (-0.4,0.3) arc(90:270:0.15);
            \altbox{-0.5,-0.8}{0.5,-0.5}{r};
        \end{tikzpicture}
        + \frac{q \kappa^2 + 1}{\kappa}
        \begin{tikzpicture}[anchorbase]
            \draw[vec] (0,-1.2) -- (0,-0.8) node[midway,anchor=east] {\strandlabel{r}};
            \draw[vec] (-0.2,-0.5) -- (-0.2,0) node[midway,anchor=east] {\strandlabel{r-1}};
            \draw[vec] (0.2,-0.5) -- (0.2,0.3);
            \draw[wipe] (0,0) -- (0.4,0);
            \draw[spin] (-0.4,0) -- (0.4,0) arc(-90:90:0.15) -- (-0.4,0.3) arc(90:270:0.15);
            \altbox{-0.5,-0.8}{0.5,-0.5}{r};
        \end{tikzpicture}
        \\
        &= \dotsb = (-q)^r\
        \begin{tikzpicture}[anchorbase]
            \draw[vec] (0,-1.2) -- (0,-0.8) node[midway,anchor=east] {\strandlabel{r}};
            \draw[vec] (0,-0.5) -- (0,0) node[midway,anchor=west] {\strandlabel{r}};
            \draw[spin] (-0.4,0) -- (0.4,0) arc(-90:90:0.15) -- (-0.4,0.3) arc(90:270:0.15);
            \altbox{-0.5,-0.8}{0.5,-0.5}{r};
            \draw[vec] (-0.2,0) -- (-0.2,0.3);
        \end{tikzpicture}
        + \frac{q \kappa^2 + 1}{\kappa} \sum_{s=0}^{r-1} (-q)^s\
        \begin{tikzpicture}[anchorbase]
            \draw[vec] (0,-1.2) -- (0,-0.8) node[midway,anchor=east] {\strandlabel{r}};
            \draw[vec] (-0.2,-0.5) -- (-0.2,0) node[midway,anchor=east] {\strandlabel{r-s-1}};
            \draw[vec] (0,-0.5) -- (0,0.3);
            \draw[vec] (0.2,-0.5) -- (0.2,0) node[midway,anchor=west] {\strandlabel{s}};
            \draw[wipe] (-0.1,0) -- (0.1,0);
            \draw[spin] (-0.4,0) -- (0.4,0) arc(-90:90:0.15) -- (-0.4,0.3) arc(90:270:0.15);
            \altbox{-0.5,-0.8}{0.5,-0.5}{r};
        \end{tikzpicture}
        \\
        &\overset{\mathclap{\cref{absorbneg}}}{\underset{\mathclap{\cref{typhoon}}}{=}}\ (-q)^r
        \begin{tikzpicture}[anchorbase]
            \draw[vec] (0,-1.2) -- (0,-0.8) node[midway,anchor=east] {\strandlabel{r}};
            \draw[vec] (0,-0.5) -- (0,0) node[midway,anchor=west] {\strandlabel{r}};
            \draw[spin] (-0.4,0) -- (0.4,0) arc(-90:90:0.15) -- (-0.4,0.3) arc(90:270:0.15);
            \altbox{-0.5,-0.8}{0.5,-0.5}{r};
            \draw[vec] (-0.2,0) -- (-0.2,0.3);
        \end{tikzpicture}
        + \frac{q \kappa^2 + 1}{\kappa} \sum_{s=0}^{r-1} q^{2s}\
        \begin{tikzpicture}[anchorbase]
            \draw[vec] (0,-1.2) -- (0,-0.8) node[midway,anchor=east] {\strandlabel{r}};
            \draw[vec] (-0.2,-0.5) -- (-0.2,0) node[midway,anchor=east] {\strandlabel{r-1}};
            \draw[vec] (0.2,-0.5) -- (0.2,0.3);
            \draw[wipe] (0,0) -- (0.4,0);
            \draw[spin] (-0.4,0) -- (0.4,0) arc(-90:90:0.15) -- (-0.4,0.3) arc(90:270:0.15);
            \altbox{-0.5,-0.8}{0.5,-0.5}{r};
        \end{tikzpicture}
        \\
        &\overset{\mathclap{\cref{bump}}}{\underset{\mathclap{\cref{lobster1}}}{=}}\
        \Big( (-q)^r d_\Vgo + q^{r-1}(q \kappa^2+1)[r] \Big)\
        \begin{tikzpicture}[anchorbase]
            \draw[vec] (0,-1.2) -- (0,-0.8) node[midway,anchor=east] {\strandlabel{r}};
            \draw[vec] (0,-0.5) -- (0,0) node[midway,anchor=east] {\strandlabel{r}};
            \draw[spin] (-0.4,0) -- (0.4,0) arc(-90:90:0.15) -- (-0.4,0.3) arc(90:270:0.15);
            \altbox{-0.5,-0.8}{0.5,-0.5}{r};
        \end{tikzpicture}
        \ .
    \end{align*}
    Thus
    \[
        \Big( \big( (-q)^r-1 \big) d_\Vgo + (q^r \kappa^2 + q^{r-1}) [r] \Big)
        \begin{tikzpicture}[anchorbase]
            \draw[vec] (0,-1.2) -- (0,-0.8) node[midway,anchor=east] {\strandlabel{r}};
            \draw[vec] (0,-0.5) -- (0,0) node[midway,anchor=east] {\strandlabel{r}};
            \draw[spin] (-0.4,0) -- (0.4,0) arc(-90:90:0.15) -- (-0.4,0.3) arc(90:270:0.15);
            \altbox{-0.5,-0.8}{0.5,-0.5}{r};
        \end{tikzpicture}
        \ = 0.
    \]
    Note that
    \[
        \big( (-q)^r-1 \big) d_\Vgo + (q^r \kappa^2 + q^{r-1}) [r]
        = \frac{\big( q\kappa^2+1 \big) \big( \kappa^{-2} - (-q)^{r-1} \big) \big( 1 - (-q)^r \big)}{q-q^{-1}}.
    \]
    Our assumption that $q$ is not a root of unity implies that the factor $1-(-q)^r$ is nonzero, while the assumption \cref{lime} implies that the factor $q \kappa^2+1$ is nonzero.  Since we assume $\kappa^2 \ne (-q)^{1-r}$, the result follows.
\end{proof}

Next, we prove two technical lemmas needed in the proof of \cref{royal}.

\begin{lem} \label{pork}
    For $1 \le s \le r$, we have
    \begin{equation}
        \begin{tikzpicture}[anchorbase]
            \draw[spin] (0,0) -- (0.9,0) arc(-90:90:0.15) -- (0,0.3) arc(90:270:0.15);
            \draw[vec] (0.1,-0.3) -- (0.1,0);
            \draw[vec] (0.2,-1) to[out=up,in=down] (0,-0.6);
            \draw[vec] (0.1,-0.8) node[anchor=east] {\strandlabel{s-1}};
            \draw[vec] (0.2,-0.6) to[out=down,in=down,looseness=1.5] (0.5,-0.6) -- (0.5,0);
            \draw[vec] (0.6,-1) to[out=up,in=down] (0.8,0);
            \altbox{-0.1,-0.6}{0.3,-0.3}{s};
            \altbox{0,-1.3}{0.8,-1}{r};
            \draw[vec] (0.4,-1.3) to[out=down,in=down,looseness=1.5] (1.2,-1.3) -- (1.2,0.3) to[out=up,in=up,looseness=1.5] (0.45,0.3);
        \end{tikzpicture}
        \, =
        \frac{(q\kappa^2 + 1)(q^{1-s} \kappa^{-2} - q^{s-2})}{q^s - q^{-s}}\
        \begin{tikzpicture}[anchorbase]
            \draw[spin] (-0.1,0) -- (0.1,0) arc(-90:90:0.15) -- (-0.1,0.3) arc(90:270:0.15);
            \draw[vec] (0,0) -- (0,-0.65) arc(180:360:0.2) -- (0.4,0.3) arc(0:180:0.2);
            \altbox{-0.2,-0.65}{0.2,-0.35}{r};
        \end{tikzpicture}
        + \frac{[s-1]}{[s]} \frac{q^{2s-4} \kappa^2 + q}{q^{2s-3} \kappa^2 + 1}
        \begin{tikzpicture}[anchorbase]
            \draw[spin] (0,0) -- (0.9,0) arc(-90:90:0.15) -- (0,0.3) arc(90:270:0.15);
            \draw[vec] (0.1,-0.3) -- (0.1,0);
            \draw[vec] (0.2,-1) to[out=up,in=down] (0,-0.6);
            \draw[vec] (0.1,-0.8) node[anchor=east] {\strandlabel{s-2}};
            \draw[vec] (0.2,-0.6) to[out=down,in=down,looseness=1.5] (0.5,-0.6) -- (0.5,0);
            \draw[vec] (0.6,-1) to[out=up,in=down] (0.8,0);
            \altbox{-0.3,-0.6}{0.3,-0.3}{s-1};
            \altbox{0,-1.3}{0.8,-1}{r};
            \draw[vec] (0.4,-1.3) to[out=down,in=down,looseness=1.5] (1.2,-1.3) -- (1.2,0.3) to[out=up,in=up,looseness=1.5] (0.45,0.3);
        \end{tikzpicture}
        \ .
    \end{equation}
\end{lem}

\begin{proof}
    We compute
    \begin{multline*}
        \begin{tikzpicture}[anchorbase]
            \draw[spin] (0,0) -- (0.9,0) arc(-90:90:0.15) -- (0,0.3) arc(90:270:0.15);
            \draw[vec] (0.1,-0.3) -- (0.1,0);
            \draw[vec] (0.2,-1) to[out=up,in=down] (0,-0.6);
            \draw[vec] (0.1,-0.8) node[anchor=east] {\strandlabel{s-1}};
            \draw[vec] (0.2,-0.6) to[out=down,in=down,looseness=1.5] (0.5,-0.6) -- (0.5,0);
            \draw[vec] (0.6,-1) to[out=up,in=down] (0.8,0);
            \altbox{-0.1,-0.6}{0.3,-0.3}{s};
            \altbox{0,-1.3}{0.8,-1}{r};
            \draw[vec] (0.4,-1.3) to[out=down,in=down,looseness=1.5] (1.2,-1.3) -- (1.2,0.3) to[out=up,in=up,looseness=1.5] (0.45,0.3);
        \end{tikzpicture}
        \overset{\cref{asymdef}}{=}
        \frac{1}{[s]}
        \left(
            q^{s-1}
            \begin{tikzpicture}[anchorbase]
                \draw[spin] (0,0) -- (0.9,0) arc(-90:90:0.15) -- (0,0.3) arc(90:270:0.15);
                \draw[vec] (-0.2,-0.3) to[out=up,in=down] (0,0);
                \draw[vec] (0.2,-1) to[out=up,in=down] (-0.2,-0.6);
                \draw[vec] (0.2,0) -- (0.2,-0.6) to[out=down,in=down,looseness=1.5] (0.5,-0.6) -- (0.5,0);
                \draw[vec] (0.6,-1) to[out=up,in=down] (0.8,0);
                \altbox{-0.5,-0.6}{0.1,-0.3}{s-1};
                \altbox{0,-1.3}{0.8,-1}{r};
                \draw[vec] (0.4,-1.3) to[out=down,in=down,looseness=1.5] (1.2,-1.3) -- (1.2,0.3) to[out=up,in=up,looseness=1.5] (0.45,0.3);
            \end{tikzpicture}
            \, -
            [s-1]
            \begin{tikzpicture}[anchorbase]
                \draw[spin] (0,0) -- (0.9,0) arc(-90:90:0.15) -- (0,0.3) arc(90:270:0.15);
                \draw[vec] (-0.2,-0.3) to[out=up,in=down] (0,0);
                \draw[vec] (-0.1,-0.6) to[out=down,in=down,looseness=1.5] (0.55,-0.6) -- (0.55,0);
                \draw[vec] (0.6,-1.8) \braidup (0.8,0);
                \draw[vec] (-0.2,-1.4) \braiddown (0.2,-1.8);
                \draw[vec] (-0.3,-0.6) -- (-0.3,-1.1) node[midway,anchor=east] {\strandlabel{s-2}};
                \draw[wipe] (-0.1,-1.1) \braidup (0.3,-0.6);
                \draw[vec] (-0.1,-1.1) \braidup (0.3,-0.6) -- (0.3,0);
                \altbox{-0.5,-0.6}{0.1,-0.3}{s-1};
                \altbox{-0.5,-1.1}{0.1,-1.4}{s-1};
                \altbox{0,-2.1}{0.8,-1.8}{r};
                \draw[vec] (0.4,-2.1) to[out=down,in=down,looseness=1.5] (1.2,-2.1) -- (1.2,0.3) to[out=up,in=up,looseness=1.5] (0.45,0.3);
            \end{tikzpicture}
            - \frac{q^{s-1}-q^{1-s}}{1+q^{3-2s} \kappa^{-2}}
            \begin{tikzpicture}[anchorbase]
                \draw[spin] (0,0) -- (0.9,0) arc(-90:90:0.15) -- (0,0.3) arc(90:270:0.15);
                \draw[vec] (-0.2,-0.3) to[out=up,in=down] (0,0);
                \draw[vec] (-0.1,-0.6) to[out=down,in=down,looseness=1.5] (0.3,-0.6) -- (0.3,0);
                \draw[vec] (0.6,-1.8) \braidup (0.8,0);
                \draw[vec] (-0.2,-1.4) \braiddown (0.2,-1.8);
                \draw[vec] (-0.3,-0.6) -- (-0.3,-1.1) node[midway,anchor=east] {\strandlabel{s-2}};
                \draw[vec] (-0.1,-1.1) \braidup (0.55,-0.6) -- (0.55,0);
                \altbox{-0.5,-0.6}{0.1,-0.3}{s-1};
                \altbox{-0.5,-1.1}{0.1,-1.4}{s-1};
                \altbox{0,-2.1}{0.8,-1.8}{r};
                \draw[vec] (0.4,-2.1) to[out=down,in=down,looseness=1.5] (1.2,-2.1) -- (1.2,0.3) to[out=up,in=up,looseness=1.5] (0.45,0.3);
            \end{tikzpicture}
        \right)
        \\
        \overset{\cref{asymidem}}{\underset{\cref{bump}}{=}}
        \frac{1}{[s]}
        \left(
            q^{s-1} d_\Vgo\,
            \begin{tikzpicture}[anchorbase]
                \draw[spin] (-0.1,0) -- (0.1,0) arc(-90:90:0.15) -- (-0.1,0.3) arc(90:270:0.15);
                \draw[vec] (0,0) -- (0,-0.65) arc(180:360:0.2) -- (0.4,0.3) arc(0:180:0.2);
                \altbox{-0.2,-0.65}{0.2,-0.35}{r};
            \end{tikzpicture}
            \, -
            [s-1]\,
            \begin{tikzpicture}[anchorbase]
                \draw[vec] (-0.2,-0.3) to[out=up,in=down] (0,0);
                \draw[vec] (0.2,-1.2) \braidup (-0.3,-0.6);
                \draw[vec] (-0.1,-0.95) node[anchor=east] {\strandlabel{s-2}};
                \draw[vec] (0.6,-1.2) \braidup (0.8,0);
                \draw[vec] (0.7,-0.6) node[anchor=west] {\strandlabel{r-s+1}};
                \draw[vec] (-0.1,-0.6) to[out=down,in=down,looseness=1.5] (0.2,-0.6) -- (0.2,-0.4) \braidup (0.5,0);
                \draw[wipe] (0.5,-0.4) \braidup (0.2,0);
                \draw[vec] (0.4,-1.2) \braidup (0.5,-0.4) \braidup (0.2,0);
                \draw[spin] (0,0) -- (0.9,0) arc(-90:90:0.15) -- (0,0.3) arc(90:270:0.15);
                \altbox{-0.5,-0.6}{0.1,-0.3}{s-1};
                \altbox{0,-1.5}{0.8,-1.2}{r};
                \draw[vec] (0.4,-1.5) to[out=down,in=down] (2,-1.5) -- (2,0.3) to[out=up,in=up] (0.45,0.3);
            \end{tikzpicture}
            - \frac{q^{s-1}-q^{1-s}}{1+q^{3-2s} \kappa^{-2}}
            \begin{tikzpicture}[anchorbase]
                \draw[spin] (0,0) -- (0.9,0) arc(-90:90:0.15) -- (0,0.3) arc(90:270:0.15);
                \draw[vec] (0.1,-0.3) -- (0.1,0);
                \draw[vec] (0.2,-1) to[out=up,in=down] (0,-0.6);
                \draw[vec] (0.1,-0.8) node[anchor=east] {\strandlabel{s-2}};
                \draw[vec] (0.2,-0.6) to[out=down,in=down,looseness=1.5] (0.5,-0.6) -- (0.5,0);
                \draw[vec] (0.6,-1) to[out=up,in=down] (0.8,0);
                \altbox{-0.3,-0.6}{0.3,-0.3}{s-1};
                \altbox{0,-1.3}{0.8,-1}{r};
                \draw[vec] (0.4,-1.3) to[out=down,in=down,looseness=1.5] (1.2,-1.3) -- (1.2,0.3) to[out=up,in=up,looseness=1.5] (0.45,0.3);
            \end{tikzpicture}
        \right)
        \\
        \overset{\cref{oister}}{\underset{\cref{asymidem}}{=}}
        \frac{1}{[s]}
        \left(
            \big( q^{s-1} d_\Vgo - [s-1] (\kappa^{-2} + q) \big)\,
            \begin{tikzpicture}[anchorbase]
                \draw[spin] (-0.1,0) -- (0.1,0) arc(-90:90:0.15) -- (-0.1,0.3) arc(90:270:0.15);
                \draw[vec] (0,0) -- (0,-0.65) arc(180:360:0.2) -- (0.4,0.3) arc(0:180:0.2);
                \altbox{-0.2,-0.65}{0.2,-0.35}{r};
            \end{tikzpicture}
            + \left( q[s-1] - \frac{q^{s-1}-q^{1-s}}{1+q^{3-2s} \kappa^{-2}} \right)
            \begin{tikzpicture}[anchorbase]
                \draw[spin] (0,0) -- (0.9,0) arc(-90:90:0.15) -- (0,0.3) arc(90:270:0.15);
                \draw[vec] (0.1,-0.3) -- (0.1,0);
                \draw[vec] (0.2,-1) to[out=up,in=down] (0,-0.6);
                \draw[vec] (0.1,-0.8) node[anchor=east] {\strandlabel{s-2}};
                \draw[vec] (0.2,-0.6) to[out=down,in=down,looseness=1.5] (0.5,-0.6) -- (0.5,0);
                \draw[vec] (0.6,-1) to[out=up,in=down] (0.8,0);
                \altbox{-0.3,-0.6}{0.3,-0.3}{s-1};
                \altbox{0,-1.3}{0.8,-1}{r};
                \draw[vec] (0.4,-1.3) to[out=down,in=down,looseness=1.5] (1.2,-1.3) -- (1.2,0.3) to[out=up,in=up,looseness=1.5] (0.45,0.3);
            \end{tikzpicture}
        \right).
    \end{multline*}
    Then the result follows after simplifying the coefficients.
\end{proof}

\begin{lem}
    For $1 \le s \le r$, we have
    \begin{equation} \label{burrito}
        \begin{tikzpicture}[anchorbase]
            \draw[spin] (0,0) -- (0.9,0) arc(-90:90:0.15) -- (0,0.3) arc(90:270:0.15);
            \draw[vec] (0.1,-0.3) -- (0.1,0);
            \draw[vec] (0.2,-1) to[out=up,in=down] (0,-0.6);
            \draw[vec] (0.1,-0.8) node[anchor=east] {\strandlabel{s-1}};
            \draw[vec] (0.2,-0.6) to[out=down,in=down,looseness=1.5] (0.5,-0.6) -- (0.5,0);
            \draw[vec] (0.6,-1) to[out=up,in=down] (0.8,0);
            \altbox{-0.1,-0.6}{0.3,-0.3}{s};
            \altbox{0,-1.3}{0.8,-1}{r};
            \draw[vec] (0.4,-1.3) to[out=down,in=down,looseness=1.5] (1.2,-1.3) -- (1.2,0.3) to[out=up,in=up,looseness=1.5] (0.45,0.3);
        \end{tikzpicture}
        \, =
        \frac{(q^{-1}\kappa^{-2} + 1)(q^{2-s} - \kappa^2)(q^{2-s} + \kappa^2)}{(q - q^{-1})(q^{3-2s} + \kappa^2)}
        \,
        \begin{tikzpicture}[anchorbase]
            \draw[spin] (-0.1,0) -- (0.1,0) arc(-90:90:0.15) -- (-0.1,0.3) arc(90:270:0.15);
            \draw[vec] (0,0) -- (0,-0.65) arc(180:360:0.2) -- (0.4,0.3) arc(0:180:0.2);
            \altbox{-0.2,-0.65}{0.2,-0.35}{r};
        \end{tikzpicture}
        \ .
    \end{equation}
\end{lem}

\begin{proof}
    We prove the result by induction on $s$.  When $s=1$, \cref{burrito} becomes
    \[
        \begin{tikzpicture}[anchorbase]
            \draw[spin] (0,0) -- (0.9,0) arc(-90:90:0.15) -- (0,0.3) arc(90:270:0.15);
            \draw[vec] (0.1,0) -- (0.1,-0.1) to[out=down,in=down,looseness=1.5] (0.5,-0.1) -- (0.5,0);
            \draw[vec] (0.6,-0.7) to[out=up,in=down] (0.8,0);
            \altbox{0,-1}{0.8,-0.7}{r};
            \draw[vec] (0.4,-1) to[out=down,in=down,looseness=1.5] (1.2,-1) -- (1.2,0.3) to[out=up,in=up,looseness=1.5] (0.45,0.3);
        \end{tikzpicture}
        \, = d_\Vgo\,
        \begin{tikzpicture}[anchorbase]
            \draw[spin] (-0.1,0) -- (0.1,0) arc(-90:90:0.15) -- (-0.1,0.3) arc(90:270:0.15);
            \draw[vec] (0,0) -- (0,-0.65) arc(180:360:0.2) -- (0.4,0.3) arc(0:180:0.2);
            \altbox{-0.2,-0.65}{0.2,-0.35}{r};
        \end{tikzpicture}
        \ ,
    \]
    which holds by \cref{bump}.  The inductive step follows from \cref{pork} and the fact that
    \begin{multline*}
        \frac{(q\kappa^2 + 1)(q^{1-s} \kappa^{-2} - q^{s-2})}{q^s - q^{-s}}
        \\
        + \frac{[s-1]}{[s]} \frac{q^{2s-4} \kappa^2 + q}{q^{2s-3} \kappa^2 + 1} \frac{(q^{-1}\kappa^{-2} + 1)(q^{2-(s-1)} - \kappa^2)(q^{2-(s-1)} + \kappa^2)}{(q - q^{-1})(q^{3-2(s-1)} + \kappa^2)}
        \\
        = \frac{(q^{-1}\kappa^{-2} + 1)(q^{2-s} - \kappa^2)(q^{2-s} + \kappa^2)}{(q - q^{-1})(q^{3-2s} + \kappa^2)}.
        \qedhere
    \end{multline*}
\end{proof}

\begin{prop} \label{royal}
    For $r \ge 0$, we have
    \begin{equation} \label{pain}
        \begin{tikzpicture}[anchorbase]
            \draw[spin] (-0.1,0) -- (0.1,0) arc(-90:90:0.15) -- (-0.1,0.3) arc(90:270:0.15);
            \draw[vec] (0,0) -- (0,-0.65) arc(180:360:0.25) -- (0.5,0.3) arc(0:180:0.25);
            \altbox{-0.3,-0.65}{0.3,-0.35}{r+1};
        \end{tikzpicture}
        = \frac{(q\kappa^2+1)(\kappa^{-2}-q^{2r-2}\kappa^2)}{(q-q^{-1})(q^{2r-1}\kappa^2+1)}\
        \begin{tikzpicture}[anchorbase]
            \draw[spin] (-0.1,0) -- (0.1,0) arc(-90:90:0.15) -- (-0.1,0.3) arc(90:270:0.15);
            \draw[vec] (0,0) -- (0,-0.65) arc(180:360:0.2) -- (0.4,0.3) arc(0:180:0.2);
            \altbox{-0.2,-0.65}{0.2,-0.35}{r};
        \end{tikzpicture}
        = \left( d_\Vgo - \frac{q^{r-1}[r](q\kappa^2+1)}{q^{2r-1}\kappa^2+1} \right)
        \begin{tikzpicture}[anchorbase]
            \draw[spin] (-0.1,0) -- (0.1,0) arc(-90:90:0.15) -- (-0.1,0.3) arc(90:270:0.15);
            \draw[vec] (0,0) -- (0,-0.65) arc(180:360:0.2) -- (0.4,0.3) arc(0:180:0.2);
            \altbox{-0.2,-0.65}{0.2,-0.35}{r};
        \end{tikzpicture}
        \ .
    \end{equation}
\end{prop}

\begin{proof}
    We compute
    \begin{multline*}
        \begin{tikzpicture}[anchorbase]
            \draw[spin] (-0.1,0) -- (0.1,0) arc(-90:90:0.15) -- (-0.1,0.3) arc(90:270:0.15);
            \draw[vec] (0,0) -- (0,-0.65) arc(180:360:0.25) -- (0.5,0.3) arc(0:180:0.25);
            \altbox{-0.3,-0.65}{0.3,-0.35}{r+1};
        \end{tikzpicture}
        \overset{\cref{asymdef}}{=}
        \frac{1}{[r+1]}
        \left(
            q^r\,
            \begin{tikzpicture}[anchorbase]
                \draw[spin] (-0.1,0) -- (0.4,0) arc(-90:90:0.15) -- (-0.1,0.3) arc(90:270:0.15);
                \draw[vec] (0,0) -- (0,-0.35);
                \draw[vec] (0,-0.6) to[out=down,in=down] (0.9,-0.6) -- (0.9,0.3) to[out=up,in=up,looseness=1.5] (0,0.3);
                \draw[vec] (0.3,0) -- (0.3,-0.2) to[out=down,in=down] (0.7,-0.2) -- (0.7,0.3) to[out=up,in=up] (0.3,0.3);
                \altbox{-0.2,-0.65}{0.2,-0.35}{r};
            \end{tikzpicture}
            - [r]
            \begin{tikzpicture}[anchorbase]
                \draw[spin] (-0.5,0) -- (0.4,0) arc(-90:90:0.15) -- (-0.5,0.3) arc(90:270:0.15);
                \draw[vec] (-0.2,-0.3) -- (-0.2,0);
                \draw[vec] (-0.1,-0.6) \braiddown (0.3,-1.1) -- (0.3,-1.3) to[out=down,in=down,looseness=1.5] (0.7,-1.3) -- (0.7,0.3) to[out=up,in=up,looseness=1.5] (0.3,0.3);
                \draw[vec] (-0.2,-1.4) to[out=down,in=down,looseness=1.5] (0.9,-1.4) -- (0.9,0.3) to[out=up,in=up,looseness=1.3] (0,0.3);
                \draw[vec] (-0.3,-0.6) -- (-0.3,-1.1) node[midway,anchor=east] {\strandlabel{r-1}};
                \draw[wipe] (-0.1,-1.1) \braidup (0.3,-0.6);
                \draw[vec] (-0.1,-1.1) \braidup (0.3,-0.6) -- (0.3,0);
                \altbox{-0.5,-0.6}{0.1,-0.3}{r};
                \altbox{-0.5,-1.1}{0.1,-1.4}{r};
            \end{tikzpicture}
            - \frac{q^r-q^{-r}}{1+q^{1-2r}\kappa^{-2}}
            \begin{tikzpicture}[anchorbase]
                \draw[spin] (-0.5,0) -- (0.4,0) arc(-90:90:0.15) -- (-0.5,0.3) arc(90:270:0.15);
                \draw[vec] (-0.2,-0.3) -- (-0.2,0);
                \draw[vec] (-0.1,-0.6) to[out=down,in=down,looseness=1.5] (0.3,-0.6) -- (0.3,0);
                \draw[vec] (-0.1,-1.1) to[out=up,in=up,looseness=1.5] (0.3,-1.1) -- (0.3,-1.3) to[out=down,in=down,looseness=1.5] (0.7,-1.3) -- (0.7,0.3) to[out=up,in=up,looseness=1.5] (0.3,0.3);
                \draw[vec] (-0.2,-1.4) to[out=down,in=down,looseness=1.5] (0.9,-1.4) -- (0.9,0.3) to[out=up,in=up,looseness=1.3] (0,0.3);
                \draw[vec] (-0.3,-0.6) -- (-0.3,-1.1) node[midway,anchor=east] {\strandlabel{r-1}};
                \altbox{-0.5,-0.6}{0.1,-0.3}{r};
                \altbox{-0.5,-1.1}{0.1,-1.4}{r};
            \end{tikzpicture}
        \right)
        \\
        \overset{\cref{bump}}{=}
        \frac{1}{[r+1]}
        \left(
            q^r d_\Vgo
            \begin{tikzpicture}[anchorbase]
                \draw[spin] (-0.1,0) -- (0.1,0) arc(-90:90:0.15) -- (-0.1,0.3) arc(90:270:0.15);
                \draw[vec] (0,0) -- (0,-0.65) arc(180:360:0.2) -- (0.4,0.3) arc(0:180:0.2);
                \altbox{-0.2,-0.65}{0.2,-0.35}{r};
            \end{tikzpicture}
            \, - [r]
            \begin{tikzpicture}[anchorbase]
                \draw[vec] (-0.2,-0.3) -- (-0.2,0);
                \draw[vec] (-0.1,-0.6) to[out=down,in=down,looseness=1.5] (0.3,-0.6) \braidup (0.6,0);
                \draw[wipe] (0.6,-0.6) \braidup (0.3,0);
                \draw[vec] (-0.1,-1.2) \braidup (0.6,-0.6) \braidup (0.3,0);
                \draw[vec] (-0.2,-1.5) to[out=down,in=down,looseness=1.5] (1,-1.5) -- (1,0.3) to[out=up,in=up,looseness=1.3] (0,0.3);
                \draw[vec] (-0.3,-0.6) -- (-0.3,-1.2) node[midway,anchor=east] {\strandlabel{r-1}};
                \altbox{-0.5,-0.6}{0.1,-0.3}{r};
                \altbox{-0.5,-1.2}{0.1,-1.5}{r};
                \draw[spin] (-0.5,0) -- (0.7,0) arc(-90:90:0.15) -- (-0.5,0.3) arc(90:270:0.15);
            \end{tikzpicture}
            - \frac{q^r-q^{-r}}{1+q^{1-2r}\kappa^{-2}}
            \begin{tikzpicture}[anchorbase]
                \draw[spin] (0,0) -- (0.9,0) arc(-90:90:0.15) -- (0,0.3) arc(90:270:0.15);
                \draw[vec] (0.1,-0.3) -- (0.1,0);
                \draw[vec] (0.2,-1) to[out=up,in=down] (0,-0.6);
                \draw[vec] (0.1,-0.8) node[anchor=east] {\strandlabel{r-1}};
                \draw[vec] (0.2,-0.6) to[out=down,in=down,looseness=1.5] (0.5,-0.6) -- (0.5,0);
                \draw[vec] (0.6,-1) to[out=up,in=down] (0.8,0);
                \altbox{-0.1,-0.6}{0.3,-0.3}{r};
                \altbox{0,-1.3}{0.8,-1}{r};
                \draw[vec] (0.4,-1.3) to[out=down,in=down,looseness=1.5] (1.2,-1.3) -- (1.2,0.3) to[out=up,in=up,looseness=1.5] (0.45,0.3);
            \end{tikzpicture}
        \right)
        \\
        \overset{\cref{oister}}{\underset{\cref{asymidem}}{=}}
        \frac{1}{[r+1]}
        \left(
            \big( q^r d_\Vgo - [r] (\kappa^{-2}+q) \big)
            \begin{tikzpicture}[anchorbase]
                \draw[spin] (-0.1,0) -- (0.1,0) arc(-90:90:0.15) -- (-0.1,0.3) arc(90:270:0.15);
                \draw[vec] (0,0) -- (0,-0.65) arc(180:360:0.2) -- (0.4,0.3) arc(0:180:0.2);
                \altbox{-0.2,-0.65}{0.2,-0.35}{r};
            \end{tikzpicture}
            \, +
            \left(
                q[r] - \frac{q^r-q^{-r}}{1+q^{1-2r}\kappa^{-2}}
            \right)
            \begin{tikzpicture}[anchorbase]
                \draw[spin] (0,0) -- (0.9,0) arc(-90:90:0.15) -- (0,0.3) arc(90:270:0.15);
                \draw[vec] (0.1,-0.3) -- (0.1,0);
                \draw[vec] (0.2,-1) to[out=up,in=down] (0,-0.6);
                \draw[vec] (0.1,-0.8) node[anchor=east] {\strandlabel{r-1}};
                \draw[vec] (0.2,-0.6) to[out=down,in=down,looseness=1.5] (0.5,-0.6) -- (0.5,0);
                \draw[vec] (0.6,-1) to[out=up,in=down] (0.8,0);
                \altbox{-0.1,-0.6}{0.3,-0.3}{r};
                \altbox{0,-1.3}{0.8,-1}{r};
                \draw[vec] (0.4,-1.3) to[out=down,in=down,looseness=1.5] (1.2,-1.3) -- (1.2,0.3) to[out=up,in=up,looseness=1.5] (0.45,0.3);
            \end{tikzpicture}
        \right).
    \end{multline*}
    The result then follows after applying the $s=r$ case of \cref{burrito} and simplifying the resulting coefficient.
\end{proof}

\begin{cor}
For $r\in \mathbb{N}$, we have
    \begin{multline} \label{smoothie}
        \begin{tikzpicture}[anchorbase]
            \draw[spin] (-0.1,0) -- (0.1,0) arc(-90:90:0.15) -- (-0.1,0.3) arc(90:270:0.15);
            \draw[vec] (0,0) -- (0,-0.65) arc(180:360:0.2) -- (0.4,0.3) arc(0:180:0.2);
            \altbox{-0.2,-0.65}{0.2,-0.35}{r};
        \end{tikzpicture}
        = \left( d_\Sgo \prod_{i=0}^{r-1}\left( d_\Vgo - \frac{q^{i-1}[i](q\kappa^2+1)}{q^{2i-1}\kappa^2+1} \right) \right) 1_\one
        \\
        = \left( d_\Sgo \left( \frac{q\kappa^2+1}{q-q\inv} \right)^r\ \prod_{i=0}^{r-1} \left( \frac{\kappa^{-2} - q^{2i-2} \kappa^2}{ q^{2i-1}\kappa^2+1 } \right) \right) 1_\one
    \end{multline}
    where we interpret the middle and right-hand expressions as $d_\Sgo 1_\one$ when $r=0$.
\end{cor}

\begin{proof}
    This is a straightforward induction, using \cref{royal} and the first relation in \cref{dimrel}.
\end{proof}

We conclude this section with a lemma that will be used in \cref{sec:full}.

\begin{lem}
    We have
    \begin{align} \label{trace1}
        \begin{tikzpicture}[centerzero]
            \draw[spin] (-0.15,-0.15) -- (-0.15,0.15) arc(180:0:0.15) -- (0.15,-0.15) arc(360:180:0.15);
            \draw[spin] (-0.5,-0.15) -- (-0.5,0.15) arc(180:0:0.5) -- (0.5,-0.15) arc(360:180:0.5);
            \draw[vec]  (-0.5,0) -- (-0.15,0);
        \end{tikzpicture}
        &=0 \qquad \text{if } \kappa^2 \ne 1,
        \\ \label{trace2}
        \begin{tikzpicture}[centerzero]
            \draw[spin] (-0.15,-0.15) -- (-0.15,0.15) arc(180:0:0.15) -- (0.15,-0.15) arc(360:180:0.15);
            \draw[spin] (-0.5,-0.15) -- (-0.5,0.15) arc(180:0:0.5) -- (0.5,-0.15) arc(360:180:0.5);
            \draw[vec]  (-0.5,0.15) -- (-0.15,0.15);
            \draw[vec]  (-0.5,-0.15) -- (-0.15,-0.15);
        \end{tikzpicture}
        &= d_\Vgo d_\Sgo^2.
    \end{align}
\end{lem}

\begin{proof}
    Equation \cref{trace1} follows from the $r=1$ case of \cref{Delprop}.  To prove \cref{trace2}, we first compute
    \begin{equation} \label{pink1}
        0
        \overset{\cref{Deligne}}{=} [2]
        \
        \begin{tikzpicture}[anchorbase]
            \draw[spin] (-0.1,0) -- (0.1,0) arc(-90:90:0.15) -- (-0.1,0.3) arc(90:270:0.15);
            \draw[vec] (0,-1) -- (0,0);
            \altbox{-0.2,-0.65}{0.2,-0.35}{2};
        \end{tikzpicture}
        \overset{\cref{ski}}{=}
        q\
        \begin{tikzpicture}[anchorbase]
            \draw[spin] (-0.2,0.4) -- (0.2,0.4) arc(-90:90:0.15) -- (-0.2,0.7) arc(90:270:0.15);
            \draw[vec] (-0.15,-0.4) -- (-0.15,0.4);
            \draw[vec] (0.15,-0.4) -- (0.15,0.4);
        \end{tikzpicture}
        -
        \begin{tikzpicture}[anchorbase]
            \draw[vec] (0.15,-0.4) -- (-0.15,0.4);
            \draw[wipe] (-0.15,-0.4) -- (0.15,0.4);
            \draw[vec] (-0.15,-0.4) -- (0.15,0.4);
            \draw[spin] (-0.2,0.4) -- (0.2,0.4) arc(-90:90:0.15) -- (-0.2,0.7) arc(90:270:0.15);
        \end{tikzpicture}
        - \frac{q-q^{-1}}{1+q^{-1}\kappa^{-2}}\
        \begin{tikzpicture}[anchorbase]
            \draw[spin] (-0.2,0.4) -- (0.2,0.4) arc(-90:90:0.15) -- (-0.2,0.7) arc(90:270:0.15);
            \draw[vec] (-0.15,-0.4) -- (-0.15,-0.3) arc(180:0:0.15) -- (0.15,-0.4);
            \draw[vec] (-0.15,0.4) -- (-0.15,0.3) arc(180:360:0.15) -- (0.15,0.4);
        \end{tikzpicture}
        \overset{\cref{bump}}{=}
        q\
        \begin{tikzpicture}[anchorbase]
            \draw[spin] (-0.2,0.4) -- (0.2,0.4) arc(-90:90:0.15) -- (-0.2,0.7) arc(90:270:0.15);
            \draw[vec] (-0.15,-0.4) -- (-0.15,0.4);
            \draw[vec] (0.15,-0.4) -- (0.15,0.4);
        \end{tikzpicture}
        -
        \begin{tikzpicture}[anchorbase]
            \draw[vec] (0.15,-0.4) -- (-0.15,0.4);
            \draw[wipe] (-0.15,-0.4) -- (0.15,0.4);
            \draw[vec] (-0.15,-0.4) -- (0.15,0.4);
            \draw[spin] (-0.2,0.4) -- (0.2,0.4) arc(-90:90:0.15) -- (-0.2,0.7) arc(90:270:0.15);
        \end{tikzpicture}
        - \frac{q-q^{-1}}{1+q^{-1}\kappa^{-2}} d_\Vgo\
        \begin{tikzpicture}[anchorbase]
            \draw[spin] (-0.2,0.4) -- (0.2,0.4) arc(-90:90:0.15) -- (-0.2,0.7) arc(90:270:0.15);
            \draw[vec] (-0.15,-0.2) -- (-0.15,-0.1) arc(180:0:0.15) -- (0.15,-0.2);
        \end{tikzpicture}
        .
    \end{equation}
    On the other hand, we also have
    \begin{equation} \label{pink2}
        \begin{tikzpicture}[anchorbase]
            \draw[spin] (-0.2,0.4) -- (0.2,0.4) arc(-90:90:0.15) -- (-0.2,0.7) arc(90:270:0.15);
            \draw[vec] (-0.15,-0.4) -- (-0.15,0.4);
            \draw[vec] (0.15,-0.4) -- (0.15,0.4);
        \end{tikzpicture}
        + q\
        \begin{tikzpicture}[anchorbase]
            \draw[vec] (0.15,-0.4) -- (-0.15,0.4);
            \draw[wipe] (-0.15,-0.4) -- (0.15,0.4);
            \draw[vec] (-0.15,-0.4) -- (0.15,0.4);
            \draw[spin] (-0.2,0.4) -- (0.2,0.4) arc(-90:90:0.15) -- (-0.2,0.7) arc(90:270:0.15);
        \end{tikzpicture}
        \overset{\cref{oist}}{=} (q \kappa^2 + 1)\
        \begin{tikzpicture}[anchorbase]
            \draw[spin] (-0.2,0.4) -- (0.2,0.4) arc(-90:90:0.15) -- (-0.2,0.7) arc(90:270:0.15);
            \draw[vec] (-0.15,-0.2) -- (-0.15,-0.1) arc(180:0:0.15) -- (0.15,-0.2);
        \end{tikzpicture}
        .
    \end{equation}
    Adding $q$ times \cref{pink1} to \cref{pink2} gives
    \begin{equation} \label{pink3}
        (1+q^2)\
        \begin{tikzpicture}[anchorbase]
            \draw[spin] (-0.2,0.4) -- (0.2,0.4) arc(-90:90:0.15) -- (-0.2,0.7) arc(90:270:0.15);
            \draw[vec] (-0.15,-0.2) -- (-0.15,0.4);
            \draw[vec] (0.15,-0.2) -- (0.15,0.4);
        \end{tikzpicture}
        = \left( q \kappa^2 + 1 + \frac{q^2-1}{1+q^{-1}\kappa^{-2}} d_\Vgo \right)
        \begin{tikzpicture}[anchorbase]
            \draw[spin] (-0.2,0.4) -- (0.2,0.4) arc(-90:90:0.15) -- (-0.2,0.7) arc(90:270:0.15);
            \draw[vec] (-0.15,-0.2) -- (-0.15,-0.1) arc(180:0:0.15) -- (0.15,-0.2);
        \end{tikzpicture}
        \overset{\cref{dimrel}}{=} \left( q \kappa^2 + 1 + \frac{q^2-1}{1+q^{-1}\kappa^{-2}} d_\Vgo \right) d_\Sgo\ \capmor{vec}\, .
    \end{equation}
    We also have
    \[
        q \kappa^2 + 1 + \frac{q^2-1}{1+q^{-1}\kappa^{-2}} d_\Vgo
        \overset{\cref{dimrel}}{=} q\kappa^2+1+q^2\kappa^2(\kappa^{-2}-q\inv) = 1+q^2.
    \]
    Since $q^2\neq -1$, we can cancel the common factor in \cref{pink3} to get
    \[
        \begin{tikzpicture}[anchorbase]
            \draw[spin] (-0.2,0.4) -- (0.2,0.4) arc(-90:90:0.15) -- (-0.2,0.7) arc(90:270:0.15);
            \draw[vec] (-0.15,-0.2) -- (-0.15,0.4);
            \draw[vec] (0.15,-0.2) -- (0.15,0.4);
        \end{tikzpicture} = d_\Sgo\ \capmor{vec}\, .
    \]
    Therefore
    \[
        \begin{tikzpicture}[centerzero]
            \draw[spin] (-0.15,-0.15) -- (-0.15,0.15) arc(180:0:0.15) -- (0.15,-0.15) arc(360:180:0.15);
            \draw[spin] (-0.5,-0.15) -- (-0.5,0.15) arc(180:0:0.5) -- (0.5,-0.15) arc(360:180:0.5);
            \draw[vec]  (-0.5,0.15) -- (-0.15,0.15);
            \draw[vec]  (-0.5,-0.15) -- (-0.15,-0.15);
        \end{tikzpicture}
        = d_\Sgo \
        \begin{tikzpicture}[centerzero]
            \draw[spin] (-0.15,-0.15) -- (-0.15,0.15) arc(180:0:0.15) -- (0.15,-0.15) arc(360:180:0.15);
            \draw[spin] (-0.5,-0.15) -- (-0.5,0.15) arc(180:0:0.5) -- (0.5,-0.15) arc(360:180:0.5);
            \draw[vec] (-0.5,0.15) -- (-0.4,0.15) arc (90:-90:0.15) -- (-0.5,-0.15);
        \end{tikzpicture}
        \overset{\cref{bump}}{\underset{\cref{dimrel}}{=}} d_\Vgo d_\Sgo^2.
        \qedhere
    \]
\end{proof}

\begin{defin}
    When $N = 2n + 1$, let $\overline{\QSB}(q,t,\kappa,d_\Sgo)$ denote the quotient of $\QSB(q,t,\kappa,d_\Sgo)$ by the relation
    \begin{equation} \label{extra}
        \begin{tikzpicture}[centerzero]
            \draw[spin] (0,0.25) circle(0.15);
            \draw[spin] (0,-0.25) circle(0.15);
            \draw[vec] (0,0.4) -- (0,1.1);
            \draw[vec] (0,-0.4) -- (0,-1.1);
            \altbox{-0.2,-0.9}{0.2,-0.6}{N};
            \altbox{-0.2,0.6}{0.2,0.9}{N};
        \end{tikzpicture}
        =
        d_\Sgo^2 \left( \frac{q\kappa^2+1}{q-q\inv} \right)^N\ \prod_{i=0}^{N-1} \left( \frac{\kappa^{-2} - q^{2i-2} \kappa^2}{ q^{2i-1}\kappa^2+1 } \right)
        \begin{tikzpicture}[centerzero]
            \draw[vec] (0,-1.1) -- (0,1.1);
            \altbox{-0.2,-0.15}{0.2,0.15}{N};
        \end{tikzpicture}
        \ .
    \end{equation}
    For $N = 2n$, let $\overline{\QSB}(q,t,\kappa,d_\Sgo)= \QSB(q,t,\kappa,d_\Sgo)$.
\end{defin}

\section{The quantized enveloping algebra}

In this section and the next, we collect important definitions and properties of the quantized enveloping algebras associated to the orthogonal groups $\rO(N)$.  Throughout this section we work over the field
\[
    \kk = \C \big( q^{\pm \frac{1}{4}} \big).
\]
Recall that $n = \lfloor \frac{N}{2} \rfloor$, so that $N=2n$ (type $D_n$) or $N=2n+1$ (type $B_n$).

\subsection{Definition of the quantized enveloping algebra\label{subsec:Uq}}

We briefly recall the definition of the quantized enveloping algebra in types $B_n$ and $D_n$.  We choose the following labelling of the nodes of the Dynkin diagrams:
\begin{equation} \label{Dynkin}
    B_n :
    \begin{tikzpicture}[centerzero]
        \draw[style=double,double distance=2pt] (1,0) -- (0,0);
        \draw[style=double,double distance=2pt,-{Classical TikZ Rightarrow[length=3mm,width=4mm]}] (1,0) -- (0.35,0);
        \filldraw (0,0) circle (2pt) node[anchor=south] {$1$};
        \filldraw (1,0) circle (2pt) node[anchor=south] {$2$};
        \filldraw (2,0) circle (2pt) node[anchor=south] {$3$};
        \filldraw (3.5,0) circle (2pt) node[anchor=south] {$n$};
        \draw (1,0) -- (2.25,0);
        \draw (3.25,0) -- (3.5,0);
        \draw[dotted] (2.25,0) -- (3.25,0);
    \end{tikzpicture}
    \qquad \qquad
    D_n :
    \begin{tikzpicture}[centerzero]
        \filldraw (0,-0.5) circle (2pt) node[anchor=north] {$1$};
        \filldraw (0,0.5) circle (2pt) node[anchor=south] {$2$};
        \filldraw (1,0) circle (2pt) node[anchor=south] {$3$};
        \filldraw (2,0) circle (2pt) node[anchor=south] {$4$};
        \filldraw (3.5,0) circle (2pt) node[anchor=south] {$n$};
        \draw (0,-0.5) -- (1,0) -- (2,0) -- (2.25,0);
        \draw (3.25,0) -- (3.5,0);
        \draw (0,0.5) -- (1,0);
        \draw[dotted] (2.25,0) -- (3.25,0);
    \end{tikzpicture}
\end{equation}
We normalize the pairing $( \cdot, \cdot )$ on the weight lattice $\bigoplus_{i=1}^n \Z \epsilon_i$ so that the long roots $\alpha$ satisfy $(\alpha, \alpha)=2$.  Thus
\begin{equation} \label{astrophage}
    (\epsilon_i, \epsilon_j) = \delta_{ij},
\end{equation}
and the simple roots are given by
\begin{equation} \label{simproots}
    \alpha_i = \epsilon_i - \epsilon_{i-1},\quad 2 \le i \le n,
    \qquad
    \alpha_1 =
    \begin{cases}
        \epsilon_1 + \epsilon_2 & \text{if $N=2n$ (type $D_n$)}, \\
        \epsilon_1 & \text{if $N=2n+1$ (type $B_n$)}.
    \end{cases}
\end{equation}
The positive roots are
\begin{equation} \label{posroots}
    \Phi^+
    =
    \begin{cases}
        \{ \epsilon_i \pm \epsilon_j : 1 \le j < i \le n\} & \text{if } N=2n \text{ (type $D_n$)}, \\
        \{ \epsilon_i \pm \epsilon_j, \epsilon_k : 1 \le j < i \le n,\ 1 \le k \le n \} & \text{if } N=2n+1 \text{ (type $B_n$)}.
    \end{cases}
\end{equation}

Let $(a_{ij})_{i,j=1}^n$ be a Cartan matrix of type $B_n$ or $D_n$.  More precisely,
\[
    a_{ij} = 2 \frac{(\alpha_i, \alpha_j)}{(\alpha_i,\alpha_i)},
\]
where the $\alpha_i$ are the simple roots, and $(a_{ij})_{i,j \in I}$ is
\[
    \begin{pmatrix}
        2 & -2 & 0 & 0 & 0 & \cdots & 0 \\
        -1 & 2 & -1 & 0 & 0 & \cdots & 0 \\
        0 & -1 & 2 & -1 & 0 & \cdots & 0 \\
        \vdots & \ddots & \ddots & \ddots & \ddots & \ddots & \vdots \\
        0 & \cdots & 0 & -1 & 2 & -1 & 0 \\
        0 & \cdots & \cdots & 0 & -1 & 2 & -1 \\
        0 & \cdots & \cdots & \cdots & 0 & -1 & 2
    \end{pmatrix}
    \quad \text{or} \quad
    \begin{pmatrix}
        2 & 0 & -1 & 0 & 0 & \cdots & 0 \\
        0 & 2 & -1 & 0 & 0 & \cdots & 0 \\
        -1 & -1 & 2 & -1 & 0 & \cdots & 0 \\
        \vdots & \ddots & \ddots & \ddots & \ddots & \ddots & \vdots \\
        0 & \cdots & 0 & -1 & 2 & -1 & 0 \\
        0 & \cdots & \cdots & 0 & -1 & 2 & -1 \\
        0 & \cdots & \cdots & 0 & 0 & -1 & 2
    \end{pmatrix}
\]
in type $B_n$ or $D_n$, respectively.

Set $d_i = \frac{(\alpha_i, \alpha_i)}{2}$.  So, in type $B_n$, we have $d_1 = \frac{1}{2}$ and $d_i = 1$ for $2 \le i \le n$.  In type $D_n$, we have $d_i = 1$ for all $1 \le i \le n$.  For $1 \le i \le n$, define the following elements of $\Z[q^{\pm \frac{1}{2}}]$:
\begin{gather*}
    q_i := q^{d_i},\qquad
    [m]_i := \frac{q_i^m - q_i^{-m}}{q_i-q_i^{-1}},\qquad
    m \in \Z,
    \\
    [m]!_i := [m]_i [m-1]_i \dotsb [1]_i,\qquad
    \qbinom{m}{r}_i = \frac{[m]!_i}{[r]!_i [m-r]!_i},\qquad
    m,r \in \N,\ r \le m.
\end{gather*}
When $d_i=1$, we will often drop the subscripts $i$ above.  In this case, the $[m]$ are referred to as \emph{quantum integers} and the $\qbinom{m}{r}$ as \emph{quantum binomial coefficients}.  A straightforward proof by induction yields the following generating function for the quantum binomial coefficients:
\begin{equation} \label{bird}
    \prod_{j=1}^m \left( 1 + q^{2j-m-1} x \right)
    = \sum_{k=0}^m \qbinom{m}{k} x^k,\qquad
    m \in \N.
\end{equation}
\details{
    The result is trivial for $m=0$.  For the induction step, we have
    \[
        \prod_{j=1}^{m+1} \left( 1 + q^{2j-m-2} x \right)
        = (1 + q^m x) \sum_{k=0}^m \qbinom{m}{k} q^{-k} x^k
        = \sum_{k=0}^{m+1} \left( \qbinom{m}{k} q^{-k} + \qbinom{m}{k-1} q^{m-k+1} \right) x^k.
    \]
    Since
    \begin{multline*}
        \qbinom{m}{k}q^{-k} + \qbinom{m}{k-1} q^{m-k+1}
        = \frac{[m]!}{[k]![m-k]!}q^{-k} + \frac{[m]!}{[k-1]![m-k+1]!} q^{m-k+1}
        \\
        = \frac{[m]!}{[k]![m-k+1]!} \left( \frac{q^{m-2k+1} - q^{-m-1}}{q-q^{-1}} + \frac{q^{m+1} - q^{m-2k+1}}{q-q^{-1}} \right)
        = \frac{[m+1]!}{[k]![m-k+1]!}
        = \qbinom{m+1}{k},
    \end{multline*}
    the result follows.
}

\begin{defin} \label{Udef}
    When $N=2n+1$, $n \ge 1$, let $U_q(N)$ be the $\kk$-algebra generated by $e_i, f_i, k_i^{\pm 1}$, $1 \le i \le n$, subject to the following relations for $1 \le i,j \le n$:
    \begin{gather} \label{Udef-k}
        k_i k_i^{-1} = 1 = k_i^{-1} k_i,\qquad
        k_i k_j = k_j k_i,
        \\ \label{Udef-kef}
        k_i e_j = q^{(\alpha_i,\alpha_j)} e_j k_i = q_i^{a_{ij}} e_j k_i, \qquad
        k_i f_j = q^{-(\alpha_i,\alpha_j)} f_j k_i = q_i^{-a_{ij}} f_j k_i,
        \\ \label{Udef-ef}
        e_i f_j - f_j e_i = \delta_{ij} \frac{k_i-k_i^{-1}}{q_i-q_i^{-1}},
        \\ \label{Udef-Serre}
        \sum_{r=0}^{1-a_{ij}} (-1)^r \qbinom{1-a_{ij}}{r}_i e_i^{1-a_{ij}-r} e_j e_i^r = 0 =
        \sum_{r=0}^{1-a_{ij}} (-1)^r \qbinom{1-a_{ij}}{r}_i f_i^{1-a_{ij}-r} f_j f_i^r, \quad i \ne j.
    \end{gather}
    (The second and fourth equalities in \cref{Udef-kef} follow from the fact that $q^{(\alpha_i,\alpha_j)} = q_i^{a_{ij}}$.)

    When $N=2n$, $n \ge 2$, let $U_q(N)$ be the $\kk$-algebra defined as above, but with an additional generator $\sigma$ subject to the relations
    \begin{equation} \label{sigma}
        \sigma^2 = 1,\quad
        \sigma e_i \sigma = e_{\sigma(i)},\quad
        \sigma f_i \sigma = f_{\sigma(i)},\quad
        \sigma k_i^{\pm 1} \sigma = k_{\sigma(i)}^{\pm 1},
    \end{equation}
    where, in the subscripts, $\sigma$ acts as the simple transposition $(1\, 2)$ on the set $\{1,2,\dotsc,n\}$.

    For $N=2n+1$, $U_q(N)$ is a Hopf algebra, with comultiplication, antipode, and counit given by
    \begin{gather} \label{comult}
        \Delta(e_i) = e_i \otimes k_i + 1 \otimes e_i,\quad
        \Delta(f_i) = f_i \otimes 1 + k_i^{-1} \otimes f_i,\quad
        \Delta(k_i^{\pm 1}) = k_i^{\pm 1} \otimes k_i^{\pm 1},
        \\ \label{antipode}
        \iota(e_i) = - e_i k_i^{-1},\quad
        \iota(f_i) = - k_i f_i,\quad
        \iota(k_i^{\pm 1}) = k_i^{\mp 1},
        \\
        \epsilon(e_i)=0,\quad
        \epsilon(f_i)=0,\quad
        \epsilon(k_i)=1.
    \end{gather}
    We use the notation $\iota$ to denote the antipode in order to reserve the notation $S$ for the spin module, to be introduced later.  When $N=2n$, $U_q(N)$ is again a Hopf algebra, extending the above definitions by
    \begin{equation} \label{Brazil}
        \Delta(\sigma) = \sigma \otimes \sigma,\quad
        \iota(\sigma) = \sigma^{-1} = \sigma,\quad
        \epsilon(\sigma) = 1.
    \end{equation}
\end{defin}

\begin{rem}
    \begin{enumerate}[wide]
        \item We draw the attention of the reader to some choices we have made that differ from some places in the literature.  First, we have chosen the labelling of the nodes of the Dynkin diagrams in \cref{Dynkin} to be compatible with the embedding of $B_n$ and $D_n$ inside $B_{n+1}$ and $D_{n+1}$, respectively.  Second, we have normalized the pairing on the root lattice so that the \emph{long} roots $\alpha$ satisfy $(\alpha,\alpha)=2$, whereas it is usually normalized so that the \emph{short} roots have this property.  Thus, in type $B_n$, our $U_q(2n+1)$ would usually be denoted $U_{q^{1/2}}(\fso(2n+1))$.  Our choice of normalization of the pairing allows us to state various results in the paper in a more uniform way.

        \item In type $D_n$, our $U_q(2n)$ is a smash product of the usual quantized enveloping algebra $U_q(\fso(2n))$ with the group algebra $\kk[\sigma]/(\sigma^2-1)$ of the cyclic group of order two, with conjugation by $\sigma$ acting by the outer automorphism of $U_q(\fso(2n))$ that swaps the vertices $1$ and $2$ of the Dynkin diagram and fixes all other vertices.  This extension, sometimes denoted $U_q(\fo(N))$, is analogous to considering the pin group instead of the spin group; see \cite[(4.23)]{MS24}.

        \item For the Hopf algebra structure defined in \cref{Udef}, we have followed the conventions of \cite[Def.-Prop.~9.1.1]{CP95}, which also matches \cite[\S3.1]{BER24}.  Note that a different Hopf algebra structure is chosen in \cite[Def.~2.1.1]{DF94}, while yet another is chosen in \cite[\S1.1]{Hay90}.  Note also that our $k_i$ and $q$ are the $k_i^2$ and $q^2$, respectively, of \cite[\S1.1]{Hay90}.

    \end{enumerate}
\end{rem}

It will be useful to extend our definition of $U_q(N)$ to allow for $N$ to be any natural number.  For small values of $N$, we define $U_q(N)$ as follows.

\begin{defin} \label{limbo}
    \begin{enumerate}
        \item We define $U_q(\fso(0)) := \kk$ and $U_q(0) := \kk[\sigma]/(\sigma^2-1)$ to be the group algebra of the group of order two, with Hopf algebra structure as in \cref{Brazil}.  We consider all modules to be of type $1$.

        \item We define $U_q(\fso(1)) = U_q(1) := \kk[\xi]/(\xi^2-1)$ to be the group algebra of the group of order $2$, with Hopf algebra structure given by
            \[
                \Delta(\xi) = \xi \otimes \xi,\quad
                \iota(\xi) = \xi^{-1} = \xi,\quad
                \epsilon(\xi) = 1.
            \]
            We consider all modules to be of type $1$.

        \item We define $U_q(\fso(2))$ to be the associative algebra generated by $k,k^{-1}$, subject to the relations $k k^{-1} = 1 = k^{-1} k$.  The Hopf algebra structure is given by
            \begin{equation} \label{Panama}
                \Delta(k) = k \otimes k,\quad
                \iota(k^{\pm 1}) = k^{\mp 1},\quad
                \epsilon(k) = 1.
            \end{equation}
            We define $U_q(2)$ to be the associative algebra generated by $k,k^{-1},\sigma$, subject to the relations
            \[
                k k^{-1} = 1 = k^{-1} k,\quad
                \sigma^2 = 1,\quad
                \sigma k = k^{-1} \sigma.
            \]
            The Hopf algebra structure is given by \cref{Panama,Brazil}.  A module is said to be type $1$ if the eigenvalues of the action of $k$ lie in $q^{\frac{1}{2} \Z}$.
    \end{enumerate}
\end{defin}

Throughout this paper, we assume $U_q(\fso(N))$-modules are type $1$.  Let $U_q(\fso(N))$-mod denote the category of finite-dimensional $U_q(\fso(N))$-modules (of type $1$) and let $\fso(N)$-mod denote the category of finite-dimensional $\fso(N)$-modules.  Both these categories are semisimple, with irreducible objects given by highest-weight representations with dominant integral highest weights.  Furthermore, tensor-product multiplicities are the same in the two categories, as are dimensions of homomorphism spaces between corresponding modules.  We let $U_q(N)$-mod denote the category of finite-dimensional $U_q(N)$-modules that are of type $1$; for $N \ge 3$, this means that the modules are of type $1$ when restricted to $U_q(\fso(N))$.

\begin{prop} \label{semisimple}
    The category $U_q(N)$-mod is semisimple.
\end{prop}

\begin{proof}
    If $N=2n+1$, then $U_q(N) = U_q(\fso(N))$, and so the result is well known.  Now suppose that $N=2n$ and that $M$ is a $U_q(N)$-submodule of $M'$.  Since $U_q(\fso(N))$-mod is semisimple, there exists a $U_q(\fso(N))$-module splitting, that is, a $U_q(\fso(N))$-module homomorphism $\pi \colon M' \to M$ such that $\pi(v) = v$ for all $v \in M$.  Then $\frac{1}{2}(\pi + \sigma \pi \sigma) \colon M' \to M$ is a splitting of $U_q(N)$-modules.
\end{proof}

We let $X$ denote the weight lattice of $U_q(\fso(N))$ and let $X^+$ denote the positive weight lattice.  For $N \le 2$, this is defined as follows:
\begin{itemize}
    \item When $N \in \{0,1\}$, we define $X = X^+ = 0$.
    \item When $N = 2$, we define $X = X^+ = \frac{1}{2}\Z$ and $(\lambda,\mu) = \lambda \mu$.
\end{itemize}
For $N \in \{0,1\}$, all vectors in all modules are considered to have weight zero.  For $N=2$, a vector $v$ has weight $\lambda$ if $k v = q^\lambda v$.

For $N \ge 2$, we let $L_q(\lambda)$, $\lambda \in X^+$, denote the simple $U_q(\fso(N))$-module of highest weight $\lambda$.  For $N=2$, this is the one-dimensional module on which $k$ acts by $q^\lambda$.

\subsection{Modules in type $D$\label{subsec:modulesD}}

Throughout this subsection, we assume that $N$ is even, so that $U_q(N) = U_q(\fso(N)) \rtimes \kk[\sigma]/(\sigma^2-1)$.  The category of finite-dimensional $\Spin(N)$-modules is equivalent to the category $\fso(N)$-mod, hence to the category $U_q(\fso(N))$-mod, as monoidal categories.  Thus, the category $U_q(N)$-mod is described via Clifford theory in a manner completely analogous to the category of finite-dimensional $\Pin(N)$-modules.  We refer the reader to \cite[\S4.2]{MS24}, where the latter category is described in detail, and state some of the corresponding facts for $U_q(N)$-modules that will be used in the current paper.

For a $U_q(\fso(N))$-module $W$, let $W^\sigma$ denote the $U_q(\fso(N))$-module that is equal to $W$ as a vector space, but with the twisted action
\[
    a \cdot w = (\sigma a \sigma^{-1}) w,\qquad a \in U_q(\fso(N)),\ w \in W^\sigma,
\]
where the juxtaposition $aw$ denotes the action of $a \in U_q(\fso(N))$ on $w \in W$.  We define an action of the cyclic group $\langle \sigma : \sigma^2=1 \rangle$ on the weight lattice $X$ by
\begin{equation} \label{thaw}
    \sigma \epsilon_i = (-1)^{\delta_{i1}} \epsilon_i,\qquad 1 \le i \le n,
\end{equation}
and extending by linearity.  Then we have
\[
    L_q(\lambda)^\sigma \cong L_q(\sigma \lambda).
\]

To pass between representations of $U_q(\fso(N))$ and  $U_q(N)$, we use the biadjoint pair of restriction and induction functors
\begin{equation} \label{cricket}
    \Res \colon U_q(N)\md \to U_q(\fso(N))\md
    \qquad \text{and} \qquad
    \Ind \colon U_q(\fso(N))\md \to U_q(N)\md.
\end{equation}
These satisfy
\begin{equation}\label{resind}
    \Res \circ \Ind(W) \cong W \oplus W^\sigma.
\end{equation}
The simple $U_q(N)$-modules are:
\begin{itemize}
    \item $\Ind(W) \cong \Ind(W^\sigma)$ for a simple $U_q(\fso(N))$-module $W$ with $W^\sigma \not\cong W$,
    \item the two simple summands of $\Ind(W)$ for a simple $U_q(\fso(N))$-module $W$ with $W^\sigma \cong W$.
\end{itemize}

We let $\triv^0$ denote the trivial $U_q(N)$-module and let $\triv^1$ be the one-dimensional module that is trivial as a $U_q(\fso(N))$-module and satisfies $\sigma v = -v$, $v \in \triv^1$.  If $W$ is a simple $U_q(\fso(N))$-module with $W^\sigma\cong W$, and $W'$ and $W''$ are its two lifts to $U_q(N)$-modules, then these are related by
\begin{equation} \label{trivflip}
    W' \otimes \triv^1 \cong W''.
\end{equation}

\subsection{The natural module\label{subsec:natural}}

We now describe the quantum analogue, $V$, of the natural module for $U_q(N)$.  We first suppose that $N \ge 3$.  Then $V$ is the finite-dimensional simple module with highest weight $\epsilon_n$.  Let
\begin{equation}
    \Vset =
    \begin{cases}
        \{n,n-1,\dotsc,1,-1,-2,\dotsc,-n\} & \text{if } N = 2n, \\
        \{n,n-1,\dotsc,1,0,-1,-2,\dotsc,-n\} & \text{if } N = 2n+1.
    \end{cases}
\end{equation}
The module $V$ has a basis $\{v_i\}_{i \in \Vset}$, with weights
\begin{equation} \label{hammock}
    \wt(v_0) = 0,\quad
    \wt(v_{\pm i}) = \pm \epsilon_i,\qquad
    1 \le i \le n,
\end{equation}
in which the $U_q(N)$-action is given by
\begin{gather} \label{doze1}
    e_i = E_{i,i-1} - q^{-1} E_{1-i,-i},\quad
    f_i  = E_{i-1,i} - q E_{-i,1-i},\qquad
    2 \le i \le n,
    \\ \label{doze2}
    e_1 =
    \begin{cases}
        E_{2,-1} - q^{-1} E_{1,-2} & \text{if } N=2n, \\
        (q+1) E_{1,0} - q^{-1} E_{0,-1} &\text{if } N=2n+1,
    \end{cases}
    \\ \label{doze3}
    f_1 =
    \begin{cases}
        E_{-1,2} - q E_{-2,1} & \text{if } N=2n, \\
        q^{-\frac{1}{2}} E_{0,1} - q^{\frac{1}{2}} (q+1) E_{-1,0} & \text{if } N=2n+1,
    \end{cases}
    \\ \label{doze4}
    k_i v = q^{(\alpha_i, \wt(v))} v,\qquad
    1 \le i \le n,\ v \in V,
    \\ \label{doze5}
    \sigma v_i
    =
    \begin{cases}
        -v_i & \text{if } i \notin \{\pm 1\}, \\
        -v_{-i} & \text{if } i \in \{\pm 1\},
    \end{cases}
\end{gather}
where $E_{i,j}$ is the matrix with a $1$ in the $(i,j)$-position and $0$ in all other positions, and we adopt the convention that $E_{i,j}=0$ if $i$ or $j$ does not lie in $\Vset$.  Above, and in what follows, we also adopt the convention that expressions involving $\sigma$ only apply in type $D_n$ ($N=2n$).  The module $V$ is a simple, self-dual module with highest-weight vector $v_n$.
\details{
    Throughout, $2 \le i,j \le n$.  Relations \cref{Udef-k,Udef-kef} are straightforward to verify.  We verify \cref{Udef-ef}.  For $2 \le i,j \le n$, we have
    \begin{multline*}
        (E_{i,i-1} - q^{-1} E_{1-i,-i}) (E_{j-1,j} - q E_{-j,1-j}) - (E_{j-1,j} - q E_{-j,1-j}) (E_{i,i-1} - q^{-1} E_{1-i,-i})
        \\
        = \delta_{ij} (E_{i,i} - E_{i-1,i-1} + E_{1-i,1-i} - E_{-i,-i}).
    \end{multline*}
    In type $B_n$, where $N=2n+1$, we have, for $2 \le i \le n$,
    \begin{multline*}
        (E_{i,i-1} - q^{-1} E_{1-i,-i}) \left( q^{-\frac{1}{2}} E_{0,1} - q^{\frac{1}{2}} (q+1) E_{-1,0} \right)
        \\
        - \left( q^{-\frac{1}{2}} E_{0,1} - q^{\frac{1}{2}} (q+1) E_{-1,0} \right) (E_{i,i-1} - q^{-1} E_{1-i,-i})
        = 0,
    \end{multline*}
    \[
        \big( (q+1) E_{1,0} - q^{-1} E_{0,-1} \big) (E_{i-1,i} - q E_{-i,1-i})
        - (E_{i-1,i} - q E_{-i,1-i}) \left( (q+1) E_{1,0} - q^{-1} E_{0,-1} \right)
        = 0,
    \]
    \begin{multline*}
        \big( (q+1) E_{1,0} - q^{-1} E_{0,-1} \big) \left( q^{-\frac{1}{2}} E_{0,1} - q^{\frac{1}{2}} (q+1) E_{-1,0} \right)
        \\
        - \left( q^{-\frac{1}{2}} E_{0,1} - q^{\frac{1}{2}} (q+1) E_{-1,0} \right) \big( (q+1) E_{1,0} - q^{-1} E_{0,-1} \big)
        \\
        = [2]_1 (E_{1,1} + E_{0,0} - E_{0,0} - E_{-1,-1})
        = [2]_1 (E_{1,1} - E_{-1,-1}).
    \end{multline*}
    For $i \in \Vset$, let $D_i = \sum_{j \in \Vset} q^{\delta_{ij}} E_{jj}$.  Then,
    \[
        k_i =
        \begin{cases}
            D_i D_{i-1}^{-1} D_{1-i} D_{-i}^{-1}, & 2 \le i \le n, \\
            D_1 D_{-1}^{-1} & i=1,
        \end{cases}
    \]
    and so
    \[
        \frac{k_i-k_i^{-1}}{q_i-q_i^{-1}}
        =
        \begin{cases}
            E_{i,i} - E_{i-1,i-1} + E_{1-i,1-i} - E_{-i,-i}, & 2 \le i \le n, \\
            [2]_1 (E_{1,1} - E_{-1,-1}), & i=1.
        \end{cases}
    \]
    Therefore, \cref{Udef-ef} is satisfied.

    For $N=2n$, we have, for $2 \le i \le n$,
    \begin{multline*}
        (E_{i,i-1} - q^{-1} E_{1-i,-i}) (E_{-1,2} - qE_{-2,1}) - (E_{-1,2} - qE_{-2,1}) (E_{i,i-1} - q^{-1} E_{1-i,-i})
        \\
        = \delta_{i,2} (E_{-1,1} - E_{-1,1}) = 0,
    \end{multline*}
    \begin{multline*}
        (E_{2,-1} - q^{-1} E_{1,-2}) (E_{i-1,i} - qE_{-i,1-i}) - (E_{i-1,i} - qE_{-i,1-i}) (E_{2,-1} - q^{-1} E_{1,-2})
        \\
        = \delta_{i,2} (E_{1,-1} - E_{1,-1}) = 0,
    \end{multline*}
    \begin{multline*}
        (E_{2,-1} - q^{-1} E_{1,-2}) (E_{-1,2} - q E_{-2,1}) - (E_{-1,2} - q E_{-2,1}) (E_{2,-1} - q^{-1} E_{1,-2})
        \\
        = E_{2,2} + E_{1,1} - E_{-1,-1} - E_{-2,-2}.
    \end{multline*}
    Now, we also have
    \[
        \alpha_1 = \epsilon_1 + \epsilon_2,\quad
        \alpha_i = \epsilon_i - \epsilon_{i-1},\qquad 2 \le i \le n.
    \]
    Thus,
    \[
        k_i =
        \begin{cases}
            D_i D_{i-1}^{-1} D_{1-i} D_{-i}^{-1}, & 2 \le i \le n, \\
            D_1 D_2 D_{-2}^{-1} D_{-1}^{-1}, & i=1,
        \end{cases}
    \]
    and so
    \[
        \frac{k_i-k_i^{-1}}{q_i-q_i^{-1}}
        =
        \begin{cases}
            E_{i,i} - E_{i-1,i-1} + E_{1-i,1-i}- E_{-i,-i}, & 2 \le i \le n, \\
            E_{2,2} + E_{1,1} - E_{-1,-1} - E_{-2,-2}, & i = 1.
        \end{cases}
    \]
    Therefore, \cref{Udef-ef} is satisfied.
}

When $N \le 2$, the natural module is given as follows (see \cite[Rem.~4.1]{MS24}):
\begin{itemize}
    \item When $N=0$, we have $V=0$.

    \item When $N=1$, $V$ is the one-dimensional module for $U_q(N) = \kk[\xi]/(\xi^2-1)$ given by $\xi$ acting as $1$.  We fix a nonzero vector $v_0 \in V$.

    \item When $N=2$, the module $V$ is $2$-dimensional, with basis $v_1,v_{-1}$, and action given by
        \[
            k v_{\pm 1} = q^{\pm 1} v_{\pm 1},\quad
            \sigma v_{\pm 1} = - v_{\mp 1}.
        \]
\end{itemize}

When $N \ge 3$, let $\rho$ denote half the sum of the positive roots in \cref{posroots}, so that
\begin{equation} \label{rho}
    \rho =
    \begin{cases}
        \sum_{i=1}^n (i-1) \epsilon_i & \text{when $N=2n$ (type $D_n$)}, \\
        \sum_{i=1}^n \left( i - \frac{1}{2} \right) \epsilon_i & \text{when $N=2n+1$ (type $B_n$)}.
    \end{cases}
\end{equation}
Define
\begin{equation}
    \rho_0 = 0,\quad
    \rho_i = (\rho,\epsilon_i),\quad
    \rho_{-i} = -(\rho,\epsilon_i),\qquad
    1 \le i \le n.
\end{equation}
Thus,
\begin{gather} \label{rhoeven}
    \rho_i = i-1,\quad
    \rho_{-i} = 1-i,\qquad
    1 \le i \le n,\qquad
    \text{when } N=2n,
    \\ \label{rhoodd}
    \rho_0 = 0,\quad
    \rho_i = i - \tfrac{1}{2},\quad
    \rho_{-i} = \tfrac{1}{2} - i,\qquad
    1 \le i \le n,\qquad
    \text{when } N=2n+1.
\end{gather}
When $N \in \{1,2\}$, we take \cref{rhoeven,rhoodd} as definitions.
\details{
    So, we have
    \begin{gather*}
        (\rho_n,\dotsc,\rho_1,\rho_{-1},\dotsc,\rho_{-n}) = (n - 1, n - 2, \dotsc, 1, 0, 0, -1, \dotsc, -n + 1),
        \quad \text{for } N=2n,
        \\
        (\rho_n,\dotsc,\rho_1,\rho_0,\rho_{-1},\dotsc,\rho_{-n})
        = \left( n - \tfrac{1}{2}, n - \tfrac{3}{2}, \dotsc, \tfrac{1}{2}, 0, -\tfrac{1}{2}, \dotsc, -n + \tfrac{1}{2} \right),
        \quad \text{for } N=2n+1.
    \end{gather*}
}

\begin{lem} \label{mink}
    For $N \ge 1$, the bilinear map
    \begin{equation} \label{Vform}
        \Phi_V \colon V \otimes V \to \triv,\quad \Phi_V(v_i \otimes v_j) = \delta_{i,-j} q^{-2 \delta_{j>0} \rho_j} (q+1)^{\delta_{j,0}}, \qquad
        i,j \in \Vset,
    \end{equation}
    is a homomorphism of $U_q(N)$-modules.
\end{lem}

\begin{proof}
    Since $V$ is simple and self-dual, there is a unique $U_q(N)$-module homomorphism $V \otimes V \to \triv$, up to scalar multiple.  It follows from \cref{hammock} that $\Phi_V(v_i \otimes v_j)$ is zero unless $i=-j$.

    First consider the case $N=2n$.  We normalize $\Phi_V$ so that $\Phi_V(v_1 \otimes v_{-1}) = 1$.  For $2 \le i \le n$, we compute
    \begin{multline*}
        \Phi_V(v_i, v_{-i})
        \overset{\cref{doze1}}{=} \Phi_V(e_i v_{i-1}, v_{-i})
        \overset{\cref{antipode}}{=} - \Phi_V(v_{i-1}, e_i k_i^{-1} v_{-i})
        \\
        \overset{\cref{astrophage}, \cref{simproots}}{\underset{\cref{hammock}, \cref{doze4}}{=}} -q \Phi_V(v_{i-1}, e_i v_{-i})
        \overset{\cref{doze1}}{=} \Phi_V(v_{i-1}, v_{1-i}).
    \end{multline*}
    Thus, \cref{Vform} holds for $1 \le i \le n$ by induction.  Next, we compute
    \begin{multline*}
        \Phi_V(v_{-1} \otimes v_1)
        \overset{\cref{doze3}}{=} \Phi_V(f_1 v_2, v_1)
        \overset{\cref{antipode}}{=} \Phi_V(v_2, - k_1 f_1 v_1)
        \\
        \overset{\cref{doze3}}{=} q \Phi_V(v_2, k_1 v_{-2})
        \overset{\cref{astrophage}, \cref{simproots}}{\underset{\cref{hammock}, \cref{doze4}}{=}} \Phi_V(v_2, v_{-2})
        = 1.
    \end{multline*}
    Then, since, for $2 \le i \le n$,
    \begin{multline*}
        \Phi_V(v_{-i}, v_i)
        \overset{\cref{doze1}}{=} - q^{-1} \Phi_V(f_i v_{1-i}, v_i)
        \overset{\cref{antipode}}{=} q^{-1} \Phi_V(v_{1-i}, k_i f_i v_i)
        \\
        \overset{\cref{doze1}}{=} q^{-1} \Phi_V(v_{1-i}, k_i v_{i-1})
        \overset{\cref{astrophage}, \cref{simproots}}{\underset{\cref{hammock}, \cref{doze4}}{=}} q^{-2} \Phi_V(v_{1-i},v_{i-1}),
    \end{multline*}
    we see that \cref{Vform} also holds for $-n \le i \le -1$ by induction.

    Now suppose $N=2n+1$.  This time we normalize the form so that $\Phi_V(v_0 \otimes v_0) = q+1$.  Then we have
    \begin{multline*}
        \Phi_V(v_1,v_{-1})
        \overset{\cref{doze2}}{=} (q+1)^{-1} \Phi_V(e_1 v_0, v_{-1})
        \overset{\cref{antipode}}{=} - (q+1)^{-1} \Phi_V(v_0, e_1 k_1^{-1} v_{-1})
        \\
        \overset{\cref{astrophage}, \cref{simproots}}{\underset{\cref{hammock}, \cref{doze4}}{=}} - q (q+1)^{-1} \Phi_V(v_0, e_1 v_{-1})
        \overset{\cref{doze2}}{=} (q+1)^{-1} \Phi_V(v_0, v_0)
        = 1
    \end{multline*}
    and
    \begin{multline*}
        \Phi_V(v_{-1},v_1)
        \overset{\cref{doze3}}{=} - q^{-\frac{1}{2}} (q+1)^{-1} \Phi_V(f_1 v_0, v_1)
        \overset{\cref{antipode}}{=} q^{-\frac{1}{2}} (q+1)^{-1} \Phi_V(v_0, k_1 f_1 v_1)
        \\
        \overset{\cref{doze3}}{=} q^{-1} (q+1)^{-1} \Phi_V(v_0, k_1 v_0)
        \overset{\cref{astrophage}, \cref{simproots}}{\underset{\cref{hammock}, \cref{doze4}}{=}}  q^{-1} (q+1)^{-1} \Phi_V(v_0, v_0)
        = q^{-1}
        \overset{\cref{rhoodd}}{=} q^{-2\rho_1}.
    \end{multline*}
    Then the result follows by induction as in the $N=2n$ case.
\end{proof}

\subsection{The braiding}

We recall some well-known facts about the braiding on the category of finite-dimensional $U_q(\fso(N))$-modules.  First suppose that $N \ge 3$.  Let $T_{W,W'} \colon W \otimes W' \to W' \otimes W$ denote the component of the braiding on modules $W$ and $W'$.  We can choose the braiding so that there exist $\Theta_\nu \in U_q(\fso(N))_\nu \otimes_\kk U_q(\fso(N))_{-\nu}$, for $\nu$ in the positive root lattice of $\fg$, such that $\Theta_0 = 1 \otimes 1$ and, for two diagonalizable highest-weight modules $W$ and $W'$ of $U_q(\fg)$,
\begin{equation} \label{hades}
    T_{W,W'} = \flip \circ D_{W,W'} \circ \left( 1 + \sum_\nu \varrho(\Theta_\nu) \right) \colon W \otimes W' \to W' \otimes W,
\end{equation}
where
\begin{itemize}
    \item $\flip$ is the tensor flip given by $w \otimes w' \mapsto w' \otimes w$,
    \item $D_{W,W'}$ acts as multiplication by $q^{(\lambda,\mu)}$ on $W_\lambda \otimes W'_\mu$, where $M_\lambda$ denotes the $\lambda$-weight space of a module $M$, and
    \item $\varrho(\Theta_\nu) := (\varrho_W \otimes \varrho_{W'})(\Theta_\nu)$, where $\varrho_M \colon U_q(\fso(N)) \to \End_\kk(M)$ is the representation associated to the module $M$.
\end{itemize}
See, for example, \cite[(2.3.3)]{DF94}.  Note that, since $W$ and $W'$ are finite dimensional, the sum $\sum_\nu \varrho(\Theta_\nu)$ is finite.

Recall, from \cref{subsec:Uq}, our definition of weights for $N \le 2$.  We define
\begin{equation} \label{styx}
    T_{W,W'} := \flip_{W,W'} D_{W,W'}
    \qquad \text{for } N \in \{0,1,2\}.
\end{equation}
In particular, this means that $T_{W,W'} = \flip_{W,W'}$ for $N \in \{0,1\}$.

While it is well known that $T_{W,W'}$ is an isomorphism of $U_q(\fso(N))$-modules, we need the following stronger statement.

\begin{prop} \label{dentist}
    For $W,W' \in U_q(N)\md$, the map $T_{W,W'}$ is an isomorphism of $U_q(N)$-modules.
\end{prop}

\begin{proof}
    For $N\in\{0,1,2\}$, the result follows immediately from \cref{styx}.  If $N\geq 3$ is odd, then $U_q(N)=U_q(\fso(N))$, and the result is the usual $U_q(\fso(N))$-linearity of the braiding.

    Suppose, therefore, that $N=2n\geq 4$, and let
    \[
        \vartheta
        :=
        \operatorname{Ad}(\sigma) \big|_{U_q(\fso(N))}
    \]
    be the corresponding diagram automorphism.  For a $U_q(\fso(N))$-module $M$, let ${}^{\vartheta}M$ denote the
    pullback of $M$ along $\vartheta$, so that $a\in U_q(\fso(N))$ acts on ${}^{\vartheta}M$ as $\vartheta(a)$ acts on $M$.   Pullback along $\vartheta$ is a braided tensor autoequivalence of $U_q(\fso(N))\md$; see \cite[Prop.~2.18]{KM23}.
    \details{
        The result is Proposition~2.15 in the arXiv version at \url{https://arxiv.org/abs/2107.12348v3}.
    }
    For $N=4$, where $D_2=A_1\times A_1$, the same statement follows factorwise, with $\vartheta$ interchanging the two factors.

    If $M$ is a $U_q(N)$-module, then the action of $\sigma$ defines a $U_q(\fso(N))$-module isomorphism
    \[
        \sigma_M \colon M \to {}^{\vartheta}M,
        \qquad
        m \mapsto \sigma m.
    \]
    \details{
        Indeed,
        \[
            \sigma_M(am)
            =
            \sigma a m
            =
            \vartheta(a)\sigma m
        \]
        for all $a\in U_q(\fso(N))$ and $m\in M$.
    }
    Naturality of the braiding therefore gives
    \[
        (\sigma_{W'} \otimes \sigma_W)T_{W,W'}
        =
        T_{{}^{\vartheta}W,{}^{\vartheta}W'} (\sigma_W\otimes\sigma_{W'}).
    \]
    Since pullback along $\vartheta$ is braided,
    \[
        T_{{}^{\vartheta}W,{}^{\vartheta}W'}
        =
        T_{W,W'}
    \]
    as underlying linear maps.  Consequently,
    \[
        (\sigma_{W'}\otimes\sigma_W)T_{W,W'}
        =
        T_{W,W'}(\sigma_W\otimes\sigma_{W'}).
    \]
    Thus, $T_{W,W'}$ intertwines the action of $\sigma$.  Since it is already an isomorphism of $U_q(\fso(N))$-modules, it is an isomorphism of $U_q(N)$-modules.
\end{proof}

The following lemma, which follows immediately from \cref{hades,styx}, will be useful in our computations to come.

\begin{lem}
    Suppose that $W$ and $W'$ are $U_q(N)$-modules, that $w$ is a highest-weight vector of $W$ of weight $\lambda$, and that $w'$ is a lowest-weight vector of $W'$ of weight $\lambda'$.  Then
    \begin{equation} \label{persephone}
        T_{W,W'}(w \otimes w') = q^{(\lambda,\lambda')} w' \otimes w.
    \end{equation}
\end{lem}

The only \emph{explicit} braiding we will need in the current paper is $T_{V,V}$, where $V$ is the natural module defined in \cref{subsec:natural}.  This explicit expression is well known; see, for example, \cite[\S 7.3\, (20)]{CP95}, \cite[\S8.4.2\, (61)]{KS97}, or \cite[Prop.~5.1.2]{DF94}.  In our chosen basis $\{v_i\}_{i \in \Vset}$, we have
\begin{multline} \label{TVV}
    T := T_{V,V}
    = q \sum_{\substack{i \in \Vset \\ i \ne 0}} E_{ii} \otimes E_{ii}
    + \delta_{N,2n+1} E_{0,0} \otimes E_{0,0}
    + \sum_{\substack{i,j \in \Vset \\ i \ne \pm j}} E_{ji} \otimes E_{ij}
    + q^{-1} \sum_{\substack{i \in \Vset \\ i \ne 0}} E_{-i,i} \otimes E_{i,-i}
    \\
    + (q-q^{-1}) \sum_{\substack{i,j \in \Vset \\ i<j}} E_{ii} \otimes E_{jj}
    - (q-q^{-1}) \sum_{\substack{i,j \in \Vset \\ i<j}} (q+1)^{\delta_{i,0}-\delta_{j,0}} E_{-j,i} \otimes E_{j,-i},
\end{multline}
where, as usual, the terms with indices $0$ only appear in the case $N=2n+1$.

\subsection{Tensor product decompositions}

We now note some tensor-product decompositions that will be important for us.  As noted at the end of \cref{subsec:Uq}, these are the same as the corresponding tensor-product decompositions in the non-quantum case.  We will therefore state the results without proof, referring the reader to \cite[\S4.4]{MS24}, which handles the non-quantum case.  Note, however, that the labelling of the Dynkin diagrams is different in \cite{MS24}; node $i$ in \cite{MS24} corresponds to node $n-i+1$ in the current paper.

\begin{lem} \label{SV}
    For $N \ge 1$, we have $\dim \Hom_{U_q(N)}(S \otimes V, S) = 1$.
\end{lem}

\begin{proof}
    When $N=1$, this follows from the fact that $V$ is the trivial module and $S$ is simple.  For $N \ge 2$, it follows from \cite[Cor.~4.11]{MS24}.
\end{proof}

\begin{prop}\label{sdub}
    Let $n \in \N$.  By convention, let $L_q(\epsilon_n + \epsilon_{n-1} + \dotsb + \epsilon_k)$ be the trivial $U_q(\fso(N))$-module $L_q(0)$ when $k=n+1$.
    \begin{enumerate}
        \item When $N=2n+1$ (type $B_n$), we have
        \begin{equation} \label{SdubB}
            S^{\otimes 2} \cong \bigoplus_{k=1}^{n+1} L_q(\epsilon_n + \epsilon_{n-1} + \dotsb + \epsilon_k)
            \qquad \text{as $U_q(N)$-modules}.
        \end{equation}

        \item When $N=2n$ (type $D_n$) with $n \ge 1$, we have
        \begin{equation} \label{SdubD}
            S^{\otimes 2} \cong \bigoplus_{k=1}^{n+1} \Ind(L_q(\epsilon_n + \epsilon_{n-1} + \dotsb + \epsilon_k))
            \qquad \text{as $U_q(N)$-modules}.
        \end{equation}
        When $N=0$, we have $S^{\otimes 2} \cong \triv^0$.
    \end{enumerate}
\end{prop}

In \cref{SdubB}, all summands on the right-hand side are irreducible.  For \cref{SdubD}, it follows from the discussion in \cref{subsec:modulesD} that the $k=1$ summand is irreducible, while the $k \ne 1$ summands are direct sums of two irreducible $U_q(N)$-modules.

\subsection{Quantum-dimension formulae}

We now give some evaluations of the quantum-dimension formula.  While standard, we include the proofs for completeness.  Recall the definition of $\rho$ from \cref{rho}.  The \emph{quantum dimension} of a $U_q(\fso(N))$-module $M$ is
\[
    \dim_q M = \sum_{\nu \in X} q^{\langle 2\rho,\nu \rangle} \dim M_\nu.
\]

\begin{lem}
    We have
    \[
        \dim_q V = [N-1] + 1.
    \]
\end{lem}

\begin{proof}
    In type $D_n$, the weights of $V$ are $\{\pm \epsilon_i :1 \leq i \leq n \}$, each with multiplicity one.  Thus, we have
    \[
        \dim_q V = \sum_{i=1}^n q^{2i-2} + \sum_{i=1}^n q^{2-2i}
        = 1 + q^{N-2} + q^{N-4} + \dotsb + q^{2-N}
        = 1 + [N-1].
    \]
    In type $B_n$ the weights of $V$ are $\{\pm \epsilon_i : 1\leq i \leq n \}\cup \{0 \}$, each with multiplicity one.  So, we have
    \[
        \dim_q V = \sum_{i=1}^n q^{2i-1} + \sum_{i=1}^n q^{1-2i} + 1
        = q^{N-2} + q^{N-4} + \dotsb + q^{2-N} + 1
        = 1 + [N-1].
        \qedhere
    \]
\end{proof}

\begin{lem}
    \begin{enumerate}
        \item If $N=2n+1 \ge 3$, then
            \begin{equation} \label{screen}
                \dim_q L_q(\epsilon_n + \dotsb + \epsilon_k)
                = \qbinom{N-1}{n-k+1} +\qbinom{N-1}{n-k},
            \end{equation}
            for $1 \le k \le n$.

        \item If $N=2n \ge 4$, then \cref{screen} holds for $2 \le k \le n$, while
            \[
                \dim_q L_q(\epsilon_n+\dotsb+\epsilon_1)
                = \qbinom{2n-1}{n}.
            \]
    \end{enumerate}
\end{lem}

\begin{proof}
    For $1 \le k \le n$, let $M_k$ denote the $U_q(\fso(N))$-module corresponding, under the equivalence of categories between $U_q(\fso(N))$-mod and $\fso(N)$-mod, to the exterior power $\Lambda^{n-k+1}(V)$.  If $N=2n+1$, or if $N=2n$ and $2 \le k \le n$, then
    \[
        M_k \cong L_q(\epsilon_n+\dotsb+\epsilon_k).
    \]
    If $N=2n$ and $k=1$, then
    \[
        M_1
        \cong
        L_q(\epsilon_n+\dotsb+\epsilon_1)
        \oplus
        L_q(\epsilon_n+\dotsb+\epsilon_2-\epsilon_1).
    \]
    See, for example, \cite[(4.31), (4.33), (4.34)]{MS24}, noting our different convention for labelling the simple roots.

    Recall from \cref{hammock} that the weights of $V$ are $\epsilon_i$, $i\in\Vset$, where we adopt the convention that $\epsilon_{-i}=-\epsilon_i$ for $1\le i\le n$.  Thus, the weights of $M_k$, counted according to multiplicity, are
    \[
        \sum_{i\in I}\epsilon_i,
        \qquad
        I\subseteq\Vset,\quad |I|=n-k+1.
    \]
    Therefore,
    \[
        \dim_q M_k
        =
        \sum_{\substack{I\subseteq\Vset\\|I|=n-k+1}}
        \prod_{i\in I}q^{2\rho_i}.
    \]
    By \cref{rhoeven,rhoodd}, the values $2\rho_i$, $i\in\Vset$, are
    \[
        0,N-2,N-4,\dotsc,2-N.
    \]
    Note that the value $0$ occurs twice when $N$ is even.  Thus, letting $A=\{N-2,N-4,\dotsc,2-N\}$, we have
    \[
        \dim_q M_k
        =
        \sum_{\substack{J\subseteq A\\|J|=n-k+1}}
        \prod_{j\in J}q^j
        +
        \sum_{\substack{J\subseteq A\\|J|=n-k}}
        \prod_{j\in J}q^j.
    \]
    For $l\in\N$, equating the coefficients of $x^l$ in \cref{bird}, we have
    \[
        \sum_{\substack{J\subseteq A\\|J|=l}}
        \prod_{j\in J}q^j
        =
        \qbinom{N-1}{l}.
    \]
    Consequently,
    \[
        \dim_q M_k
        =
        \qbinom{N-1}{n-k+1}
        +
        \qbinom{N-1}{n-k}.
    \]
    If $N=2n+1$, or if $N=2n$ and $2\le k\le n$, then $M_k\cong L_q(\epsilon_n+\dotsb+\epsilon_k)$, and hence this proves \cref{screen} in all the stated cases.

    It remains to consider $N=2n$ and $k=1$.  Set
    \[
        W=L_q(\epsilon_n+\dotsb+\epsilon_1).
    \]
    By \cref{thaw}, we have
    \[
        W^\sigma
        \cong
        L_q(\epsilon_n+\dotsb+\epsilon_2-\epsilon_1).
    \]
    Since $\sigma\rho=\rho$ by \cref{rho,thaw}, twisting by $\sigma$ preserves quantum dimension.  Indeed,
    \[
        \dim_q W^\sigma
        =
        \sum_{\nu\in X}
        q^{\langle2\rho,\sigma\nu\rangle}\dim W_\nu
        =
        \sum_{\nu\in X}
        q^{\langle2\rho,\nu\rangle}\dim W_\nu
        =
        \dim_q W.
    \]
    The two summands of $M_1$ therefore have the same quantum dimension.  Hence,
    \[
        2\dim_q L_q(\epsilon_n+\dotsb+\epsilon_1)
        =
        \dim_q M_1
        =
        \qbinom{2n-1}{n}
        +
        \qbinom{2n-1}{n-1}
        =
        2\qbinom{2n-1}{n},
    \]
    where the last equality follows from the symmetry of the quantum binomial coefficients.  Dividing by $2$ gives the result.
\end{proof}

\section{The quantum Clifford algebra\label{sec:qCliff}}

Our goal in this section is to introduce the quantum analogue of the spin module for the orthogonal groups.  This module is constructed using a quantum analogue of the Clifford algebra.  Throughout this section, we continue to assume that $\kk = \C \big( q^{\pm \frac{1}{4}} \big)$.

\subsection{Definition of the quantum Clifford algebra}

We now introduce the quantum Clifford algebra, following the approach of \cite{DF94}.

\begin{defin}[{\cite[Defs.~3.1.1, 3.2.1]{DF94}}]
    The \emph{quantum Clifford algebra} $\Cl_q = \Cl_q(N)$ is the quotient of the tensor algebra of $V$ by the relations
    \begin{equation} \label{chelsea}
        (\id + q T) (u \otimes v) = \Phi_V(u,v),
        \qquad u,v \in V,
    \end{equation}
    where $T$ is as in \cref{TVV}.  For $i \in \Vset$, we let $\psi_i$ denote the image of $v_i$ in $\Cl_q(N)$.  We also define
    \begin{equation} \label{hilt}
        \psi_i^\dagger = \psi_{-i},\qquad 0 \le i \le n.
    \end{equation}
\end{defin}

It follows from \cref{Vform} that, for $0 \le i,j \le n$,
\begin{equation} \label{ruby}
    \begin{gathered}
        \Phi_V(v_i, v_{-j}) = \delta_{ij} (q+1)^{\delta_{i,0}},\quad
        \Phi_V(v_{-i}, v_j) = \delta_{ij} q^{-2\rho_i} (q+1)^{\delta_{i,0}},
        \\
        \Phi_V(v_i,v_j) = 0 = \Phi_V(v_{-i}, v_{-j}).
    \end{gathered}
\end{equation}

The following corollary is closely related to \cite[Prop.~5.2.3]{DF94}.  However, since no proof is given in \cite{DF94}, and our conventions (in particular, our chosen basis of $V$) are different, we provide a proof.  Recall our convention that expressions involving the index $0$ only apply when $N=2n+1$.

\begin{prop} \label{apostille}
    The quantum Clifford algebra $\Cl_q(N)$ is isomorphic to the associative algebra with generators $\psi_i$, $\psi_i^\dagger$, $1 \le i \le n$, subject to the relations
    \begin{align} \label{apostille1}
        \psi_i \psi_j &= - q \psi_j \psi_i,& 0 \le j < i \le n, \\ \label{apostille2}
        \psi_i^\dagger \psi_j^\dagger &= -q^{-1} \psi_j^\dagger \psi_i^\dagger,& 0 \le j < i \le n, \\ \label{apostille3}
        \psi_i \psi_i &= 0 = \psi_i^\dagger \psi_i^\dagger,& 1 \le i \le n, \\ \label{apostille4}
        \psi_i \psi_j^\dagger &= - q \psi_j^\dagger \psi_i,& 0 \le i,j \le n,\ i \ne j, \\ \label{apostille5}
        \psi_i \psi_i^\dagger + \psi_i^\dagger \psi_i &= (q^2-1) \sum_{j>i} \psi_j^\dagger \psi_j + 1,& 1 \le i \le n, \\ \label{apostille6}
        \psi_0^2 &= (q^2-1) \sum_{j=1}^n \psi_j^\dagger \psi_j + 1.
    \end{align}
\end{prop}

\begin{proof}
    For $i,j \in \Vset$, $j < i$, $j \ne -i$, we have
    \[
        T(\psi_i \otimes \psi_j)
        \overset{\cref{TVV}}{=} \psi_j \otimes \psi_i.
    \]
    Hence, \cref{chelsea,Vform} gives
    \[
        \psi_i \psi_j + q \psi_j \psi_i = 0,
    \]
    proving \cref{apostille1}.  Using \cref{hilt}, it also proves \cref{apostille2,apostille4}.

    Next, for $i \in \Vset$, $i \ne 0$, we have
    \[
        T(\psi_i \otimes \psi_i)
        \overset{\cref{TVV}}{=} q \psi_i \otimes \psi_i,
    \]
    and so \cref{chelsea,Vform} give
    \[
        (1+q^2) \psi_i \psi_i = 0 \implies \psi_i \psi_i = 0,
    \]
    proving \cref{apostille3}.

    Finally, suppose $0 \le i \le n$.  Then we have
    \[
        T(v_i \otimes v_{-i})
        \overset{\cref{TVV}}{=} q^{-\delta_{i \ne 0}} v_{-i} \otimes v_i - (q-q^{-1}) (q+1)^{\delta_{i,0}} \sum_{j>i} v_{-j} \otimes v_j.
    \]
    Hence, \cref{chelsea,ruby} give
    \[
        \psi_i \psi_i^\dagger + q^{\delta_{i,0}} \psi_i^\dagger \psi_i
        = (q^2-1) (q+1)^{\delta_{i,0}} \sum_{j> i} \psi_j^\dagger \psi_j + (q+1)^{\delta_{i,0}},
    \]
    proving \cref{apostille5,apostille6}.
\end{proof}

Recall that the classical Clifford algebra $\Cl(2n)$ is the associative $\kk$-algebra with generators $\Psi_i, \Psi_{-i} = \Psi_i^\dagger$, $1 \le i \le n$, subject to the relations
\begin{equation} \label{classicCl}
    \Psi_i \Psi_j + \Psi_j \Psi_i = 0 = \Psi_i^\dagger \Psi_j^\dagger + \Psi_j^\dagger \Psi_i^\dagger,\quad
    \Psi_i \Psi_j^\dagger + \Psi_j^\dagger \Psi_i = \delta_{ij},\qquad
    1 \le i,j \le n.
\end{equation}
The classical Clifford algebra $\Cl(2n+1)$ is the associative $\kk$-algebra with generators $\Psi_i, \Psi_i^\dagger$, $0 \le i \le n$, subject to the relations \cref{classicCl} and
\begin{equation} \label{mesh}
    \Psi_0 = \Psi_0^\dagger,\quad
    \Psi_0 \Psi_i + \Psi_i \Psi_0 = 0 = \Psi_0 \Psi_i^\dagger + \Psi_i^\dagger \Psi_0,\quad
    \Psi_0^2 = 1,\qquad
    1 \le i \le n.
\end{equation}

Define
\begin{equation} \label{omega}
    \omega_i := \Psi_i \Psi_i^\dagger + q^{-1} \Psi_i^\dagger \Psi_i,\qquad
    1 \le i \le n.
\end{equation}
These elements are invertible and, for $1 \le i,j \le n$, we have
\begin{gather} \label{notary1}
    \omega_i \Psi_i = \Psi_i,\quad
    \Psi_i \omega_i = q^{-1} \Psi_i,\quad
    \Psi_i^\dagger \omega_i = \Psi_i^\dagger,\quad
    \omega_i \Psi_i^\dagger = q^{-1} \Psi_i^\dagger,\quad
    \omega_i \omega_j = \omega_j \omega_i,
    \\ \label{notary2}
    \omega_j \Psi_i \omega_j^{-1} = q^{\delta_{ij}} \Psi_i,\quad
    \omega_j \Psi_i^\dagger \omega_j^{-1} = q^{-\delta_{ij}} \Psi_i^\dagger,
    \\ \label{notary3}
    \omega_i^k = \Psi_i \Psi_i^\dagger + q^{-k} \Psi_i^\dagger \Psi_i
    = 1 + (q^{-k}-1) \Psi_i^\dagger \Psi_i,\qquad k \in \Z.
\end{gather}
For $0 \le i \le n$, define
\begin{equation}
    \omega_{>i} := \prod_{j>i} \omega_j,\quad
    \omega_{\ge i} := \prod_{j \ge i} \omega_j,
    \quad \text{so that} \quad
    \omega_{>i}^{-1} = \prod_{j>i} \omega_j^{-1},\quad
    \omega_{\ge i}^{-1} = \prod_{j \ge i} \omega_j^{-1}.
\end{equation}

\begin{prop} \label{Clisom}
    There is an isomorphism of $\kk$-algebras $\Cl_q(N) \xrightarrow{\cong} \Cl(N)$ given by
    \begin{equation} \label{moca} \textstyle
        \psi_i \mapsto q^{\rho_i} \omega_{>i}^{-1} \Psi_i,\qquad
        \psi_i^\dagger \mapsto q^{-\rho_i} \omega_{>i}^{-1} \Psi_i^\dagger,\qquad
        0 \le i \le n.
    \end{equation}
\end{prop}

\begin{proof}
    Let $\theta$ be the given map.  We must verify that $\theta$ respects the relations of \cref{apostille}.  For $0 \le j < i \le n$,
    \[
        \theta(\psi_i) \theta(\psi_j)
        = q^{\rho_i+\rho_j} \omega_{>i}^{-1} \Psi_i \omega_{>j}^{-1} \Psi_j
        \overset{\cref{classicCl}}{\underset{\cref{notary2}}{=}} -q^{\rho_i+\rho_j+1} \omega_{>j}^{-1} \Psi_j \omega_{>i}^{-1} \Psi_i
        = -q \theta(\psi_j) \theta(\psi_i)
    \]
    and
    \[
        \theta(\psi_i^\dagger) \theta(\psi_j^\dagger)
        = q^{-\rho_i-\rho_j} \omega_{>i}^{-1} \Psi_i^\dagger \omega_{>j}^{-1} \Psi_j^\dagger
        \overset{\cref{classicCl}}{\underset{\cref{notary2}}{=}} - q^{-\rho_i-\rho_j-1} \omega_{>j}^{-1} \Psi_j^\dagger \omega_{>i}^{-1} \Psi_i^\dagger
        = -q^{-1} \theta(\psi_j^\dagger) \theta(\psi_i^\dagger).
    \]
    It is straightforward to verify that $\theta(\psi_i) \theta(\psi_i) = 0 = \theta(\psi_i^\dagger) \theta(\psi_i^\dagger)$ for all $1 \le i \le n$.

    For $0 \le i,j \le n$, $i \ne j$, we have
    \[
        \theta(\psi_i) \theta(\psi_j^\dagger)
        = q^{\rho_i-\rho_j} \omega_{>i}^{-1} \Psi_i \omega_{>j}^{-1} \Psi_j^\dagger
        \overset{\cref{classicCl}}{\underset{\cref{notary2}}{=}} -q^{\rho_i-\rho_j+1} \omega_{>j}^{-1} \Psi_j^\dagger \omega_{>i}^{-1} \Psi_i
        = - q \theta(\psi_j^\dagger) \theta(\psi_i).
    \]

    Next, for $1 \le i \le n$, we have
    \begin{multline*}
        \theta(\psi_i) \theta(\psi_i^\dagger)
        = \omega_{>i}^{-1} \Psi_i \omega_{>i}^{-1} \Psi_i^\dagger
        \overset{\cref{notary2}}{=} \omega_{>i}^{-2} \Psi_i \Psi_i^\dagger
        \overset{\cref{classicCl}}{=} \omega_{>i}^{-2} (1 - \Psi_i^\dagger \Psi_i)
        \\
        \overset{\cref{notary2}}{=} - \omega_{>i}^{-1} \Psi_i^\dagger \omega_{>i}^{-1} \Psi_i + \omega_{>i}^{-2}
        = -\theta(\psi_i^\dagger) \theta(\psi_i) + \omega_{>i}^{-2}.
    \end{multline*}
    Thus,
    \begin{align*}
        \theta(\psi_i) \theta(\psi_i^\dagger) + \theta(\psi_i^\dagger) \theta(\psi_i)
        &= \omega_{>i}^{-2} \\
        &\overset{\mathclap{\cref{notary3}}}{=}\ \left( 1 + (q^2 - 1) \Psi_{i+1}^\dagger \Psi_{i+1} \right) \omega_{>i+1}^{-2} \\
        &= (q^2-1) \theta(\psi_{i+1}^\dagger) \theta(\psi_{i+1}) + \omega_{>i+1}^{-2} \\
        &= (q^2-1) \Big( \theta(\psi_{i+1}^\dagger) \theta(\psi_{i+1}) + \theta(\psi_{i+2}^\dagger) \theta(\psi_{i+2}) \Big) + \omega_{>i+2}^{-2} \\
        &\ \vdots \\
        &= (q^2-1) \sum_{j>i} \theta(\psi_j^\dagger) \theta(\psi_j) + 1.
    \end{align*}
    A similar argument shows that
    \[
        \theta(\psi_0)^2
        = \omega_{>0}^{-2}
        = (q^2-1) \sum_{j=1}^n \theta(\psi_j^\dagger) \theta(\psi_j) + 1.
        \qedhere
    \]
\end{proof}

\Cref{Clisom} is closely related to \cite[Prop.~5.3.1]{DF94}.  The difference is the factors of $q^{\pm \rho_i}$ in \cref{moca}, which do not appear in \cite{DF94}.  We include these factors since they are needed for \cref{whitman}.  From now on, we will use \cref{Clisom} to identify $\Cl_q(N)$ and $\Cl(N)$.

\begin{rem}
    Quantum Clifford algebras have appeared in different places in the literature.  We have followed the approach of \cite{DF94} in our definition of $\Cl_q(2n)$.  \Cref{apostille,Clisom}, which are motivated by \cite[\S5.3]{DF94}, provide the link to the alternate approach of \cite[\S2.1]{Hay90}.  Precisely, $\Cl_q(2n)$ is isomorphic to the quotient of the algebra $\mathcal{A}_q^+(n)$ of \cite[\S2.1]{Hay90} by any of the following sets of relations (where we denote the generators $\psi_i$ and $\psi_i^\dagger$ of \cite{Hay90} by $\Psi_i$ and $\Psi_i^\dagger$, respectively):
    \begin{itemize}
        \item $\Psi_i \Psi_i^\dagger + \Psi_i^\dagger \Psi_i = 1$ for all $1 \le i \le n$,
        \item $\omega_i \Psi_i = \Psi_i$ for all $1 \le i \le n$,
        \item $\Psi_i^\dagger \omega_i = \Psi_i^\dagger$ for all $1 \le i \le n$.
    \end{itemize}
    These extra relations also appear in \cite[Def.~B.1]{BER24}; see \cite[Rem.~B.2]{BER24}.
    \details{
        \cite[Def.~B.1]{BER24} is equivalent to \cite[(5.3.1)--(5.3.4)]{DF94} with $q$ replaced everywhere by $q^2$.  This replacement corresponds to the fact that in the homomorphism of \cite[Th.~3.2(B)]{Hay90}, the codomain is $\mathcal{A}_{q^2}^+$, which is the quantum Clifford algebra with $q$ replaced by $q^2$.  In \cite{BER24}, they only consider type $B$, so they just build this $q^2$ into their definition of the quantum Clifford algebra.
    }
\end{rem}

\subsection{The spin module for the quantum Clifford algebra}

Suppose $n \ge 1$, and let
\[
    S := \Lambda(W) = \bigoplus_{r=0}^n \Lambda^r(W),
    \qquad \text{where }
    W = \Span_\kk \{\Psi_i^\dagger : 1 \le i \le n\}.
\]
As a $\kk$-module, $S$ has basis
\begin{equation} \label{xIdef}
    \begin{gathered}
        x_I := \Psi_{i_1}^\dagger \wedge \Psi_{i_2}^\dagger \wedge \dotsb \wedge \Psi_{i_k}^\dagger,\\
        I = \{i_1,\dotsc,i_k\} \subseteq \{1,2,\dotsc,n\},\quad i_1 > i_2 > \dotsc > i_k,\
        0 \le k \le n.
    \end{gathered}
\end{equation}
In particular,
\begin{equation} \label{paneer}
    \dim_\kk(S) = 2^n.
\end{equation}

The classical Clifford algebra $\Cl(2n)$ acts naturally on $S$ via wedging and contracting operators.  (The precise action is given by taking $q=1$ and replacing $\psi$ by $\Psi$ in \cref{gouda1,gouda2} below.)  Using \cref{omega}, we see that
\begin{equation} \label{drwho}
    \omega_i x_I =
    \begin{cases}
        q^{-1} x_I & \text{if } i \in I, \\
        x_I & \text{if } i \notin I.
    \end{cases}
\end{equation}
Via the isomorphism of \cref{Clisom}, $S$ has the structure of a $\Cl_q(2n)$-module by defining, for $I \subseteq \{1,2,\dotsc,n\}$, $1 \le i \le n$,
\begin{align} \label{gouda1}
    \psi_i^\dagger x_I
    &=
    \begin{cases}
        0 & \text{if } i \in I, \\
        q^{-\rho_i} (-q)^{|\{ j \in I : j>i \}|} x_{I \cup \{i\}} & \text{if } i \notin I,
    \end{cases}
    \\ \label{gouda2}
    \psi_i x_I
    &=
    \begin{cases}
        q^{\rho_i} (-q)^{|\{ j \in I : j>i \}|} x_{I \setminus \{i\}} & \text{if } i \in I, \\
        0 & \text{if } i \notin I.
    \end{cases}
\end{align}
We can define two $\Cl_q(2n+1)$-module structures on $S$, depending on a choice of $\varepsilon \in \{\pm 1\}$.  We again use the action defined in \cref{gouda1,gouda2}, and additionally define
\begin{equation} \label{diablo}
    \Psi_0 x_I = \varepsilon (-1)^{|I|} x_I,
    \quad \text{so that} \quad
    \psi_0 x_I = \varepsilon (-q)^{|I|} x_I.
\end{equation}
In both cases, $N=2n$ or $N=2n+1$, we call $S$ the \emph{spin module} for $\Cl_q(N)$.

\subsection{The spin module for $U_q(N)$}

We now introduce the quantum analogue of the spin module for the quantized enveloping algebra, which is one of the fundamental ingredients in our main result.

\begin{prop} \label{ukulele}
    \begin{enumerate}
        \item If $N=2n \ge 4$, then there is a $\kk$-algebra homomorphism $U_q(2n) \to \Cl_q(2n)$ given by
            \begin{gather*}
                e_i \mapsto \Psi_i \Psi_{i-1}^\dagger,\quad
                f_i \mapsto \Psi_{i-1} \Psi_i^\dagger,\quad
                k_i \mapsto \omega_i \omega_{i-1}^{-1},\qquad
                2 \le i \le n,
                \\
                e_1 \mapsto \Psi_2 \Psi_1,\quad
                f_1 \mapsto \Psi_1^\dagger \Psi_2^\dagger,\quad
                k_1 \mapsto q \omega_2 \omega_1,\quad
                \sigma \mapsto \sqrt{-1} \left( \Psi_1 - \Psi_1^\dagger \right).
            \end{gather*}

        \item If $N=2n+1 \ge 3$, then there is a $\kk$-algebra homomorphism $U_q(2n+1) \to \Cl_q(2n+1)$ given by
            \begin{gather*}
                e_i \mapsto \Psi_i \Psi_{i-1}^\dagger,\quad
                f_i \mapsto \Psi_{i-1} \Psi_i^\dagger,\quad
                1 \le i \le n,
                \\
                k_1 \mapsto q^{\frac{1}{2}} \omega_1,\quad
                k_i \mapsto \omega_i \omega_{i-1}^{-1},\qquad
                2 \le i \le n.
            \end{gather*}
    \end{enumerate}
\end{prop}

\begin{proof}
    This is proved in \cite[Th.~3.2]{Hay90}, except that the labelling of the Dynkin diagram is different there, the element $\sigma$ is not included in the definition of the quantized enveloping algebra there, and \cite[Th.~3.2]{Hay90} does not have the factors of $\Psi_0 = \Psi_0^\dagger$ that appear in the images of $e_1$ and $f_1$ above when $N=2n+1$.
    \details{
        Adding the factors of $\Psi_0$ does not affect the result.  For example, if $N=2n+1$, we have
        \[
            e_1 f_1 - f_1 e_1
            \mapsto \Psi_1 \Psi_1^\dagger - \Psi_1^\dagger \Psi_1
        \]
        and
        \[
            k_1 - k_1^{-1}
            \mapsto q^{\frac{1}{2}} \omega_1 - q^{-\frac{1}{2}} \omega_1^{-1}
            \overset{\cref{notary3}}{=} \left( q^{\frac{1}{2}} - q^{-\frac{1}{2}}\right) (\Psi_1 \Psi_1^\dagger - \Psi_1^\dagger \Psi_1).
        \]
        Since $q_1=q^{\frac{1}{2}}$, we see that \cref{Udef-ef} is satisfied for $i=j=1$.
    }
    The relations \cref{sigma} are straightforward to verify.  For example,
    \[
        -(\Psi_1 - \Psi_1^\dagger)(\Psi_2 \Psi_1) (\Psi_1 - \Psi_1^\dagger)
        = - \Psi_1^\dagger \Psi_2 \Psi_1 \Psi_1^\dagger
        = \Psi_2 \Psi_1^\dagger (1 - \Psi_1^\dagger \Psi_1)
        = \Psi_2 \Psi_1^\dagger,
    \]
    showing that the second relation in \cref{sigma} with $i=1$ is satisfied.  We also have
    \begin{multline*}
        -(\Psi_1 - \Psi_1^\dagger)(q \omega_2 \omega_1)(\Psi_1 - \Psi_1^\dagger)
        = q (\Psi_1^\dagger - \Psi_1)(\Psi_2 \Psi_2^\dagger + q^{-1} \Psi_2^\dagger \Psi_2)(\Psi_1 \Psi_1^\dagger + q^{-1} \Psi_1^\dagger \Psi_1) (\Psi_1 - \Psi_1^\dagger)
        \\
        = q (\Psi_1^\dagger - \Psi_1)(\Psi_2 \Psi_2^\dagger + q^{-1} \Psi_2^\dagger \Psi_2)(\Psi_1 - q^{-1} \Psi_1^\dagger)
        = q(\Psi_2 \Psi_2^\dagger + q^{-1} \Psi_2^\dagger \Psi_2)(\Psi_1^\dagger - \Psi_1)(\Psi_1 - q^{-1} \Psi_1^\dagger)
        \\
        = (\Psi_2 \Psi_2^\dagger + q^{-1} \Psi_2^\dagger \Psi_2)(\Psi_1 \Psi_1^\dagger + q \Psi_1^\dagger \Psi_1)
        = \omega_2 \omega_1^{-1},
    \end{multline*}
    showing that the last relation in \cref{sigma} for the positive exponent is satisfied.  Verification of the remaining relations in \cref{sigma} is similar.
\end{proof}

\Cref{ukulele} allows us to view the spin module $S$ as a module for $U_q(N)$.  We continue to refer to this $U_q(N)$-module as the \emph{spin module}.  It is straightforward to see that $S$ is a self-dual simple module.  For $S$ to be simple, the inclusion of the element $\sigma$ in the definition of $U_q(2n)$ is crucial.  When $N=2n+1$, $S$ has highest-weight vector $x_\varnothing$ of weight $\frac{1}{2}(\epsilon_1 + \epsilon_2 + \dotsb + \epsilon_n)$.  When $N=2n$, we have $S \cong \Ind(L_q(\pm \epsilon_1 + \epsilon_2 + \dotsb + \epsilon_n))$.

When $N \le 2$, the spin module is given as follows (see \cite[Rem.~4.1]{MS24}):
\begin{itemize}
    \item When $N=0$, then $S$ is the one-dimensional module for $U_q(0) = \kk[\sigma]/(\sigma^2-1)$ where $\sigma$ acts by $-1$.  We fix a nonzero element $x_\varnothing \in S$.

    \item When $N=1$, then $S$ is the one-dimensional module for $U_q(1) = \kk[\xi]/(\xi^2-1)$ where $\xi$ acts by $-1$.  We fix a nonzero element $x_\varnothing \in S$.

    \item When $N=2$, the module $S$ is $2$-dimensional, with basis $x_\varnothing,x_{\{1\}}$, and action given by
        \[
            k x_\varnothing = q^{\frac{1}{2}} x_\varnothing,\quad
            k x_{\{1\}} = q^{- \frac{1}{2}} x_{\{1\}},\quad
            \sigma x_\varnothing = - \sqrt{-1} x_{\{1\}},\quad
            \sigma x_{\{1\}} = \sqrt{-1} x_\varnothing.
        \]
\end{itemize}

\subsection{Quantum Clifford multiplication}

The inclusion $V \hookrightarrow \Cl(N)$, together with the action of $\Cl(N)$ on $S$ given by \cref{gouda1,gouda2,diablo} yields a map
\begin{equation} \label{tau}
    \tau \colon V \otimes S \to S,\quad
    v_i \otimes x \mapsto \psi_i x,\qquad i \in \Vset.
\end{equation}
We refer to $\tau$ as \emph{quantum Clifford multiplication}.

\begin{prop} \label{whitman}
    The map $\tau$ is a homomorphism of $U_q(N)$-modules.
\end{prop}

The proof of \cref{whitman}, which is a lengthy but straightforward direct computation, is given in \cref{sec:rope}.

\subsection{Invariant bilinear form}

For a subset $I$ of $\{1,2,\dotsc,n\}$, we let $I^\complement = \{1,2,\dotsc,n\} \setminus I$ denote its complement.  Define a bilinear form on $S$ by
\begin{equation} \label{Sform}
    \Phi_S(x_I, x_J)
    = \delta_{I,J^\complement} (-1)^{n(N+1)|I|} \prod_{i \in I} (-1)^i q^{-\rho_i},
\end{equation}
and extending by bilinearity.  When $N=2n+1$, this is the form from \cite[Example~4.16]{BER24}, except that their $q$ is our $q^{\frac{1}{2}}$ and we label the Dynkin diagram differently, so one needs to replace $i$ by $n+1-i$.

\begin{lem} \label{bouncy}
    For $0 \le i \le n$, $x,y \in S$, we have
    \begin{equation} \label{bounce}
        \Phi_S(\Psi_i^\dagger x, y) = (-1)^{Nn} q^{-\rho_i} \Phi_S(x, \Psi_i^\dagger y),\qquad
        \Phi_S(\Psi_i x, y) = (-1)^{Nn} q^{\rho_i} \Phi_S(x, \Psi_i y).
    \end{equation}
\end{lem}

\begin{proof}
    It suffices to consider the case where $x = x_I$, $y = x_J$, $i \notin I$, $J = I^\complement \setminus \{i\}$.  In this case, for $1 \le i \le n$, we have
    \[
        \Phi_S(\Psi_i^\dagger x_I, x_J)
        = (-1)^{|\{j \in I : j>i\}|} \Phi_S(x_{I \cup \{i\}},x_J)
        = (-1)^{|\{j \in I : j>i\}| + n(N+1)(|I|+1)} \prod_{j \in I \cup \{i\}} (-1)^j q^{-\rho_j}
    \]
    and
    \[
        \Phi_S(x_I, \Psi_i^\dagger x_J)
        = (-1)^{|\{j \notin I : j>i\}|} \Phi_S(x_I, x_{I^\complement})
        = (-1)^{|\{j \notin I : j>i\}| + n(N+1)|I|}  \prod_{j \in I} (-1)^j q^{-\rho_j}.
    \]
    The first equality in \cref{bounce} now follows from the fact that
    \[
        |\{ j \in I : j > i \}| + |\{ j \notin I : j > i\}| = n-i.
    \]
    The proof of the second equality is similar.  The proof of both equalities in \cref{bounce} for $i=0$ are straightforward.
\end{proof}

\begin{rem}
    It follows from \cref{bouncy} and \cite[Lem.~4.6]{MS24}, together with the fact that $\Phi_S(x_\varnothing,x_{\varnothing^\complement})=1$, that setting $q=1$ in \cref{Sform} recovers the bilinear form \cite[(4.16)]{MS24}.
\end{rem}

\begin{cor}
    The bilinear form $\Phi_S$ induces a homomorphism $S \otimes S \to \triv$ of $U_q(N)$-modules.
\end{cor}

\begin{proof}
    For $N \le 2$, this is a straightforward direct verification.  Thus, we suppose $N \ge 3$.  It suffices to show that
    \[
        \Phi_S(a x,y) = \Phi_S(x, \iota(a)y)
        \qquad \text{for all } x,y \in S,\ a \in \{e_i, f_i, k_i^{\pm 1}, \sigma : i \in I\}.
    \]
    For $2 \le i \le n$, we have
    \[
        \Phi_S(e_i x, y)
        = \Phi_S(\Psi_i \Psi_{i-1}^\dagger x, y)
        \overset{\cref{bounce}}{=} q^{\rho_i - \rho_{i-1}} \Phi_S(x, \Psi_{i-1}^\dagger \Psi_i y)
        \overset{\cref{rhoeven}}{\underset{\cref{rhoodd}}{=}} - q \Phi_S(x, \Psi_i \Psi_{i-1}^\dagger y).
    \]
    Now, since $\Psi_i \Psi_{i-1}^\dagger x_I = 0$ unless $i \in I$ and $i-1 \notin I$, we have
    \[
        q \Psi_i \Psi_{i-1}^\dagger y
        = \Psi_i \Psi_{i-1}^\dagger \omega_i^{-1} \omega_{i-1} y
        = e_i k_i^{-1} y,
    \]
    as desired.  The proofs that $\Phi_S(f_i x,y) = -\Phi_S(x,k_i f_i y)$ and $\Phi_S(k_i^{\pm 1}x,y) = \Phi_S(x,k_i^{\mp 1}y)$ are similar.

    Now suppose that $N=2n$.  Then
    \[
        \Phi_S(e_1 x, y)
        = \Phi_S(\Psi_2 \Psi_1 x, y)
        \overset{\cref{bounce}}{=} q^{\rho_1+\rho_2} \Phi_S(x, \Psi_1 \Psi_2 y)
        \overset{\cref{rhoeven}}{=} -q \Phi_S(x, \Psi_2 \Psi_1 y).
    \]
    Since $\Psi_2 \Psi_1 x_I = 0$ unless $1,2 \in I$, we have
    \[
        q \Psi_2 \Psi_1 y
        = q^{-1} \Psi_2 \Psi_1 \omega_2^{-1} \omega_1^{-1} y
        = e_1 k_1^{-1} y,
    \]
    as desired.  The proof that $\Phi_S(f_1x,y) = -\Phi_S(x,k_1f_1y)$ is similar.  We also have
    \[
        \Phi_S(\sigma x, y)
        = \sqrt{-1} \Phi_S(\Psi_1 x,y) - \sqrt{-1} \Phi_S(\Psi_1^\dagger x,y)
        \overset{\cref{bounce}}{\underset{\cref{rhoeven}}{=}} \sqrt{-1} \Phi_S(x,\Psi_1 y) - \sqrt{-1} \Phi_S(x,\Psi_1^\dagger y)
        = \Phi_S(x, \sigma y).
    \]

    Now suppose that $N=2n+1$.  Then, for $I,J \subseteq \{1,2,\dotsc,n\}$, we have
    \[
        \Phi_S(e_1 x_I, x_J)
        = \Phi_S(\Psi_1 \Psi_0 x_I, x_J)
        \overset{\cref{bounce}}{=} q^{\rho_0 + \rho_1} \Phi_S(x_I, \Psi_0 \Psi_1 x_J)
        \overset{\cref{rhoodd}}{=} - q^{\frac{1}{2}} \Phi_S(x_I, \Psi_1 \Psi_0 x_J).
    \]
    It suffices to consider the case $1 \in I$, $J = I^\complement \cup \{1\}$, since otherwise all the above expressions are equal to zero.  In this case, we have
    \[
        q^{\frac{1}{2}} \Psi_1 \Psi_0 x_J
        = q^{-\frac{1}{2}} \Psi_1 \Psi_0 \omega_1^{-1} x_J
        = e_1 k_1^{-1} x_J,
    \]
    as desired.
\end{proof}

\section{The incarnation functor\label{sec:incarnation}}

In this section, we relate the quantum spin Brauer category to the representation theory of $U_q(N)$ by defining a functor from $\QSB$ to $U_q(N)\md$. In \cref{sec:full}, we will show that this functor is full.  In \cref{sec:essential}, we will show that the functor induced on $\overline{\QSB}$ is essentially surjective after passing to the additive Karoubi envelope.

Throughout this section, we assume $\kk = \C \big( q^{\pm \frac{1}{4}} \big)$.  Let
\begin{equation} \label{squiddy}
    \sigma_N = (-1)^{\binom{n}{2}+nN}.
\end{equation}
We fix the parameters
\begin{equation} \label{crabby}
    t = q^{N(1-N)/8},\qquad
    \kappa = (-1)^{nN} q^{(1-N)/2},\qquad
    d_\Sgo = \sigma_N \prod_{i=1}^n \left( q^{\frac{N}{2}-i} + q^{i-\frac{N}{2}} \right),
\end{equation}
and set
\begin{equation}
    \QSB(N) := \QSB(q,t,\kappa,d_\Sgo),\qquad
    \overline{\QSB}(N) := \overline{\QSB}(q,t,\kappa,d_\Sgo).
\end{equation}
By \cref{dimrel}, we have
\begin{equation} \label{shimmy}
    d_\Vgo = \frac{\kappa^{-2} - \kappa^2}{q-q^{-1}} + 1 = [N-1] + 1.
\end{equation}

Fix a basis $\bB_S$ of $S$, and let $\bB_S^\vee = \{x^\vee : x \in \bB_S\}$ denote the left dual basis with respect to $\Phi_S$, from \cref{Sform}, defined by
\[
    \Phi_S(x^\vee, y) = \delta_{x,y},\qquad x,y \in \bB_S.
\]
We fix a basis $\bB_V$ of $V$ and define the left dual basis $\bB_V^\vee = \{v^\vee : v \in \bB_V\}$ similarly, using $\Phi_V$ from \cref{Vform}.  Then we have $U_q(N)$-module homomorphisms
\begin{align}
    \Phi_S^\vee &\colon \kk \to S \otimes S,&
    \lambda &\mapsto \lambda \sum_{x \in \bB_S} x \otimes x^\vee,\quad
    \lambda \in \kk,
    \\
    \Phi_V^\vee &\colon \kk \to V \otimes V,&
    \lambda &\mapsto \lambda \sum_{v \in \bB_V} v \otimes v^\vee,\quad
    \lambda \in \kk.
\end{align}
These are independent of the choices of bases.

It follows from \cref{Sform} that the left dual to the basis $x_I$, $I \subseteq \{1,2,\dotsc,n\}$, of $S$ is given by
\begin{equation} \label{hercules}
    x_I^\vee = (-1)^{n(N+1)|I^\complement|} \prod_{i \notin I} (-1)^i q^{\rho_i} x_{I^\complement},
\end{equation}
and hence
\begin{equation} \label{dubstepx}
    (x_I^\vee)^\vee = \sigma_N \left( \prod_{i \in I} q^{\rho_i} \right) \left( \prod_{i \notin I} q^{-\rho_i} \right) x_I.
\end{equation}
Similarly, it follows from \cref{Vform} that the left dual to the basis $v_i$, $i \in \Vset$, of $V$ is given by
\begin{equation} \label{vdual}
    v_i^\vee = q^{2 \delta_{i>0} \rho_i} (q+1)^{-\delta_{i,0}} v_{-i},
\end{equation}
and hence
\begin{equation} \label{dubstepv}
    (v_i^\vee)^\vee = q^{-2\rho_i} v_i.
\end{equation}
\details{
    We have
    \[
        \Phi_V \left( q^{-2\rho_i} v_i, v_i^\vee \right)
        = q^{-2\rho_i} q^{2\delta_{i>0} \rho_i} (q+1)^{-\delta_{i,0}} \Phi_V(v_i, v_{-i})
        \overset{\cref{Vform}}{=} q^{-2\rho_i} q^{2\delta_{i>0} \rho_i} q^{2 \delta_{i<0} \rho_i}
        = 1.
    \]
}

Recall the map $\tau \colon V \otimes S \to S$ from \cref{tau}.
\begin{theo} \label{incarnation}
    There is a unique $\kk$-linear monoidal functor
    \[
        \bF \colon \QSB(N) \to U_q(N)\md
    \]
    given on objects by $\Sgo \mapsto S$, $\Vgo \mapsto V$, and on morphisms by
    \begin{gather} \label{incarnate1}
        \capmor{spin} \mapsto \Phi_S,\qquad
        \capmor{vec} \mapsto (q^{2-N}+1)^{-1} \Phi_V,\qquad
        \mergemor{vec}{spin}{spin} \mapsto \tau,
        \\ \label{incarnate2}
        \poscross{spin}{spin} \mapsto \sigma_N T_{S,S},\qquad
        \poscross{spin}{vec} \mapsto T_{S,V},\qquad
        \poscross{vec}{spin} \mapsto T_{V,S},\qquad
        \poscross{vec}{vec} \mapsto T_{V,V}.
    \end{gather}
    Furthermore, we have
    \begin{gather} \label{incarnate3}
        \negcross{spin}{spin} \mapsto \sigma_N T_{S,S}^{-1},\qquad
        \negcross{spin}{vec} \mapsto T_{V,S}^{-1},\qquad
        \negcross{vec}{spin} \mapsto T_{S,V}^{-1},\qquad
        \negcross{vec}{vec} \mapsto T_{V,V}^{-1},
        \\ \label{incarnate4}
        \cupmor{spin} \mapsto \Phi_S^\vee,\qquad
        \cupmor{vec} \mapsto (q^{2-N}+1) \Phi_V^\vee.
    \end{gather}
\end{theo}

We call $\bF$ the \emph{incarnation functor}.

\begin{proof}
    We first show that \cref{incarnate1,incarnate2,incarnate3,incarnate4} indeed yield a functor $\bF$.  We must show that $\bF$ respects the relations of \cref{QSBdef}.  When $N=0$, all diagrams involving a blue strand are zero and the verification is straightforward.  Thus, we assume $N \ge 1$.

    It is well known that the $T_{S,S}$, $T_{S,V}$, $T_{V,S}$, and $T_{V,V}$ yield a braiding on the full monoidal subcategory of $U_q(\fso(N))$-mod generated by $V$ and $S$.  (See, for example, \cite[Ch.~32]{Lus10}.)  By \cref{dentist}, they also yield a braiding on the corresponding subcategory of $U_q(N)$-mod.  The relations \cref{typhoon} and the first three relations in \cref{braid} follow immediately.  The final two relations in \cref{braid} also follow from this braided structure.  Indeed, it follows from the property of a braided monoidal category that
    \[
        \bF \left(
            \begin{tikzpicture}[anchorbase]
                \draw (-0.3,-0.4) to[out=up,in=down] (0,0) arc(180:0:0.15) to[out=down,in=up] (0,-0.4);
                \draw[wipe] (0.3,-0.4) to[out=up,in=down] (-0.3,0.4);
                \draw (0.3,-0.4) to[out=up,in=down] (-0.3,0.4);
            \end{tikzpicture}
        \right)
        =
        \bF \left(
            \begin{tikzpicture}[anchorbase]
                \draw (-0.3,-0.4) -- (-0.3,-0.2) arc(180:0:0.15) -- (0,-0.4);
                \draw (0.3,-0.4) -- (0.3,0.4);
            \end{tikzpicture}
        \right).
    \]
    Composing on the bottom with $\bF \left( \idstrand{spin}\ \poscross{spin}{spin} \right)$ then shows that the fourth equality in \cref{braid} holds for black strands.  The verification for other colours of the strands, as well as the last equality in \cref{braid}, is analogous.

    The verification of \cref{yank} is standard.  The verification of \cref{skein} is also standard, dating back to the study of the surjection of the Birman--Murakami--Wenzl algebra onto the endomorphism algebra of $V^{\otimes 2}$.  See, for example, \cite[Th.~10.2.5]{CP95}.

    Next, we consider the relations \cref{deloop}.  Since $S$ is a simple self-dual module, we have
    \[
        \dim \Hom_{U_q(N)}(S^{\otimes 2}, \triv) = 1.
    \]
    Thus,
    \[
        \bF
        \left(
            \begin{tikzpicture}[anchorbase]
                \draw[spin] (0.15,0) arc(0:180:0.15) to[out=-90,in=120] (0.15,-0.4);
                \draw[wipe] (-0.15,-0.4) to[out=60,in=-90] (0.15,0);
                \draw[spin] (-0.15,-0.4) to[out=60,in=-90] (0.15,0);
            \end{tikzpicture}
        \right)
        = \sigma_N \Phi_S \circ T_{S,S}
        = c_S \Phi_S
        = \bF \left( \capmor{spin} \right)
    \]
    for some $c_S \in \kk$.  Then we compute
    \[
        \Phi_S \circ T_{S,S}(x_\varnothing, x_{\varnothing^\complement})
        \overset{\cref{persephone}}{=} q^{-\frac{n}{4}} \Phi_S(x_{\varnothing^\complement}, x_\varnothing)
        \overset{\cref{Sform}}{=} (-1)^{n(N+1)} q^{-\frac{n}{4}} \prod_{i=1}^n (-1)^i q^{-\rho_i}
        = \sigma_N q^{-\frac{n}{4} - \sum_{i=1}^n \rho_i}.
    \]
    When $N=2n$, we have
    \[
        -\frac{n}{4} - \sum_{i=1}^n \rho_i
        \overset{\cref{rhoeven}}{=} -\frac{n}{4} - \sum_{i=1}^n (i-1)
        = -\frac{n}{4} - \frac{n(n-1)}{2}
        = - \frac{n(2n-1)}{4}
        = \frac{N(1-N)}{8}.
    \]
    On the other hand, when $N=2n+1$, we have
    \[
        -\frac{n}{4} - \sum_{i=1}^n \rho_i
        \overset{\cref{rhoodd}}{=} -\frac{n}{4} - \sum_{i=1}^n \left( i - \frac{1}{2} \right)
        = - \frac{n}{4} - \frac{n^2}{2}
        = \frac{N(1-N)}{8}.
    \]
    Since $\Phi_S(x_{\varnothing}, x_{\varnothing^\complement}) = 1$, it follows that $c_S = q^{\frac{N(1-N)}{8}} = t$, verifying that $\bF$ respects the first relation in \cref{deloop}.  Similarly, since $V$ is a simple self-dual module, we have
    \[
        \bF
        \left(
            \begin{tikzpicture}[anchorbase]
                \draw[vec] (0.15,0) arc(0:180:0.15) to[out=-90,in=120] (0.15,-0.4);
                \draw[wipe] (-0.15,-0.4) to[out=60,in=-90] (0.15,0);
                \draw[vec] (-0.15,-0.4) to[out=60,in=-90] (0.15,0);
            \end{tikzpicture}
        \right)
        = \Phi_V \circ T_{V,V}
        = c_V \Phi_V
        = \bF \left( \capmor{vec} \right)
    \]
    for some $c_V \in \kk$.  Then we compute
    \[
        \Phi_V \circ T_{V,V} (v_n \otimes v_{-n})
        \overset{\cref{TVV}}{=} q^{-1} \Phi_V \left( v_{-n} \otimes v_n \right)
        \overset{\cref{Vform}}{=} q^{-2\rho_n-1} \Phi_V(v_n \otimes v_{-n}).
    \]
    It follows from \cref{rhoeven,rhoodd} that $c_V=q^{-2\rho_n-1}=q^{1-N}=\kappa^2$, verifying that $\bF$ respects the second relation in \cref{deloop}.

    Now consider the relation \cref{swishy}.  The image under $\bF$ of the left-hand side is the $U_q(N)$-module homomorphism $S \to S \otimes V$ given by
    \begin{multline*}
        x_J
        \mapsto \sum_{\substack{v \in \bB_V \\ x \in \bB_S}} \Phi_S \big( x_J, \tau(v \otimes x) \big) x^\vee \otimes v^\vee
        \overset{\cref{dubstepx}}{\underset{\cref{dubstepv}}{=}} \sum_{\substack{i \in \Vset \\ I \subseteq \{1,2,\dotsc,n\}}} \Phi_S \big( x_J, \tau(v_i^\vee \otimes x_I^\vee) \big) (x_I^\vee)^\vee \otimes (v_i^\vee)^\vee
        \\
        \overset{\cref{dubstepx}}{\underset{\cref{dubstepv}}{=}} \sigma_N \sum_{\substack{i \in \Vset \\ I \subseteq \{1,2,\dotsc,n\}}} \left( \prod_{j \in I} q^{\rho_j} \right) \left( \prod_{j \notin I} q^{-\rho_j} \right) q^{-2\rho_i} \Phi_S \big( x_J, \tau(v_i^\vee \otimes x_I^\vee) \big) x_I \otimes v_i.
    \end{multline*}
    Now, if $1 \le i \le n$, then $\Phi_S \big( x_J, \tau(v_i^\vee \otimes x_I^\vee) \big) = 0$ unless $i \in I$ and $J = I \setminus \{i\}$, and in this case
    \begin{equation} \label{wipers}
        \left( \prod_{j \in I} q^{\rho_j} \right) \left( \prod_{j \notin I} q^{-\rho_j} \right) q^{-2\rho_i}
        = \left( \prod_{j \in J} q^{\rho_j} \right) \left( \prod_{j \notin J} q^{-\rho_j} \right).
    \end{equation}
    Similarly, if $-n \le i \le -1$, then $\Phi_S \big( x_J, \tau(v_i^\vee \otimes x_I^\vee) \big)=0$ unless $-i \notin I$  and $J = I \cup \{-i\}$, and in this case \cref{wipers} again holds, noting that $\rho_{-i} = -\rho_i$.  The equality \cref{wipers} also clearly holds when $i=0$.  Thus, it follows from \cref{dubstepx} that the image under $\bF$ of the left-hand side of \cref{swishy} is given by
    \[
        x_J \mapsto \sum_{\substack{i \in \Vset \\ I \subseteq \{1,2,\dotsc,n\}}} \Phi_S \left( (x_J^\vee)^\vee, \tau(v_i^\vee \otimes x_I^\vee) \right) x_I \otimes v_i
        = \sum_{\substack{i \in \Vset \\ I \subseteq \{1,2,\dotsc,n\}}} \Phi_S \left( \tau(v_i^\vee \otimes x_I^\vee), x_J \right) x_I \otimes v_i,
    \]
    which is precisely the image under $\bF$ of the right-hand side of \cref{swishy}.
    \details{
        For the equality in the centered equation above, we use the fact that
        \[
            \Phi_S \left( (x_J^\vee)^\vee, y \right) = \Phi_S (y, x_J)
            \quad \text{for all } J \subseteq \{1,2,\dotsc,n\},\ y \in S.
        \]
        This fact is clear when $y = x_I^\vee$, $I \subseteq \{1,2,\dotsc,n\}$, and then the general case follows by linearity.
    }

    Next, we prove \cref{fishy}.  Since, by \cref{SV}, $\dim \Hom_{U_q(N)}(S \otimes V, S)=1$, we have
    \begin{equation} \label{dragon}
        \bF
        \left(
            \begin{tikzpicture}[anchorbase]
                \draw[vec] (0.2,-0.5) to [out=135,in=down] (-0.15,-0.2) to[out=up,in=-135] (0,0);
                \draw[wipe] (-0.2,-0.5) to[out=45,in=down] (0.15,-0.2);
                \draw[spin] (-0.2,-0.5) to[out=45,in=down] (0.15,-0.2) to[out=up,in=-45] (0,0) -- (0,0.2);
            \end{tikzpicture}
        \right)
        = c \bF
        \left(
            \begin{tikzpicture}[anchorbase]
                \draw[spin] (0,0) -- (0,0.1) arc (0:180:0.2) -- (-0.4,-0.4);
                \draw[spin] (0,0) to[out=-45, in=left] (0.2,-0.2) arc(-90:0:0.15) -- (0.35,0.4);
                \draw[vec] (0,0) to[out=-135,in=up] (-0.2,-0.4);
            \end{tikzpicture}
        \right)
    \end{equation}
    for some $c \in \kk$.  We determine $c$ by computing the image of $x_\varnothing \otimes v_{-n}$ under both sides of \cref{dragon}.  For the left-hand side, we have, using \cref{persephone},
    \[
        x_\varnothing \otimes v_{-n}
        \xmapsto{T_{S,V}} q^{-\frac{1}{2}} v_{-n} \otimes x_\varnothing
        \xmapsto{\tau} q^{-\frac{1}{2}} \psi_n^\dagger x_\varnothing
        \overset{\cref{gouda1}}{=} q^{-\rho_n-\frac{1}{2}} x_{\{n\}}.
    \]
    For the right-hand side, we have
    \[
        x_\varnothing \otimes v_{-n}
        \xmapsto{1_S \otimes 1_V \otimes \Phi_S^\vee} \sum_I x_\varnothing \otimes v_{-n} \otimes x_I \otimes x_I^\vee
        \xmapsto{1_S \otimes \tau \otimes 1_S} \sum_I x_\varnothing \otimes \psi_n^\dagger x_I \otimes x_I^\vee
        \xmapsto{\Phi_S \otimes 1_S} \sum_I \Phi_S(x_\varnothing, \psi_n^\dagger x_I) x_I^\vee.
    \]
    Since $\Phi_S(x_\varnothing, \psi_n^\dagger x_I) = 0$ unless $I = \{n\}^\complement$, we have
    \[
        x_\varnothing \otimes v_{-n}
        \mapsto \Phi_S \left( x_\varnothing, \psi_n^\dagger x_{\{n\}^\complement} \right) x_{\{n\}^\complement}^\vee
        \overset{\cref{gouda1}}{=} \Phi_S \left( x_\varnothing, x_{\varnothing^\complement} \right) x_{\{n\}^\complement}^\vee
        \overset{\cref{Sform}}{\underset{\cref{hercules}}{=}} (-1)^{nN} x_{\{n\}}.
    \]
    Thus, $c = (-1)^{nN} q^{-\rho_n-\frac{1}{2}} \overset{\cref{rhoeven}}{\underset{\cref{rhoodd}}{=}} (-1)^{nN} q^{(1-N)/2} = \kappa$, verifying that $\bF$ respects \cref{fishy}.

    The fact that $\bF$ respects \cref{oist} follows immediately from the definition \cref{tau} of $\tau$ and the relation \cref{chelsea} defining the quantum Clifford algebra.

    Finally, we consider the relations \cref{dimrel}.  We have
    \[
        \bF \left( \bubble{spin} \right)
        = \sum_I \Phi_S(x_I \otimes x_I^\vee)
        \overset{\cref{dubstepx}}{=} \sum_I \sigma_N \left(\prod_{i \in I} q^{-\rho_i}\right)   \left(\prod_{i \notin I} q^{\rho_i} \right)
        = \sigma_N \prod_{i=1}^n (q^{\rho_i}+q^{-\rho_i}) = d_\Sgo
    \]
    and
    \[
        \bF \left( \bubble{vec} \right)
        = \sum_{i \in \Vset} \Phi_V(v_i \otimes v_i^\vee)
        \overset{\cref{dubstepv}}{=} \sum_{i \in \Vset} q^{2\rho_i}
        \overset{\cref{rhoeven}}{\underset{\cref{rhoodd}}{=}} [N-1] + 1 = d_\Vgo.
    \]

    This completes the proof of the existence part of the theorem.  The uniqueness follows from the fact that the images of $\negcross{spin}{spin}$, $\negcross{spin}{vec}$, $\negcross{vec}{spin}$, and $\negcross{vec}{vec}$ must be the inverses of the images of $\poscross{spin}{spin}$, $\poscross{vec}{spin}$, $\poscross{spin}{vec}$, and $\poscross{vec}{vec}$, respectively, while the images of the cups must be as in \cref{incarnate4} by the same argument given at the end of the proof of \cite[Th.~6.1]{MS24}.
\end{proof}

\begin{rem}
    The images under the incarnation functor of the relations \cref{zombie} are
    \begin{gather*}
        \begin{tikzpicture}[centerzero]
            \draw[vec] (-0.35,0.35) -- (0.35,-0.35);
            \draw[wipe] (-0.35,-0.35) -- (0.35,0.35);
            \draw[spin] (-0.35,-0.35) -- (0.35,0.35);
        \end{tikzpicture}
        = \frac{(-1)^{Nn}}{q^{\frac{N}{2}-1}+q^{1-\frac{N}{2}}}
        \left( q^{-1/2}\,
            \begin{tikzpicture}[centerzero]
                \draw[spin] (0.35,0.35) -- (0,0.15) -- (0,-0.15) -- (-0.35,-0.35);
                \draw[vec] (0,0.15) -- (-0.35,0.35);
                \draw[vec] (0,-0.15) -- (0.35,-0.35);
            \end{tikzpicture}
            + q^{1/2}\,
            \begin{tikzpicture}[centerzero]
                \draw[spin] (0.35,0.35) -- (0.15,0) -- (-0.15,0) -- (-0.35,-0.35);
                \draw[vec] (0.35,-0.35) -- (0.15,0);
                \draw[vec] (-0.35,0.35) -- (-0.15,0);
            \end{tikzpicture}
        \, \right)
        ,
        \\
        \begin{tikzpicture}[centerzero]
            \draw[spin] (-0.35,0.35) -- (0.35,-0.35);
            \draw[wipe] (-0.35,-0.35) -- (0.35,0.35);
            \draw[vec] (-0.35,-0.35) -- (0.35,0.35);
        \end{tikzpicture}
        = \frac{(-1)^{Nn}}{q^{\frac{N}{2}-1}+q^{1-\frac{N}{2}}}
        \left( q^{-1/2}\,
            \begin{tikzpicture}[centerzero]
                \draw[spin] (-0.35,0.35) -- (0,0.15) -- (0,-0.15) -- (0.35,-0.35);
                \draw[vec] (0,0.15) -- (0.35,0.35);
                \draw[vec] (0,-0.15) -- (-0.35,-0.35);
            \end{tikzpicture}
            + q^{1/2}\,
            \begin{tikzpicture}[centerzero]
                \draw[spin] (-0.35,0.35) -- (-0.15,0) -- (0.15,0) -- (0.35,-0.35);
                \draw[vec] (-0.35,-0.35) -- (-0.15,0);
                \draw[vec] (0.35,0.35) -- (0.15,0);
            \end{tikzpicture}
        \, \right)
        .
    \end{gather*}
    At $q=1$, this recovers \cite[(5.14)]{MS24}.  The image of the skein relation \cref{ash} is
    \[
        \begin{tikzpicture}[centerzero]
            \draw[vec] (-0.35,0.35) -- (0.35,-0.35);
            \draw[wipe] (-0.35,-0.35) -- (0.35,0.35);
            \draw[spin] (-0.35,-0.35) -- (0.35,0.35);
        \end{tikzpicture}
        -
        \begin{tikzpicture}[centerzero]
            \draw[spin] (-0.35,-0.35) -- (0.35,0.35);
            \draw[wipe] (-0.35,0.35) -- (0.35,-0.35);
            \draw[vec] (-0.35,0.35) -- (0.35,-0.35);
        \end{tikzpicture}
        = (-1)^{Nn} \frac{q^{1/2}-q^{-1/2}}{q^{\frac{N}{2}-1}+q^{1-\frac{N}{2}}}
        \left(\,
            \begin{tikzpicture}[centerzero]
                \draw[spin] (0.35,0.35) -- (0.15,0) -- (-0.15,0) -- (-0.35,-0.35);
                \draw[vec] (0.35,-0.35) -- (0.15,0);
                \draw[vec] (-0.35,0.35) -- (-0.15,0);
            \end{tikzpicture}
            -
            \begin{tikzpicture}[centerzero]
                \draw[spin] (0.35,0.35) -- (0,0.15) -- (0,-0.15) -- (-0.35,-0.35);
                \draw[vec] (0,0.15) -- (-0.35,0.35);
                \draw[vec] (0,-0.15) -- (0.35,-0.35);
            \end{tikzpicture}
        \, \right).
    \]
\end{rem}

\begin{prop}
    The incarnation functor $\bF$ factors through $\overline{\QSB}(N)$.
\end{prop}

\begin{proof}
    If $N$ is even, there is nothing to prove, since $\overline{\QSB}(N) = \QSB(N)$ in this case.  Therefore, we suppose that $N$ is odd.  It follows from \cref{absorbcr,absorbcup,asymidem} that the image under $\bF$ of the quantum antisymmetrizer in \cref{grove} is an idempotent whose image is isomorphic to the $r$-th quantum exterior power $\Lambda_q^r(V)$; see \cite[\S6.1 and Lem.~7.6]{TW05}.  When $r=N$, we have
    \begin{equation} \label{tacos}
        \Lambda_q^N V \cong L_q(0) \cong \triv,
    \end{equation}
    where the first isomorphism is \cite[Lem.~3.53]{BW23}.  Thus, there exists a scalar $a \in \kk$ such that
    \begin{equation} \label{wings}
        \bF
        \left(
            \begin{tikzpicture}[centerzero]
                \draw[spin] (0,0.25) circle(0.15);
                \draw[spin] (0,-0.25) circle(0.15);
                \draw[vec] (0,0.4) -- (0,1.1);
                \draw[vec] (0,-0.4) -- (0,-1.1);
                \altbox{-0.2,-0.9}{0.2,-0.6}{N};
                \altbox{-0.2,0.6}{0.2,0.9}{N};
            \end{tikzpicture}
        \right)
        = a \bF
        \left(
            \begin{tikzpicture}[centerzero]
                \draw[vec] (0,-1.1) -- (0,1.1);
                \altbox{-0.2,-0.15}{0.2,0.15}{N};
            \end{tikzpicture}
        \right)
        \ .
    \end{equation}
    We must show that $a$ is equal to the coefficient on the right side of \cref{extra}.

    For the remainder of the proof, we will suppress the functor $\bF$, and all string diagrams will denote their image under $\bF$.  By \cref{tacos} there exist
    \[
        \begin{tikzpicture}[centerzero]
            \draw[vec] (0,-0.4) -- (0,0.4);
            \altbox{-0.2,-0.1}{0.2,0.2}{N};
            \filldraw[vec] (0,-0.4) circle (1.5pt);
        \end{tikzpicture}
        \colon \triv \to \Lambda_q^N(V)
        \qquad \text{and} \qquad
        \begin{tikzpicture}[centerzero]
            \draw[vec] (0,-0.4) -- (0,0.4);
            \altbox{-0.2,0.1}{0.2,-0.2}{N};
            \filldraw[vec] (0,0.4) circle (1.5pt);
        \end{tikzpicture}
        \colon \Lambda_q^N(V) \to \triv
    \]
    such that
    \begin{equation} \label{poke}
        \begin{tikzpicture}[centerzero]
            \draw[vec] (0,0.1) -- (0,0.8);
            \draw[vec] (0,-0.1) -- (0,-0.8);
            \altbox{-0.2,0.3}{0.2,0.6}{N};
            \altbox{-0.2,-0.3}{0.2,-0.6}{N};
            \filldraw[vec] (0,0.1) circle (1.5pt);
            \filldraw[vec] (0,-0.1) circle (1.5pt);
        \end{tikzpicture}
        =
        \begin{tikzpicture}[centerzero]
            \draw[vec] (0,-0.8) -- (0,0.8);
            \altbox{-0.2,-0.15}{0.2,0.15}{N};
        \end{tikzpicture}
        \ .
    \end{equation}
    Since $S$ is an absolutely simple module, Schur's lemma implies that there exist $b,c \in \kk$ such that
    \begin{equation} \label{acai}
        \begin{tikzpicture}[centerzero]
            \draw[spin] (0.6,-0.8) -- (0.6,0.8);
            \draw[vec] (0,-0.5) -- (0,0.1) to[out=up,in=200] (0.6,0.5);
            \altbox{-0.2,-0.2}{0.2,0.1}{N};
            \filldraw[vec] (0,-0.5) circle (1.5pt);
        \end{tikzpicture}
        = b\
        \begin{tikzpicture}[centerzero]
            \draw[spin] (0.6,-0.8) -- (0.6,0.8);
        \end{tikzpicture}
        \qquad \text{and} \qquad
        \begin{tikzpicture}[centerzero,rotate=180]
            \draw[spin] (0.6,-0.8) -- (0.6,0.8);
            \draw[vec] (0,-0.5) -- (0,0.1) to[out=up,in=200] (0.6,0.5);
            \altbox{-0.2,-0.2}{0.2,0.1}{N};
            \filldraw[vec] (0,-0.5) circle (1.5pt);
        \end{tikzpicture}
        = c\
        \begin{tikzpicture}[centerzero]
            \draw[spin] (0.6,-0.8) -- (0.6,0.8);
        \end{tikzpicture}
        \ .
    \end{equation}
    We then have
    \[
        a
        \overset{\cref{monkey}}{\underset{\cref{monkey2}}{=}}
        a\,
        \begin{tikzpicture}[anchorbase]
            \draw[vec] (0,-0.15) arc(180:360:0.2) -- (0.4,0.15) arc(0:180:0.2);
            \altbox{-0.2,-0.15}{0.2,0.15}{N};
        \end{tikzpicture}
        \overset{\cref{wings}}{=}
        \begin{tikzpicture}[centerzero]
            \draw[spin] (0,0.25) circle(0.15);
            \draw[spin] (0,-0.25) circle(0.15);
            \draw[vec] (0,0.4) -- (0,0.9) arc (180:0:0.2) -- (0.4,-0.9) arc (360:180:0.2) -- (0,-0.4);
            \altbox{-0.2,-0.9}{0.2,-0.6}{N};
            \altbox{-0.2,0.6}{0.2,0.9}{N};
        \end{tikzpicture}
        \,
        \overset{\cref{asymidem}}{=}
        \begin{tikzpicture}[anchorbase]
            \draw[spin] (0,0) arc(-90:270:0.15);
            \draw[spin] (0,-1) arc(90:-270:0.15);
            \draw[vec] (0,0) -- (0,-1);
            \altbox{-0.2,-0.65}{0.2,-0.35}{N};
        \end{tikzpicture}
        \overset{\cref{poke}}{=}
        \begin{tikzpicture}[centerzero]
            \draw[vec] (0,0.1) -- (0,0.8);
            \draw[vec] (0,-0.1) -- (0,-0.8);
            \draw[spin] (0,0.8) arc(-90:270:0.15);
            \draw[spin] (0,-0.8) arc(90:-270:0.15);
            \altbox{-0.2,0.3}{0.2,0.6}{N};
            \altbox{-0.2,-0.3}{0.2,-0.6}{N};
            \filldraw[vec] (0,0.1) circle (1.5pt);
            \filldraw[vec] (0,-0.1) circle (1.5pt);
        \end{tikzpicture}
        \overset{\cref{acai}}{\underset{\cref{dimrel}}{=}} bc d_\Sgo^2.
    \]
    Similarly,
    \[
        \begin{tikzpicture}[anchorbase]
            \draw[spin] (0,0) arc(-90:270:0.15);
            \draw[vec] (0,0) -- (0,-0.65) arc(180:360:0.2) -- (0.4,0.3) arc(0:180:0.2);
            \altbox{-0.2,-0.65}{0.2,-0.35}{N};
        \end{tikzpicture}
        \overset{\cref{asymbend}}{\underset{\cref{poke}}{=}}
        \begin{tikzpicture}[centerzero]
            \draw[vec] (0,0.1) -- (0,0.8);
            \draw[vec] (0,-0.1) -- (0,-0.8);
            \draw[spin] (-0.2,1) arc(180:360:0.2) arc(0:180:0.6) -- (-1,-1) arc(180:360:0.6) arc(0:180:0.2) arc(360:180:0.2) -- (-0.6,1) arc(180:0:0.2);
            \altbox{-0.2,0.3}{0.2,0.6}{N};
            \altbox{-0.2,-0.3}{0.2,-0.6}{N};
            \filldraw[vec] (0,0.1) circle (1.5pt);
            \filldraw[vec] (0,-0.1) circle (1.5pt);
        \end{tikzpicture}
        \overset{\cref{acai}}{\underset{\cref{dimrel}}{=}} bc d_\Sgo.
    \]
    Thus, the result follows from \cref{smoothie}.
\end{proof}

We henceforth also write
\[
    \bF \colon \overline{\QSB}(N)\to U_q(N)\md
\]
for the induced functor.

\section{Fullness of the incarnation functor\label{sec:full}}

In the current section, we prove that the incarnation functor of \cref{incarnation} is full.  Throughout this section, we assume that $\kk = \C(q^{\pm \frac{1}{4}})$ and that $\sigma_N$, $t$, $\kappa$, and $d_\Sgo$ are given by \cref{squiddy,crabby}.   Recall the convention that $\psi_{-i} = \psi_i^\dagger$ and $\Psi_{-i} = \Psi_i^\dagger$ for $0 \le i \le n$.

\begin{lem}
    We have
    \begin{equation} \label{tauflip}
        \bF\left(\mergemor{spin}{vec}{spin}\right) \colon x_I \otimes v_i
        \mapsto (-1)^{Nn} q^{|\{ j \notin I : j>|i| \}|} \Psi_i x_I,
        \qquad i \in \Vset,\ I \subseteq \{1,2,\dotsc,n\}.
    \end{equation}
\end{lem}

\begin{proof}
    We compute that
    \[
        \bF \left( \mergemor{spin}{vec}{spin} \right)
        =
        \bF
        \left(
            \begin{tikzpicture}[anchorbase]
                \draw[spin] (0,0) -- (0,0.1) arc (0:180:0.2) -- (-0.4,-0.4);
                \draw[spin] (0,0) to[out=-45, in=left] (0.2,-0.2) arc(-90:0:0.15) -- (0.35,0.4);
                \draw[vec] (0,0) to[out=-135,in=up] (-0.2,-0.4);
            \end{tikzpicture}
        \right)
    \]
    is the map
    \begin{multline*}
        x_I \otimes v_i
        \mapsto \sum_J x_I \otimes v_i \otimes x_J \otimes x_J^\vee
        \mapsto \sum_J x_I \otimes \psi_i x_J \otimes x_J^\vee
        \mapsto \sum_J \Phi_S(x_I, \psi_i x_J) x_J^\vee
        \\
        \overset{\cref{moca}}{\underset{\cref{drwho}}{=}} \sum_J q^{\rho_i + |\{j \in J : j>|i|\}|} \Phi_S(x_I,\Psi_i x_J) x_J^\vee
        \overset{\cref{bounce}}{=} (-1)^{Nn} q^{|\{j \notin I : j>|i|\}|} \sum_J \Phi_S(\Psi_i x_I,x_J) x_J^\vee
        \\
        = (-1)^{Nn} q^{|\{j \notin I : j>|i|\}|} \Psi_i x_I,
    \end{multline*}
    where, in the second equality, we used the fact that $\Phi_S(x_I,\Psi_i x_J)=0$ unless $I = J^\complement \sqcup \{i\}$.
\end{proof}

\begin{lem} \label{rhino}
    We have
    \begin{equation} \label{hippo}
        \begin{gathered}
            \bF
            \left(
                \begin{tikzpicture}[centerzero]
                    \draw[spin] (-0.2,-0.3) -- (-0.2,0.3);
                    \draw[spin] (0.2,-0.3) -- (0.2,0.3);
                    \draw[vec] (-0.2,0) -- (0.2,0);
                \end{tikzpicture}
            \right)
            \colon x_I \otimes x_J
            \mapsto (-1)^{Nn} \left( q^{\frac{N-2}{2}} + q^{\frac{2-N}{2}} \right) B_N(x_I \otimes x_J),
            \qquad \text{where}
            \\
            B_N = \sum_{i \in \Vset} (q+1)^{-\delta_{i,0}} \omega_{>|i|} \Psi_i \otimes \omega_{>|i|}^{-1} \Psi_{-i} \in \Cl_q(N) \otimes \Cl_q(N).
        \end{gathered}
    \end{equation}
\end{lem}

\begin{proof}
    Using \cref{tau,tauflip}, we compute that
    \[
        \bF
        \left(
            \begin{tikzpicture}[centerzero]
                \draw[spin] (-0.2,-0.3) -- (-0.2,0.3);
                \draw[spin] (0.2,-0.3) -- (0.2,0.3);
                \draw[vec] (-0.2,0) -- (0.2,0);
            \end{tikzpicture}
        \right)
        = \bF
        \left(
            \begin{tikzpicture}[centerzero]
                \draw[spin] (-0.2,-0.3) -- (-0.2,0.3);
                \draw[spin] (0.2,-0.3) -- (0.2,0.3);
                \draw[vec] (-0.2,0) to[out=-45,in=-135,looseness=1.8] (0.2,0);
            \end{tikzpicture}
        \right)
    \]
    is the map
    \begin{align*}
        x_I \otimes x_J
        &\xmapsto{\cref{vdual}} (q^{2-N}+1) \sum_{i \in \Vset} q^{2 \delta_{i>0} \rho_i} (q+1)^{-\delta_{i,0}} x_I \otimes v_i \otimes v_{-i} \otimes x_J
        \\
        &\mapsto (-1)^{Nn} (q^{2-N}+1) \sum_{i \in \Vset} q^{2 \delta_{i>0} \rho_i + |\{j \notin I : j > |i|\}|}  (q+1)^{-\delta_{i,0}} \Psi_i x_I \otimes \psi_{-i} x_J
        \\
        &\overset{\mathclap{\cref{moca}}}{=}\ \ (-1)^{Nn} (q^{2-N}+1) \sum_{i \in \Vset} q^{\rho_{|i|} + |\{j \notin I : j > |i|\}|} (q+1)^{-\delta_{i,0}} \Psi_i x_I \otimes \omega_{>|i|}^{-1} \Psi_{-i} x_J
        \\
        &\overset{\mathclap{\cref{drwho}}}{=}\ \ (-1)^{Nn} (q^{2-N}+1) \sum_{i \in \Vset} q^{\rho_{|i|} + n - |i|} (q+1)^{-\delta_{i,0}} \omega_{>|i|} \Psi_i x_I \otimes \omega_{>|i|}^{-1} \Psi_{-i} x_J
        \\
        &\overset{\mathclap{\cref{rhoeven}}}{\underset{\mathclap{\cref{rhoodd}}}{=}}\ \ (-1)^{Nn} \left( q^{\frac{N-2}{2}} + q^{\frac{2-N}{2}} \right) \sum_{i \in \Vset} (q+1)^{-\delta_{i,0}} \omega_{>|i|} \Psi_i x_I \otimes \omega_{>|i|}^{-1} \Psi_{-i} x_J.
        \qedhere
    \end{align*}
\end{proof}

The element $B_N$ is denoted $C$ in \cite[(3.6), (3.9)]{Wen20}, although our conventions differ from those of \cite{Wen20}.  See \cref{iquantum}.  Note that
\begin{equation} \label{BlowN}
    B_0 = 0,\qquad
    B_1 = (q+1)^{-1},\qquad
    B_2 = \Psi_1 \otimes \Psi_1^\dagger + \Psi_1^\dagger \otimes \Psi_1.
\end{equation}

For $N\geq 2$, let
\[
    S_1 := \Span_\kk \{ x_I : n \in I \subseteq \{1,2,\dotsc,n\} \} \subseteq S.
\]
When $N \ge 5$, we identify $U_q(N-2)$ with the subalgebra of $U_q(N)$ generated by $e_i$, $f_i$, $k_i^{\pm 1}$, $1 \le i \le n-1$, and, if $N=2n$, by $\sigma$.  We can then restrict $U_q(N)$-modules to $U_q(N-2)$-modules.  For $N \le 4$, we define the restriction as follows:
\begin{itemize}
    \item For a $U_q(4)$-module, both $k_1$ and $k_2$ act diagonally with eigenvalues in $q^\Z$.  Thus, we can define the restriction to $U_q(2)$ by letting $k \in U_q(2)$ act by $(k_1 k_2^{-1})^{\frac{1}{2}}$, which acts diagonally with eigenvalues in $q^{\frac{1}{2} \Z}$.  The action of $\sigma \in U_q(2)$ is the same as the action of $\sigma \in U_q(4)$.

    \item For a $U_q(3)$-module, we define the restriction to $U_q(1)$ by letting $\xi$ act by $1$ on eigenvectors of $k_1$ with eigenvalue in $q^\Z$ and by $-1$ on eigenvectors with eigenvalue in $q^{\frac{1}{2} + \Z}$.

    \item We identify $U_q(0)$ with the subalgebra of $U_q(2)$ generated by $\sigma$.  We then have usual restriction of modules.
\end{itemize}
With the above conventions, for all $N \ge 2$, the action of $U_q(N-2)$ leaves $S_1$ invariant and $S_1$ is isomorphic to the spin module of $U_q(N-2)$.

\begin{lem}\label{restrict}
    For $N \ge 2$ and $x \in S_1 \otimes S_1$, we have $B_N x = B_{N-2} x$.
\end{lem}

\begin{proof}
    Suppose $I,J \subseteq \{1,2,\dotsc,n\}$ and $n \in I,J$.  Then
    \[
        \Psi_n x_I = 0 = \Psi_n x_J,\qquad
        \omega_n^{\pm 1} x_I \overset{\cref{drwho}}{=} q^{\mp 1} x_I,\qquad
        \omega_n^{\pm 1} x_J \overset{\cref{drwho}}{=} q^{\mp 1} x_J.
    \]
    Thus, if $N=2n+1$,
    \begin{multline*}
        B_N (x_I \otimes x_J)
        = \sum_{i=-n}^n (q+1)^{-\delta_{i,0}} \left( \prod_{j=|i|+1}^n \omega_j \right) \Psi_i x_I \otimes \left( \prod_{j=|i|+1}^n \omega_j^{-1} \right) \Psi_{-i} x_J
        \\
        = \sum_{i=1-n}^{n-1} (q+1)^{-\delta_{i,0}} \left( \prod_{j=|i|+1}^{n-1} \omega_j \right) \Psi_i x_I \otimes \left( \prod_{j=|i|+1}^{n-1} \omega_j^{-1} \right) \Psi_{-i} x_J
        = B_{N-2} (x_I \otimes x_J).
    \end{multline*}
    The proof in the case $N=2n$ is analogous; one simply omits the $i=0$ term in the above sums.
\end{proof}

\begin{prop}\label{eigenB}
    Suppose $N = 2n+1$, and recall the decomposition \cref{SdubB}.  For $1 \le k \le n+1$, the operator $B_N$ acts on the summand $L_q(\epsilon_n + \dotsb + \epsilon_k)$ as scalar multiplication by
    \[
        (-1)^{k-1} (q+1)^{-1} ([k-1]+[k]).
    \]
\end{prop}

\begin{proof}
    We induct on $n$.  When $n=0$, we have $B_N = (q+1)^{-1}$ by \cref{BlowN}.  Now suppose $n \ge 1$ and let $\lambda_k$ be the eigenvalue of $B_N$ on $L_q(\epsilon_n + \dotsb + \epsilon_k)$.  If $k \le n$, then the highest-weight vector of $L_q(\epsilon_n + \dotsb + \epsilon_k)$ lies in $S_1 \otimes S_1$ and is a highest-weight vector for the simple module $L_q(\epsilon_{n-1} + \dotsb + \epsilon_k)$ of $U_q(N-2)$. By \cref{restrict} and our inductive hypothesis, this implies that $\la_k = (-1)^{k-1} (q+1)^{-1} ([k-1]+[k])$.

    The image under $\bF$ of the left-hand side of \cref{trace1} is the quantum trace of $B_N$.  (See, for example, \cite[Lem.~7.1]{SW22}.)  Therefore,
    \[
        0
        \overset{\cref{trace1}}{=} \sum_{k=1}^{n+1} \lambda_k \dim_q L_q(\epsilon_n + \dotsb + \epsilon_k)
        \overset{\cref{screen}}{=} \sum_{k=1}^{n+1} \lambda_k \left( \qbinom{2n}{n-k+1} + \qbinom{2n}{n-k} \right).
    \]
    Thus, $\lambda_{n+1}$ is uniquely determined by $\lambda_1,\dotsc,\lambda_n$.  Hence, it suffices to prove the identity
    \[
        \sum_{k=1}^{n+1} (-1)^{k-1} ([k-1] + [k]) \left( \qbinom{2n}{n-k+1} + \qbinom{2n}{n-k} \right)
        = 0,
    \]
    where we adopt the convention that $[k] = 0 = \qbinom{2n}{k}$ when $k < 0$.  For this, we compute
    \begin{multline*}
        \sum_{k=1}^{n+1} (-1)^{k-1} ([k-1] + [k]) \left( \qbinom{2n}{n-k+1} + \qbinom{2n}{n-k} \right)
        = \sum_{k=1}^{n+1} (-1)^k [k] \left( \qbinom{2n}{n-k-1} - \qbinom{2n}{n-k+1} \right)
        \\
        = \sum_{k=0}^n (-1)^k ([k+1]-[k-1]) \qbinom{2n}{n-k}
        = \qbinom{2n}{n} + \sum_{k=1}^n (-1)^k \left( q^k + q^{-k} \right) \qbinom{2n}{n-k}
        \\
        = (-q)^{-n} \sum_{k=0}^{2n} (-q)^k \qbinom{2n}{k}
        = 0,
    \end{multline*}
    where the last equality follows from evaluating the generating function \cref{bird}, with $m=2n$, at $x=-q$.
\end{proof}

\begin{prop}\label{prop:2moment}
    If $m \ge 2$, then
    \begin{equation}\label{2moment}
        \sum_{k=-m}^m \left( \qbinom{2m-1}{m-k-1}+\qbinom{2m-1}{m-k} \right) [k]^2
        = 4 ([2m-1]+1)\prod_{j=1}^{m-2} (q^j+q^{-j})^2,
    \end{equation}
    where we interpret $\qbinom{2m-1}{m-k-1}$ and $\qbinom{2m-1}{m-k}$ as zero when $k=m$ and $k=-m$, respectively.
\end{prop}

\begin{proof}
    For $f(x) \in \kk[x^{\pm 1}]$, define $Df(x)=\frac{f(qx)-f(q\inv x)}{q-q\inv}$. Let
    \[
        F(x)=\sum_{k=-m}^m \left( \qbinom{2m-1}{m-k-1}+\qbinom{2m-1}{m-k} \right) x^{-k}
        \overset{\cref{bird}}{=} x^{-m} (1+x) \prod_{j=1}^{2m-1} (1+q^{2j-2m}x).
    \]
    \details{
        We have
        \[
            \sum_{k=-m}^m \qbinom{2m-1}{m-k-1} x^{-k}
            = \sum_{k=0}^{2m-1} \qbinom{2m-1}{k} x^{k-m+1}
            \overset{\cref{bird}}{=} x^{1-m} \prod_{j=1}^{2m-1} (1+q^{2j-2m} x)
        \]
        and
        \[
            \sum_{k=-m}^m \qbinom{2m-1}{m-k} x^{-k}
            = \sum_{k=0}^{2m-1} \qbinom{2m-1}{k} x^{k-m}
            \overset{\cref{bird}}{=} x^{-m} \prod_{j=1}^{2m-1} (1+q^{2j-2m} x).
        \]
    }
    Then the left-hand side of \cref{2moment} is equal to $(D^2F)(1)$.  Since
    \begin{align*}
        (D^2F)(1)
        &= \frac{F(q^2)-2F(1)+F(q^{-2})}{(q-q\inv)^2} \\
        &= \frac{1}{(q-q^{-1})^2}
        \Big(
            q^{-2m} (1+q^2) (1+q^{2m-2}) (1+q^{2m})
            - 4 (1+q^{2-2m}) (1+q^{2m-2})
            \\ & \qquad \qquad
            + q^{2m} (1+q^{-2}) (1+q^{-2m}) (1+q^{2-2m})
        \Big)
        \prod_{j=2}^{2m-2} \left( 1 + q^{2j-2m} \right)
        \\
        &= 2 \frac{(q^{m-1}+q^{1-m})(q^m-q^{-m})}{q-q^{-1}} \prod_{j=2}^{2m-2} q^{j-m} \left( q^{m-j} + q^{j-m} \right)
        \\
        &= 2 ([2m-1]+1) \prod_{j=2}^{2m-2} (q^{m-j} + q^{j-m})
        \\
        &= 4 ([2m-1]+1)\prod_{j=1}^{m-2} (q^j+q^{-j})^2,
    \end{align*}
    the result follows.
\end{proof}

\begin{prop}
    Suppose $N=2n>0$, and recall the decomposition \cref{SdubD}.  For $2 \le k \le n+1$, the operator $B_N$ acts on the two simple $U_q(N)$-submodules of $\Ind(L_q(\epsilon_n + \epsilon_{n-1} + \dotsb + \epsilon_k))$ as scalar multiplication by $[k-1]$ and $-[k-1]$.  Furthermore, it acts as zero on the $k=1$ summand.
\end{prop}

\begin{proof}
    The proof is similar to the proof of \cref{eigenB}. We induct on $n\ge1$. For $n=1$, the $k=2$ summand is spanned by $x_\varnothing\otimes x_{\{1\}}$ and
    $x_{\{1\}}\otimes x_\varnothing$, as can be seen from weight considerations. The result then follows from the fact that
    \[
        B_2 \left( x_\varnothing \otimes x_{\{1\}} \pm x_{\{1\}} \otimes x_\varnothing \right)
        \overset{\cref{BlowN}}{=} \pm \left( x_\varnothing \otimes x_{\{1\}} \pm x_{\{1\}} \otimes x_\varnothing \right)
    \]
    and
    \[
        B_2 \left( x_\varnothing \otimes x_\varnothing \right)
        \overset{\cref{BlowN}}{=} 0
        \overset{\cref{BlowN}}{=} B_2 \left( x_{\{1\}} \otimes x_{\{1\}} \right).
    \]

    Now suppose $n \ge 2$, and let $\lambda_k,\mu_k$ be the eigenvalues of $B_N$ acting on the $U_q(N)$-submodules of $\Ind(L_q(\epsilon_n + \dotsb + \epsilon_k))$.  When $k=1$, we set $\lambda_1 = \mu_1 = 0$.  As in the proof of \cref{eigenB}, \cref{restrict} and the induction hypothesis imply, without loss of generality, that $\lambda_k = [k-1] = -\mu_k$ for $1 \le k \le n$.

    It follows from \cref{trivflip} that the two simple submodules of $\Ind(L_q(\epsilon_n + \dotsb + \epsilon_k))$, $2 \le k \le n$, have the same quantum dimension.  Thus,
    \[
        0
        \overset{\cref{trace1}}{=} \sum_{k=2}^{n+1} (\lambda_k + \mu_k) \dim_q L_q(\epsilon_n + \dotsb + \epsilon_k)
        = \lambda_{n+1} + \mu_{n+1},
    \]
    where the last equality follows from the induction hypothesis and the fact that $\dim_q L_q(0) = 1$.  Thus, $\lambda_{n+1} = - \mu_{n+1}$.

    By \cref{rhino}, the image under $\bF$ of the left-hand side of \cref{trace2} is the quantum trace of $(q^{n-1}+q^{1-n})^2 B_N^2$.  Thus, by the quantum-dimension formula \cref{screen} and the induction hypothesis, we have
    \begin{multline*}
        \lambda_{n+1}^2 + \mu_{n+1}^2 + \sum_{k=1-n}^{n-1} \left( \qbinom{2n-1}{n-k-1} + \qbinom{2n-1}{n-k}\right) [k]^2
        \\
        \overset{\cref{trace2}}{=} \frac{d_\Vgo d_\Sgo^2}{(q^{n-1}+q^{1-n})^2}
        \overset{\cref{crabby}}{\underset{\cref{shimmy}}{=}} 4([2n-1]+1) \prod_{j=1}^{n-2} (q^j+q^{-j})^2.
    \end{multline*}
    By \cref{prop:2moment}, it follows that $\lambda_{n+1}^2 + \mu_{n+1}^2 = 2[n]^2$.  Thus, without loss of generality, $\lambda_{n+1} = [n]$ and $\mu_{n+1} = -[n]$.
\end{proof}

\begin{cor}\label{eigenD}
    If $N=2n$, then the eigenvalues of $B_N$ are $[i]$, $i \in \Z$, $-n \le i \le n$.
\end{cor}

We define a \emph{barbell} to be any element of the form $1^{\otimes t} \otimes B_N \otimes 1^{\otimes(r-t-2)} \in \Cl_q(N)^{\otimes r}$, $0 \le t \le r-2$, $r \ge 2$.  The action of a barbell yields an element of $\End_{U_q(N)}(S^{\otimes r})$.

\begin{lem}\label{barbellgen}
    The action of the barbell $B_N$ generates $\End_{U_q(N)}(S \otimes S)$.
\end{lem}

\begin{proof}
    For $N=0$, we have $\End_{U_q(0)}(S^{\otimes 2}) = \kk 1$.  For $N > 0$, the result follows from the fact that, by \cref{eigenB,eigenD}, the operator $B_N$ acts on the simple summands in the decompositions \cref{SdubB,SdubD} by pairwise distinct scalars.
\end{proof}

\begin{theo} \label{surly}
    Suppose $r,r_1,r_2 \in \N$.
    \begin{enumerate}
        \item \label{Sfull} The incarnation functor $\bF$ induces a surjection
            \[
                \Hom_{\QSB(N)}(\Sgo^{\otimes r_1}, \Sgo^{\otimes r_2})
                \twoheadrightarrow \Hom_{U_q(N)}(S^{\otimes r_1}, S^{\otimes r_2}).
            \]

        \item \label{BBgen} The algebra $\End_{U_q(N)}(S^{\otimes r})$ is generated by barbells.
    \end{enumerate}
\end{theo}

\begin{proof}
    The proof is the same as that of \cite[Th.~7.8]{MS24}.  In order to run this argument, we need the following facts:
    \begin{itemize}
        \item the barbell generates $\End_{U_q(N)}(S\otimes S)$, which is \cref{barbellgen},
        \item the category $U_q(N)$-mod is semisimple, which is \cref{semisimple},
        \item restricting from $U_q(N)$ to $U_q(N-2)$ takes the barbell to the barbell, which is \cref{restrict}.
    \end{itemize}
\end{proof}

\begin{theo} \label{full}
    The incarnation functor $\bF$ is full.
\end{theo}

\begin{proof}
    The proof is analogous to that of \cite[Th.~7.9]{MS24}.  The only difference is the diagrammatic computation appearing there.  In the quantum setting, this becomes

\begin{align*}
        0
        &\overset{\cref{Deligne}}{\underset{\substack{\cref{ski} \\ \cref{bump}}}{=}} q\,
        \begin{tikzpicture}[anchorbase]
            \draw[spin] (-0.1,0) -- (0.1,0) arc(-90:90:0.15) -- (-0.1,0.3) arc(90:270:0.15);
            \draw[vec] (-0.1,-0.4) -- (-0.1,0);
            \draw[vec] (0.1,-0.4) -- (0.1,0);
        \end{tikzpicture}
        -
        \begin{tikzpicture}[anchorbase]
            \draw[vec] (0.1,-0.4) -- (-0.1,0);
            \draw[wipe] (-0.1,-0.4) -- (0.1,0);
            \draw[vec] (-0.1,-0.4) -- (0.1,0);
            \draw[spin] (-0.1,0) -- (0.1,0) arc(-90:90:0.15) -- (-0.1,0.3) arc(90:270:0.15);
        \end{tikzpicture}
        - d_\Vgo \frac{q-q^{-1}}{1+q^N}
        \begin{tikzpicture}[anchorbase]
            \draw[spin] (-0.1,0) -- (0.1,0) arc(-90:90:0.15) -- (-0.1,0.3) arc(90:270:0.15);
            \draw[vec] (-0.15,-0.4) -- (-0.15,-0.3) arc(180:0:0.15) -- (0.15,-0.4);
        \end{tikzpicture}
        \\
        &\overset{\cref{oist}}{=} (q+q^{-1})\
        \begin{tikzpicture}[anchorbase]
            \draw[spin] (-0.1,0) -- (0.1,0) arc(-90:90:0.15) -- (-0.1,0.3) arc(90:270:0.15);
            \draw[vec] (-0.1,-0.4) -- (-0.1,0);
            \draw[vec] (0.1,-0.4) -- (0.1,0);
        \end{tikzpicture}
        - \left( \frac{q^{N-1} - q^{1-N} + q - q^{-1}}{1+q^N} + q^{1-N} + q^{-1} \right) \
        \begin{tikzpicture}[anchorbase]
            \draw[spin] (-0.1,0) -- (0.1,0) arc(-90:90:0.15) -- (-0.1,0.3) arc(90:270:0.15);
            \draw[vec] (-0.15,-0.4) -- (-0.15,-0.3) arc(180:0:0.15) -- (0.15,-0.4);
        \end{tikzpicture}
        \\
        &\overset{\cref{dimrel}}{=}
        (q+q^{-1})\
        \begin{tikzpicture}[anchorbase]
            \draw[spin] (-0.1,0) -- (0.1,0) arc(-90:90:0.15) -- (-0.1,0.3) arc(90:270:0.15);
            \draw[vec] (-0.1,-0.4) -- (-0.1,0);
            \draw[vec] (0.1,-0.4) -- (0.1,0);
        \end{tikzpicture}
        - 2 d_\Sgo \frac{q+q^{N-1}}{1+q^N}\
        \begin{tikzpicture}[anchorbase]
            \draw[vec] (-0.2,-0.4) -- (-0.2,-0.2) arc(180:0:0.2) -- (0.2,-0.4);
        \end{tikzpicture}
        \ .
    \end{align*}
    Therefore
    \[
        \begin{tikzpicture}[centerzero]
            \draw[vec] (0,-0.5) -- (0,-0.2);
            \draw[vec] (0,0.5) -- (0,0.2);
            \draw[spin] (0,0) circle(0.2);
        \end{tikzpicture}
        = 2 d_\Sgo \frac{q+q^{N-1}}{(1+q^N)(q+q^{-1})}\
        \begin{tikzpicture}[centerzero]
            \draw[vec] (0,-0.5) -- (0,0.5);
        \end{tikzpicture}
        \ .
    \]
    Then one replaces the $D$ in the last line in the proof of \cite[Th.~7.9]{MS24} with $2 d_\Sgo \frac{q+q^{N-1}}{(1+q^N)(q+q^{-1})}$.
\end{proof}

\section{Essential surjectivity of the incarnation functor\label{sec:essential}}

Throughout this section, we suppose that $N \in \N$, $\kk = \C(q^{\pm \frac{1}{4}})$, and that $t$, $\kappa$, and $d_\Sgo$ are given by \cref{crabby}.  In particular, this means that $\kappa^2 = q^{1-N}$, and so the equality in \cref{Deligne} holds unless $r=N$ is a positive odd integer.

Let $\Kar(\overline{\QSB}(N))$ be the additive Karoubi envelope (that is, the idempotent completion of the additive envelope) of $\overline{\QSB}(N)$.  Since $U_q(N)\md$ is additive and idempotent complete, $\bF$ induces a monoidal functor
\begin{gather} \label{Karnation}
    \Kar(\bF) \colon \Kar \big( \overline{\QSB}(N) \big) \to U_q(N)\md.
\end{gather}
Our goal in this final section is to show that $\Kar(\bF)$ is essentially surjective.  Our arguments will involve the following subcategory of the quantum spin Brauer category.

\begin{defin}
    Suppose $d_\Sgo \in \kk$ and $q,t,\kappa \in \kk^\times$ satisfy $q-q^{-1} \in \kk^\times$.  We define $\QSB' = \QSB'(q,t,\kappa,d_\Sgo)$ to be the $\kk$-linear monoidal subcategory of $\overline{\QSB}(q,t,\kappa,d_\Sgo)$ generated by the objects $\Sgo$ and $\Vgo$ and the morphisms
    \[
        \capmor{spin}\ ,\qquad
        \cupmor{spin}\ ,\qquad
        \capmor{vec}\ ,\qquad
        \cupmor{vec}\ ,\qquad
        \poscross{vec}{vec}\ ,\qquad
        \mergemor{vec}{spin}{spin}.
    \]
\end{defin}

Note that, by \cref{zombie}, the subcategory $\QSB'$ also contains the morphisms $\poscross{spin}{vec}$ and $\poscross{vec}{spin}$ as long as $q\kappa^2+1$ is invertible (which we will assume below).  Then, using cups and caps to rotate crossings, we see that $\QSB'$ contains all blue-blue and blue-black crossings.  Thus, the essential point is that we omit the crossings $\poscross{spin}{spin}$ and $\negcross{spin}{spin}$.  As we will see below, the result is that certain key morphism spaces in $\QSB'$ are finite dimensional.  In contrast, the corresponding morphism spaces in $\QSB$ may be infinite dimensional; the lack of a skein relation on black strands means we cannot reduce arbitrary black links to multiples of the empty diagram.

For the remainder of this section we assume that $\kk$ is a field and $q,t,\kappa \in \kk^\times$ such that $q$ is not a root of unity and \cref{lime} is satisfied.  The next result shows that all closed diagrams in $\QSB'$ can be reduced to a multiple of the empty diagram $1_\one$.  Its proof is similar to that of \cite[Prop.~5.9]{MS24}.

\begin{prop} \label{blackhole}
    We have $\End_{\QSB'}(\one) = \kk 1_\one$.
\end{prop}

\begin{proof}
    We freely use isotopy invariance of diagrams.  We first prove that, for $r \in \N$,
    \begin{equation} \label{snail}
        \begin{tikzpicture}[anchorbase]
            \draw[spin] (-0.25,0.3) -- (-0.25,0.15) arc(180:270:0.15) -- (0.1,0) arc(-90:0:0.15) -- (0.25,0.3);
            \draw[vec] (0,-1) -- (0,0);
            \altbox{-0.2,-0.65}{0.2,-0.35}{r};
        \end{tikzpicture}
        \equiv
        \begin{tikzpicture}[anchorbase]
            \draw[spin] (-0.25,0.3) -- (-0.25,0.15) arc(180:270:0.15) -- (0.1,0) arc(-90:0:0.15) -- (0.25,0.3);
            \draw[vec] (0,-1) -- (0,0) node[midway,anchor=west] {\strandlabel{r}};
        \end{tikzpicture}
        \ ,
    \end{equation}
    where $\equiv$ denotes equivalence modulo linear combinations of diagrams with fewer than $r$ blue strands attached to the black arc.
    Since \cref{snail} trivially holds when $r \in \{0,1\}$, we assume $r \ge 2$. Then we have
    \[
        \begin{tikzpicture}[centerzero]
            \draw[spin] (-0.25,0.9) -- (-0.25,0.75) arc(180:270:0.15) -- (0.1,0.6) arc(-90:0:0.15) -- (0.25,0.9);
            \draw[vec] (0,-0.6) -- (0,0.6);
            \altbox{-0.4,-0.15}{0.4,0.15}{r+1};
        \end{tikzpicture}
        \overset{\cref{asymdef}}{\underset{\cref{bump}}{\equiv}} \frac{1}{[r+1]}
        \left(
            q^r\
            \begin{tikzpicture}[centerzero]
                \draw[spin] (-0.25,0.9) -- (-0.25,0.75) arc(180:270:0.15) -- (0.6,0.6) arc(-90:0:0.15) -- (0.75,0.9);
                \draw[vec] (0,-0.6) -- (0,0.6);
                \altbox{-0.3,-0.15}{0.3,0.15}{r};
                \draw[vec] (0.5,-0.6) -- (0.5,0.6);
            \end{tikzpicture}
            - [r]\
            \begin{tikzpicture}[centerzero]
                \draw[spin] (-0.25,1) -- (-0.25,0.85) arc(180:270:0.15) -- (0.6,0.7) arc(-90:0:0.15) -- (0.75,1);
                \draw[vec] (-0.15,-0.2) -- (-0.15,0.2) node[midway,anchor=east] {\strandlabel{r-1}};
                \draw[vec] (0.15,0.2) \braiddown (0.5,-0.2) -- (0.5,-0.7);
                \draw[wipe] (0.15,-0.2) \braidup (0.5,0.2);
                \draw[vec] (0.15,-0.2) \braidup (0.5,0.2) -- (0.5,0.7);
                \draw[vec] (0,0.5) -- (0,0.7);
                \draw[vec] (0,-0.5) -- (0,-0.7);
                \altbox{-0.3,-0.5}{0.3,-0.2}{r};
                \altbox{-0.3,0.2}{0.3,0.5}{r};
            \end{tikzpicture}\
        \right)
        \overset{\cref{oist}}{\underset{\cref{asymidem}}{\equiv}}
        \frac{q^r + q^{-1}[r]}{[r+1]}\
        \begin{tikzpicture}[centerzero]
            \draw[spin] (-0.25,0.9) -- (-0.25,0.75) arc(180:270:0.15) -- (0.6,0.6) arc(-90:0:0.15) -- (0.75,0.9);
            \draw[vec] (0,-0.6) -- (0,0.6);
            \altbox{-0.3,-0.15}{0.3,0.15}{r};
            \draw[vec] (0.5,-0.6) -- (0.5,0.6);
        \end{tikzpicture}
        =
        \begin{tikzpicture}[centerzero]
            \draw[spin] (-0.25,0.9) -- (-0.25,0.75) arc(180:270:0.15) -- (0.6,0.6) arc(-90:0:0.15) -- (0.75,0.9);
            \draw[vec] (0,-0.6) -- (0,0.6);
            \altbox{-0.3,-0.15}{0.3,0.15}{r};
            \draw[vec] (0.5,-0.6) -- (0.5,0.6);
        \end{tikzpicture}
        \ ,
    \]
    and so \cref{snail} holds for all $r \in \N$ by induction.

    We now prove that any closed diagram is equal to a multiple of the empty diagram.  We proceed by induction on the number of trivalent vertices in the diagram.  Suppose we have a closed diagram with at least one trivalent vertex.  Consider the black curve that is part of that trivalent vertex.  Since every trivalent vertex has exactly two black strings incident to it, this curve is part of a loop.  By the definition of $\QSB'$, we may assume that this curve has no self-intersections and does not cross any other strands.  Using \cref{lobster1,braid}, we may move any blue strings in the interior of the loop to the exterior.  Thus, we may assume that the interior of the loop is empty.  Let $r$ be the number of trivalent vertices on this loop. Unless $r=N$ is an odd positive integer, we have
    \[
        0
        \overset{\cref{Deligne}}{=}
        \begin{tikzpicture}[anchorbase]
            \draw[spin] (-0.1,0) -- (0.1,0) arc(-90:90:0.15) -- (-0.1,0.3) arc(90:270:0.15);
            \draw[vec] (0,-1) -- (0,0);
            \altbox{-0.2,-0.65}{0.2,-0.35}{r};
        \end{tikzpicture}
        \overset{\cref{snail}}{\equiv}
        \begin{tikzpicture}[anchorbase]
            \draw[spin] (-0.1,0) -- (0.1,0) arc(-90:90:0.15) -- (-0.1,0.3) arc(90:270:0.15);
            \draw[vec] (0,-1) -- (0,0) node[midway,anchor=west] {\strandlabel{r}};
        \end{tikzpicture}
        \ ,
    \]
    and so we can write our diagram as a linear combination of diagrams with fewer trivalent vertices.

    Suppose instead that $r=N$ is an odd positive integer. Since the total number of trivalent vertices is even, there must be another black loop with a trivalent vertex. We can repeat the preceding argument for that loop. We then either rewrite the original diagram as a linear combination of diagrams with fewer trivalent vertices or find that the second loop also has $N$ trivalent vertices. In the latter case, we obtain a subdiagram of the form
    \[
        \begin{tikzpicture}[anchorbase]
            \draw[spin] (-0.1,0) -- (0.1,0) arc(-90:90:0.15) -- (-0.1,0.3) arc(90:270:0.15);
            \draw[vec] (0,-0.5) -- (0,0) node[midway,anchor=west] {\strandlabel{N}};
        \end{tikzpicture}
        \begin{tikzpicture}[anchorbase]
            \draw[spin] (-0.1,0) -- (0.1,0) arc(-90:90:0.15) -- (-0.1,0.3) arc(90:270:0.15);
            \draw[vec] (0,-0.5) -- (0,0) node[midway,anchor=west] {\strandlabel{N}};
        \end{tikzpicture}
        \ .
    \]
        Then we have
    \[
        \left( d_\Sgo^2 \left( \frac{q\kappa^2+1}{q-q\inv} \right)^N\ \prod_{i=0}^{N-1} \left( \frac{\kappa^{-2} - q^{2i-2} \kappa^2}{ q^{2i-1}\kappa^2+1 } \right)\
        \right)\
        \begin{tikzpicture}[anchorbase]
            \draw[vec] (-0.2,0) -- (-0.2,1) arc(180:0:0.2) -- (0.2,0);
            \altbox{-0.4,0.3}{0,0.6}{N};
        \end{tikzpicture}
        \overset{\cref{extra}}{=}
        \begin{tikzpicture}[anchorbase]
            \draw[spin] (-0.1,0) -- (0.1,0) arc(-90:90:0.15) -- (-0.1,0.3) arc(90:270:0.15);
            \draw[vec] (0,-1) -- (0,0);
            \altbox{-0.2,-0.65}{0.2,-0.35}{N};
        \end{tikzpicture}
        \
        \begin{tikzpicture}[anchorbase]
            \draw[spin] (-0.1,0) -- (0.1,0) arc(-90:90:0.15) -- (-0.1,0.3) arc(90:270:0.15);
            \draw[vec] (0,-1) -- (0,0);
            \altbox{-0.2,-0.65}{0.2,-0.35}{N};
        \end{tikzpicture}
        \overset{\cref{snail}}{\equiv}
        \begin{tikzpicture}[anchorbase]
            \draw[spin] (-0.1,0) -- (0.1,0) arc(-90:90:0.15) -- (-0.1,0.3) arc(90:270:0.15);
            \draw[vec] (0,-1) -- (0,0) node[midway,anchor=west] {\strandlabel{N}};
        \end{tikzpicture}
        \begin{tikzpicture}[anchorbase]
            \draw[spin] (-0.1,0) -- (0.1,0) arc(-90:90:0.15) -- (-0.1,0.3) arc(90:270:0.15);
            \draw[vec] (0,-1) -- (0,0) node[midway,anchor=west] {\strandlabel{N}};
        \end{tikzpicture}
        ,
    \]
    and we proceed as before.

    We have now reduced to the case of diagrams with no trivalent vertices.  In this case, the relations \cref{braid,deloop} suffice to rewrite our diagrams as a disjoint union of circles, which are evaluated as scalars by \cref{dimrel}.
\end{proof}

\begin{prop} \label{tulip}
    For all $r,s \in \N$, the spaces $\Hom_{\QSB'}(\Vgo^{\otimes r}, \Vgo^{\otimes s})$ and $\Hom_{\QSB'}(\Vgo^{\otimes r} \otimes \Sgo, \Vgo^{\otimes s} \otimes \Sgo)$ are finite dimensional.
\end{prop}

\begin{proof}
    As in the proof of \cref{blackhole}, closed diagrams can be reduced to a scalar multiple of the empty diagram.  Thus, $\Hom_{\QSB'}(\Vgo^{\otimes r}, \Vgo^{\otimes s})$ is spanned by diagrams with no loops.  In particular, it is spanned by diagrams with no black strands.  Using \cref{braid,skein,deloop,reloop}, it is then a standard skein theory argument in the theory of tangles to show that this space is finite dimensional.

    Similarly, $\Hom_{\QSB'}(\Vgo^{\otimes r} \otimes \Sgo, \Vgo^{\otimes s} \otimes \Sgo)$ is spanned by diagrams with no loops.  In particular, it is spanned by diagrams that have a single black string running from the bottom of the diagram to the top that does not intersect itself.  Using \cref{braid,lobster1}, we can move all blue strings to the left of the black string.

    We now claim that we can remove all instances of a blue string being incident to the black string twice. Using \cref{braid,deloop,reloop,skein} we can reduce the case to where the picture locally looks like the following:
    \[
        \begin{tikzpicture}
            \draw[vec] (0,-0.5) to[out=left,in=left,looseness=2] (0,0.5);
            \draw[wipe] (0,0.3) to[out=left,in=right] (-0.7,0.6);
            \draw[vec] (0,0.3) to[out=left,in=right] (-0.7,0.6);
            \draw[wipe] (0,-0.3) to[out=left,in=right] (-0.7,-0.6);
            \draw[vec] (0,-0.3) to[out=left,in=right] (-0.7,-0.6);
            \node at (-0.2,0.1) {$\vdots$};
            \draw[spin] (0,-0.7) -- (0,0.7);
        \end{tikzpicture}
    \]
    We then use \cref{oist,oister} to reduce to the case where there are no strands in the interior of the blue-black loop.  Then we can use \cref{bump} to simplify the diagram.

    Now we are in a position where we can assume that no blue string is incident to the black string more than once.  Furthermore, using \cref{braid,deloop,reloop,skein}, we may assume that no two blue strings cross more than once and no blue strings intersect themselves (as in the standard skein theory argument mentioned above).  An example of such a diagram for $r=6$ and $s=5$ is the following:
    
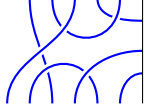
\[
        \begin{tikzpicture}
            \draw[vec] (1.5,-0.7) to[out=up,in=left] (1.8,-0.3);
            \draw[vec] (1.2,0.7) to[out=down,in=left] (1.8,0.4);
            \draw[vec] (0.6,-0.7) \braidup (0.3,0.7);
            \draw[wipe] (0.3,-0.7) to[out=up,in=up,looseness=2] (1.2,-0.7);
            \draw[vec] (0.3,-0.7) to[out=up,in=up,looseness=2] (1.2,-0.7);
            \draw[wipe] (0.9,-0.7) to[out=up,in=left] (1.8,0);
            \draw[vec] (0.9,-0.7) to[out=up,in=left] (1.8,0);
            \draw[wipe] (0.6,0.7) to[out=down,in=down,looseness=2] (1.5,0.7);
            \draw[vec] (0.6,0.7) to[out=down,in=down,looseness=2] (1.5,0.7);
            \draw[wipe] (0,-0.7) \braidup (0.9,0.7);
            \draw[vec] (0,-0.7) \braidup (0.9,0.7);
            \draw[spin] (1.8,-0.7) -- (1.8,0.7);
        \end{tikzpicture}
    \]
    Since the number of diagrams satisfying these conditions is finite, the result follows.
\end{proof}

Let
\begin{equation} \label{incarprime}
    \bF' \colon \QSB' \to U_q(N)\md
\end{equation}
be the restriction to $\QSB'$ of the functor induced by $\bF$ on $\overline{\QSB}(N)$.

\begin{lem} \label{easter}
    The functor $\bF'$ is full.
\end{lem}

\begin{proof}
    Since none of the arguments in \cref{sec:full} used the crossings $\poscross{spin}{spin}$ or $\negcross{spin}{spin}$, this follows from \cref{full}.
\end{proof}

\begin{theo} \label{essential}
    The functors $\Kar(\bF')$ and $\Kar(\bF)$ are essentially surjective.
\end{theo}

\begin{proof}
    The cases $N \in \{0,1\}$ are immediate, since the two simple $U_q(N)$-modules are $\triv$ and $S$. Hence, suppose that $N\ge2$.

    Since $U_q(N)\md$ is semisimple by \cref{semisimple}, it suffices to show that every finite-dimensional simple $U_q(N)$-module is in the essential image of $\Kar(\bF')$, which also implies it is in the essential image of $\Kar(\bF)$.  Let $M$ be a finite-dimensional simple $U_q(N)$-module.  Let $s=0$ if the components of the highest weight of $M$ are integers and let $s=1$ otherwise (i.e., if the components of the highest weight are half-integers).  Then $M$ is isomorphic to a summand of $V^{\otimes r} \otimes S^{\otimes s}$ for some $r \in \N$.  By \cref{easter}, the idempotent $e_M$ in $\End_{U_q(N)}(V^{\otimes r} \otimes S^{\otimes s})$ projecting onto $M$ is in the image of
    \[
        \bF' \colon \End_{\QSB'}(\Vgo^{\otimes r} \otimes \Sgo^{\otimes s}) \twoheadrightarrow \End_{U_q(N)}(V^{\otimes r} \otimes S^{\otimes s}).
    \]
    By \cref{tulip}, $\End_{\QSB'}(\Vgo^{\otimes r} \otimes \Sgo^{\otimes s})$ is finite dimensional.  Hence, we can lift $e_M$ to an idempotent $e \in \End_{\QSB'}(\Vgo^{\otimes r} \otimes \Sgo^{\otimes s})$ such that $\bF'(e) = e_M$.
    \details{
        Indeed, let $\pi \colon A \twoheadrightarrow B$ be a surjective homomorphism of unital finite-dimensional $\kk$-algebras, and let $e \in B$ be idempotent. Choose $x \in A$ with $\pi(x)=e$. The cases $e=0,1$ are immediate. Otherwise, if $m(T)$ is the minimal polynomial of $x$, then $T(T-1)$ divides $m(T)$, since the minimal polynomial of $e$ is $T(T-1)$. Write
        \[
            m(T)=T^a(T-1)^b g(T),
            \qquad
            a,b\geq 1,
            \qquad
            \gcd\bigl(g(T),T(T-1)\bigr)=1.
        \]
        By the Chinese remainder theorem, there is a polynomial $p(T)$ such that
        \[
            p(T)\equiv 0 \pmod{T^a},\qquad
            p(T)\equiv 1 \pmod{(T-1)^b},\qquad
            p(T)\equiv 0 \pmod{g(T)}.
        \]
        Then $m(T)$ divides $p(T)^2-p(T)$, so $p(x)$ is idempotent. Moreover, $p(T)\equiv T \pmod{T(T-1)}$, and hence $\pi\bigl(p(x)\bigr)=p(e)=e$.
    }
    Thus, $M$ is in the essential image of $\Kar(\bF')$ as desired.
\end{proof}

\begin{rem}
    The non-quantum analogue of \cref{essential} is \cite[Th.~8.1]{MS24}.  There is a gap in the proof of \cite[Th.~8.1]{MS24}, where the lifting of idempotents is not justified, since the morphism spaces in the spin Brauer category may be infinite dimensional.  The strategy of the proof of \cref{essential}, which also applies to the non-quantum setting, fixes this gap.
\end{rem}

It is straightforward to verify that $\QSB'$ is a spherical pivotal category, hence so is its idempotent completion $\Kar(\QSB')$.  (We refer the reader to \cite[\S4.4.3]{Sel11} for the definition of a spherical pivotal category.)  In any spherical pivotal category $\cC$, we have a trace map $\Tr \colon \bigoplus_{X \in \cC} \End_\cC(X) \to \End_\cC(\one)$.  In terms of string diagrams, this corresponds to closing a diagram off to the right or left:
\begin{equation} \label{natal}
    \Tr
    \left(
        \begin{tikzpicture}[centerzero]
            \draw[line width=2] (0,-0.5) -- (0,0.5);
            \filldraw[fill=white,draw=black] (-0.25,0.2) rectangle (0.25,-0.2);
            \node at (0,0) {$\scriptstyle{f}$};
        \end{tikzpicture}
    \right)
    =
    \begin{tikzpicture}[centerzero]
        \draw[line width=2] (0,0.2) arc(180:0:0.3) -- (0.6,-0.2) arc(360:180:0.3);
        \filldraw[fill=white,draw=black] (-0.25,0.2) rectangle (0.25,-0.2);
        \node at (0,0) {$\scriptstyle{f}$};
    \end{tikzpicture}
    =
    \begin{tikzpicture}[centerzero]
        \draw[line width=2] (0,0.2) arc(0:180:0.3) -- (-0.6,-0.2) arc(180:360:0.3);
        \filldraw[fill=white,draw=black] (-0.25,0.2) rectangle (0.25,-0.2);
        \node at (0,0) {$\scriptstyle{f}$};
    \end{tikzpicture}
    \ ,
\end{equation}
where the second equality follows from the axioms of a spherical category.  We say that a morphism $f \in \Hom_\cC(X,Y)$ is \emph{negligible} if $\Tr(f \circ g) = 0$ for all $g \in \Hom_\cC(Y,X)$.  The negligible morphisms form a two-sided tensor ideal $\cN$ of $\cC$, and the quotient $\cC/\cN$ is called the \emph{semisimplification} of $\cC$.

\begin{theo} \label{negligible}
    The kernel of the functor $\Kar(\bF')$ is equal to the tensor ideal of negligible morphisms of $\Kar(\QSB')$.  The functor $\Kar(\bF')$ induces an equivalence of categories from the semisimplification of $\Kar(\QSB')$ to $U_q(N)\md$.
\end{theo}

\begin{proof}
    By \cref{easter,essential} the functor $\Kar(\bF')$ is full and essentially surjective.  It follows from \cref{blackhole} and \cite[Prop.~6.9]{SW22} that its kernel is the tensor ideal of negligible morphisms.
\end{proof}

\Cref{negligible} states that $\QSB'$ is an interpolating category for the categories of $U_q(N)$-modules.  The reason we need to consider the subcategory $\QSB'$ is that the proof of \cref{negligible} relies on \cref{blackhole}.  In the larger category $\QSB$, it does not seem possible to reduce all closed diagrams to a multiple of the empty diagram due to the possible presence of nontrivial links in the black strings.  In fact, we could have worked with the smaller category $\QSB'$ throughout the paper.  We have chosen to work with the larger category $\QSB$ because it has the nice property of being braided monoidal.

\appendix

\section{Proof of \cref{whitman}\label{sec:rope}}

Our goal is to show that the map $\tau$, defined in \cref{tau}, is a homomorphism of $U_q(N)$-modules.  We first need a technical lemma.

\begin{lem}
    For $2 \le i \le n$, we have
    \begin{gather} \label{demon1}
        \Psi_i \Psi_{i-1}^\dagger \psi_{i-1} = \psi_i \omega_i \omega_{i-1}^{-1} + \psi_{i-1} \Psi_i \Psi_{i-1}^\dagger,
        \\ \label{demon2}
        \Psi_{i-1} \Psi_i^\dagger \psi_i = \psi_{i-1} + q^{-1} \psi_i \Psi_{i-1} \Psi_i^\dagger.
    \end{gather}
\end{lem}

\begin{proof}
    To prove \cref{demon1}, we compute
    \begin{multline*}
        \Psi_i \Psi_{i-1}^\dagger \psi_{i-1}
        \overset{\cref{moca}}{=} q^{\rho_{i-1}} \Psi_i \Psi_{i-1}^\dagger \omega_{>i-1}^{-1} \Psi_{i-1}
        \overset{\cref{notary1}}{\overset{\cref{notary2}}{=}} q^{\rho_{i-1}} \omega_{>i-1}^{-1} \Psi_i \Psi_{i-1}^\dagger \Psi_{i-1} \omega_{i-1}^{-1}
        \\
        \overset{\cref{classicCl}}{=} q^{\rho_{i-1}} \omega_{>i-1}^{-1} \Psi_i (1 - \Psi_{i-1} \Psi_{i-1}^\dagger) \omega_{i-1}^{-1}
        \overset{\cref{notary1}}{\underset{\cref{classicCl}}{=}} \psi_i \omega_i \omega_{i-1}^{-1} + \psi_{i-1} \Psi_i \Psi_{i-1}^\dagger,
    \end{multline*}
    where we used the fact that $\rho_{i-1}+1=\rho_i$, by \cref{rhoeven,rhoodd}.  To prove \cref{demon2}, we compute
    \begin{multline*}
        \Psi_{i-1} \Psi_i^\dagger \psi_i
        \overset{\cref{moca}}{=} q^{\rho_i} \Psi_{i-1} \Psi_i^\dagger \omega_{>i}^{-1} \Psi_i
        \overset{\cref{notary2}}{=} q^{\rho_i} \omega_{>i}^{-1} \Psi_{i-1} \Psi_i^\dagger \Psi_i
        \overset{\cref{notary1}}{=} \psi_{i-1} \Psi_i^\dagger \Psi_i
        \\
        \overset{\cref{classicCl}}{=} \psi_{i-1} (1 - \Psi_i \Psi_i^\dagger)
        \overset{\cref{classicCl}}{=} \psi_{i-1} + q^{\rho_{i-1}} \omega_{>i-1}^{-1} \Psi_i \Psi_{i-1} \Psi_i^\dagger
        \overset{\cref{notary2}}{=} \psi_{i-1} + q^{-1} \psi_i \Psi_{i-1} \Psi_i^\dagger.
        \qedhere
    \end{multline*}
\end{proof}

To prove \cref{whitman}, we need to show that
\[
    a \tau(v \otimes x) = \tau( \Delta(a) (v \otimes x) ),\qquad
    a \in U_q(N),\ v \in V,\ x \in S.
\]
Using \cref{comult,Brazil,tau}, we see that it suffices to show that, for $1 \le i \le n$, $j \in \Vset$,
\begin{gather} \label{froste}
    e_i \psi_j = \overline{e_i v_j} k_i + \psi_j e_i,
    \\ \label{frostf}
    f_i \psi_j = \overline{f_i v_j} + \overline{k_i^{-1} v_j} f_i,
    \\ \label{frostk}
    k_i^{\pm 1} \psi_j = \overline{k_i^{\pm 1} v_j} k_i^{\pm 1},
    \\ \label{frosts}
    \sigma \psi_j = \overline{\sigma v_j} \sigma,
\end{gather}
where $\bar{v}$ denotes the image in $\Cl_q(N)$ of $v \in V$, and both sides of \cref{froste,frostf,frostk} are to be interpreted as elements of $\Cl_q(N)$, using \cref{ukulele}.

\begin{proof}[{Proof of \cref{froste}}]
    If $2 \le i \le n$, $0 \le j \le n$, $j \ne i-1$, we have
    \[
        e_i \psi_j
        \overset{\cref{moca}}{=} q^{\rho_j} \Psi_i \Psi_{i-1}^\dagger \omega_{>j}^{-1} \Psi_j
        \overset{\cref{notary2}}{=} q^{\rho_j} \omega_{>j}^{-1} \Psi_i \Psi_{i-1}^\dagger \Psi_j
        \overset{\cref{classicCl}}{=} q^{\rho_j} \omega_{>j}^{-1} \Psi_j \Psi_i \Psi_{i-1}^\dagger
        \overset{\cref{moca}}{\underset{\cref{notary1}}{=}} \psi_j e_i.
    \]
    Since $e_i v_j = 0$ by \cref{doze1}, the equation \cref{froste} follows.

    If $2 \le i \le n$, $0 < j \le n$, $j \ne i$, we have
    \[
        e_i \psi_{-j}
        \overset{\cref{moca}}{=} q^{-\rho_j} \Psi_i \Psi_{i-1}^\dagger \omega_{>j}^{-1} \Psi_j^\dagger
        \overset{\cref{notary2}}{=} q^{-\rho_j} \omega_{>j}^{-1} \Psi_i \Psi_{i-1}^\dagger \Psi_j^\dagger
        \overset{\cref{classicCl}}{=} q^{-\rho_j} \omega_{>j}^{-1} \Psi_j^\dagger \Psi_i \Psi_{i-1}^\dagger
        \overset{\cref{moca}}{\underset{\cref{notary1}}{=}} \psi_{-j} e_i,
    \]
    where we have used the fact that $\Psi_{i-1}^\dagger \Psi_j^\dagger = 0$, by \cref{classicCl}, when $j=i-1$.  Since $e_i v_{-j} = 0$ by \cref{doze1}, the equation \cref{froste} follows.

    The remaining cases with $i=1$ and $j\ne0$ follow by direct calculations using \cref{doze1,doze2,moca,notary1,notary2,classicCl}.
    \details{
        First note that, by \cref{doze1}, $e_1 v_j = 0$ unless $j \in \{-1,-2\}$ (when $N=2n$) or $j \in \{0,-1\}$ (when $N=2n+1$).   When $N=2n$, $i=1$, and $0 < j \le n$, we have
        \[
            e_1 \psi_j
            \overset{\cref{moca}}{=} q^{\rho_j} \Psi_2 \Psi_1 \omega_{>j}^{-1} \Psi_j
            \overset{\cref{notary2}}{\underset{\cref{classicCl}}{=}} q^{\rho_j} \omega_{>j}^{-1} \Psi_j \Psi_2 \Psi_1
            = \psi_j e_1.
        \]
        When $N=2n$, $i=1$, and $2 < j \le n$, we have
        \[
            e_1 \psi_{-j}
            \overset{\cref{moca}}{=} q^{-\rho_j} \Psi_2 \Psi_1 \omega_{>j}^{-1} \Psi_j^\dagger
            \overset{\cref{notary2}}{\underset{\cref{classicCl}}{=}} q^{-\rho_j} \omega_{>j}^{-1} \Psi_j^\dagger \Psi_2 \Psi_1
            = \psi_{-j} e_1.
        \]
        When $N=2n$, $i=1$ and $j=2$, we have
        \begin{multline*}
            e_1 \psi_{-2}
            \overset{\cref{moca}}{\underset{\cref{rhoeven}}{=}} q^{-1} \Psi_2 \Psi_1 \omega_{>2}^{-1} \Psi_2^\dagger
            \overset{\cref{notary2}}{\underset{\cref{classicCl}}{=}} - q^{-1} \omega_{>2}^{-1} \Psi_2 \Psi_2^\dagger \Psi_1
            \underset{\cref{classicCl}}{=} - q^{-1} \omega_{>2}^{-1} (1 - \Psi_2^\dagger \Psi_2) \Psi_1
            \\
            = - q^{-1} \omega_{>2}^{-1} \Psi_1 + q^{-1} \omega_{>2}^{-1} \Psi_2^\dagger \Psi_2 \Psi_1
            \\
            \overset{\cref{notary1}}{=} -\omega_{>1}^{-1} \Psi_1 \omega_2 \omega_1 + q^{-1} \omega_{>2}^{-1} \Psi_2^\dagger \Psi_2 \Psi_1
            \overset{\cref{moca}}{\underset{\cref{rhoeven}}{=}} -q^{-1} \psi_1 k_1 + \psi_{-2} \Psi_2 \Psi_1
            \overset{\cref{doze2}}{=} \overline{e_1 v_{-2}} k_1 + \psi_{-2} e_1
        \end{multline*}
        When $N=2n$ and $i=1$, we have
        \begin{multline*}
            e_1 \psi_{-1}
            \overset{\cref{moca}}{\underset{\cref{rhoeven}}{=}} \Psi_2 \Psi_1 \omega_{>1}^{-1} \Psi_1^\dagger
            \overset{\cref{notary1}}{\underset{\cref{notary2}}{=}} q \omega_{>1}^{-1} \Psi_2 \omega_1 \Psi_1 \Psi_1^\dagger
            \overset{\cref{classicCl}}{=} q \omega_{>1}^{-1} \Psi_2 \omega_1 (1 - \Psi_1^\dagger \Psi_1)
            \\
            \overset{\cref{classicCl}}{\underset{\cref{notary2}}{=}} q \omega_{>2}^{-1} \Psi_2 \omega_1 + q \omega_{>1}^{-1} \omega_1 \Psi_1^\dagger \Psi_2 \Psi_1
            \overset{\cref{notary1}}{=} q^2 \omega_{>2}^{-1} \Psi_2 \omega_2 \omega_1 + \omega_{>1}^{-1} \Psi_1^\dagger \Psi_2 \Psi_1
            \\
            \overset{\cref{moca}}{\underset{\cref{rhoeven}}{=}} q \psi_2 \omega_2 \omega_1 + \omega_{>1}^{-1} \Psi_1^\dagger \Psi_2 \Psi_1
            \overset{\cref{doze2}}{=} \overline{e_1 v_{-1}} k_1 + \psi_{-1} e_1.
        \end{multline*}

        When $N=2n+1$, $i=1$, and $j \ne 0$, we have
        \[
            e_1 \psi_j
            \overset{\cref{moca}}{=} q^{\rho_j} \Psi_1 \Psi_0^\dagger \omega_{>j}^{-1} \Psi_j
            \overset{\cref{notary2}}{\underset{\cref{classicCl}}{=}} q^{\rho_j} \omega_{>j}^{-1} \Psi_j \Psi_1 \Psi_0^\dagger
            = \psi_j e_1.
        \]
        When $N=2n+1$, $i=1$, and $1 < j \le n$, we have
        \[
            e_1 \psi_{-j}
            \overset{\cref{moca}}{=} q^{-\rho_j} \Psi_1 \Psi_0^\dagger \omega_{>j}^{-1} \Psi_j^\dagger
            \overset{\cref{notary2}}{\underset{\cref{classicCl}}{=}} q^{-\rho_j} \omega_{>j}^{-1} \Psi_j^\dagger \Psi_1 \Psi_0^\dagger
            = \psi_{-j} e_1.
        \]
        Finally, when $N=2n+1$ and $i=1$, we have
        \begin{multline*}
            e_1 \psi_{-1}
            \overset{\cref{moca}}{\underset{\cref{rhoodd}}{=}} q^{-\frac{1}{2}} \Psi_1 \Psi_0^\dagger \omega_{>1}^{-1} \Psi_1^\dagger
            \overset{\cref{notary2}}{\underset{\cref{classicCl}}{=}} - q^{-\frac{1}{2}} \omega_{>1}^{-1} \Psi_1 \Psi_1^\dagger \Psi_0
            \overset{\cref{notary1}}{\underset{\cref{classicCl}}{=}} - q^{-\frac{1}{2}} \omega_{>0}^{-1} (1 - \Psi_1^\dagger \Psi_1) \omega_1 \Psi_0
            \\
            \overset{\cref{notary2}}{=} -q^{-\frac{1}{2}} \omega_{>0}^{-1} \Psi_0 \omega_1 + q^{-\frac{1}{2}} \omega_{>1}^{-1} \Psi_1^\dagger \Psi_1 \Psi_0
            \overset{\cref{moca}}{\underset{\cref{rhoodd}}{=}} -q^{-1} \psi_0 k_1 + \psi_{-1} e_1
            \overset{\cref{doze2}}{=} \overline{e_1 v_{-1}} k_1 + \psi_{-1} e_1.
        \end{multline*}
    }

    Now suppose $2 \le i \le n$ and $j=i-1$.  Then
    \[
        e_i \psi_{i-1}
        \overset{\cref{moca}}{=} \Psi_i \Psi_{i-1}^\dagger \psi_{i-1}
        \overset{\cref{demon1}}{=} \psi_i \omega_i \omega_{i-1}^{-1} + \psi_{i-1} \Psi_i \Psi_{i-1}^\dagger
        \overset{\cref{doze1}}{=} \overline{e_i v_{i-1}} k_i + \psi_{i-1} e_i.
    \]

    If $2 \le i \le n$, $j=-i$, we have
    \begin{multline*}
        e_i \psi_{-i}
        \overset{\cref{moca}}{=} q^{-\rho_i} \Psi_i \Psi_{i-1}^\dagger \omega_{>i}^{-1} \Psi_i^\dagger
        \overset{\cref{notary1}}{\underset{\cref{notary2}}{=}} q^{-\rho_i} \omega_{>i}^{-1} \Psi_i \Psi_{i-1}^\dagger \Psi_i^\dagger
        \overset{\cref{classicCl}}{=} -q^{-\rho_i} \omega_{>i}^{-1} (1 - \Psi_i^\dagger \Psi_i) \Psi_{i-1}^\dagger
        \\
        \overset{\cref{notary1}}{=} - q^{-\rho_i} \omega_{>i}^{-1} \Psi_{i-1}^\dagger + q^{-\rho_i} \omega_{>i}^{-1} \Psi_i^\dagger \Psi_i \Psi_{i-1}^\dagger
        \overset{\cref{moca}}{=} - q^{-\rho_i} \omega_{>i}^{-1} \Psi_{i-1}^\dagger + \psi_i^\dagger e_i
        \\
        \overset{\cref{notary1}}{\underset{\cref{notary2}}{=}} -q^{-\rho_i} \omega_{>i-1}^{-1} \Psi_{i-1}^\dagger \omega_i \omega_{i-1}^{-1} + \psi_i^\dagger e_i
        \overset{\cref{moca}}{=} - q^{-1} \psi_{i-1}^\dagger k_i + \psi_i^\dagger e_i
        \overset{\cref{doze2}}{=} \overline{e_i v_{-i}} k_i + \psi_{-i} e_i.
    \end{multline*}

    Finally, we consider the case $N=2n+1$, $i=1$, $j=0$.  We have
    \begin{gather*}
        e_1 \psi_0
        \overset{\cref{moca}}{=} \Psi_1 \Psi_0^\dagger \omega_{>0}^{-1} \Psi_0
        \overset{\cref{notary1}}{\underset{\cref{notary2}}{=}} q \omega_{>1}^{-1} \Psi_1 \Psi_0^\dagger \Psi_0
        \overset{\cref{moca}}{\underset{\cref{classicCl}}{=}} q^{\frac{1}{2}} \psi_1,
        \\
        \psi_0 e_1
        \overset{\cref{moca}}{=} \omega_{>0}^{-1} \Psi_0 \Psi_1 \Psi_0^\dagger
        \overset{\cref{classicCl}}{=} - \omega_{>0}^{-1} \Psi_1
        \overset{\cref{notary1}}{=} - \omega_{>1}^{-1} \Psi_1
        \overset{\cref{moca}}{=} - q^{-\frac{1}{2}} \psi_1,
        \\
        \overline{e_1 v_0} k_1
        \overset{\cref{doze2}}{=} q^{\frac{1}{2}} (q+1) \psi_1 \omega_1
        \overset{\cref{notary1}}{=} \left( q^{\frac{1}{2}} + q^{-\frac{1}{2}} \right) \psi_1,
    \end{gather*}
    and so \cref{froste} follows.
\end{proof}

\begin{proof}[{Proof of \cref{frostf}}]
    If $2 \le i \le n$, $0 \le j \le n$, $j \ne i$, we have
    \[
        f_i \psi_j
        \overset{\cref{moca}}{=} q^{\rho_j} \Psi_{i-1} \Psi_i^\dagger \omega_{>j}^{-1} \Psi_j
        \overset{\cref{notary1}}{\underset{\cref{notary2}}{=}} q^{\rho_j + \delta_{i-1,j}} \omega_{>j}^{-1} \Psi_{i-1} \Psi_i^\dagger \Psi_j
        \overset{\cref{classicCl}}{=} q^{\rho_j} \omega_{>j}^{-1} \Psi_j \Psi_{i-1} \Psi_i^\dagger
        \overset{\cref{moca}}{=} \psi_j f_i.
    \]
    Since $f_i v_j = 0$ and $k_i^{-1} v_j = q^{\delta_{i-1,j}} v_j$ by \cref{doze1,doze4}, the equation \cref{frostf} follows when $j \ne i-1$.  When $j=i-1$, both sides of \cref{frostf} are equal to zero, since $\Psi_{i-1} \Psi_i^\dagger \Psi_{i-1} = - \Psi_{i-1} \Psi_{i-1} \Psi_i^\dagger = 0$ by \cref{classicCl}.  If $2 \le i=j \le n$, then
    \[
        f_i \psi_i
        = \Psi_{i-1} \Psi_i^\dagger \psi_i
        \overset{\cref{demon2}}{=} \psi_{i-1} + q^{-1} \psi_i \Psi_{i-1} \Psi_i^\dagger
        \overset{\cref{doze1}}{\underset{\cref{doze4}}{=}} \overline{f_i v_i} + \overline{k_i^{-1} v_i} f_i,
    \]
    as desired.

    If $2 \le i \le n$, $0 < j \le n$, $j \ne i-1$, we have
    \[
        f_i \psi_{-j}
        \overset{\cref{moca}}{=} q^{-\rho_j} \Psi_{i-1} \Psi_i^\dagger \omega_{>j}^{-1} \Psi_j^\dagger
        \overset{\cref{notary2}}{\underset{\cref{classicCl}}{=}} q^{-\rho_j} \omega_{>j}^{-1} \Psi_j^\dagger \Psi_{i-1} \Psi_i^\dagger
        \overset{\cref{moca}}{=} \psi_{-j} f_i.
    \]
    By \cref{doze1,doze4}, we have $f_i v_{-j} = 0$ and $k_i^{-1} v_{-j} = q^{\delta_{ij}} v_{-j}$.  When $j \ne i$, the equation \cref{frostf} follows immediately.  When $j = i$, we have
    \[
        f_i \psi_{-i}
        = \psi_{-i} f_i
        \overset{\cref{moca}}{=} q^{-\rho_i} \omega_{>i}^{-1} \Psi_i^\dagger \Psi_{i-1} \Psi_i^\dagger
        \overset{\cref{classicCl}}{=} 0,
    \]
    and so both sides of \cref{frostf} are equal to zero.  We also have
    \begin{multline*}
        f_i \psi_{1-i}
        \overset{\cref{moca}}{=} q^{-\rho_{i-1}} \Psi_{i-1} \Psi_i^\dagger \omega_{>i-1}^{-1} \Psi_{i-1}^\dagger
        \overset{\cref{notary2}}{\underset{\cref{classicCl}}{=}} q^{-1-\rho_{i-1}} \omega_{>i-1}^{-1} \Psi_{i-1} \Psi_{i-1}^\dagger \Psi_i^\dagger
        \\
        \overset{\cref{classicCl}}{=} q^{-1-\rho_{i-1}} \omega_{>i-1}^{-1} (1 - \Psi_{i-1}^\dagger \Psi_{i-1}) \Psi_i^\dagger
        = q^{-\rho_i} \omega_{>i-1}^{-1} \Psi_i^\dagger + q^{-1-\rho_{i-1}} \omega_{>i-1}^{-1} \Psi_{i-1}^\dagger \Psi_{i-1} \Psi_i^\dagger
        \\
        \overset{\cref{notary1}}{=} - q^{1-\rho_i} \omega_{>i}^{-1} \Psi_i^\dagger + q^{-1-\rho_{i-1}} \omega_{>i-1}^{-1} \Psi_{i-1}^\dagger \Psi_{i-1} \Psi_i^\dagger
        \overset{\cref{moca}}{=} -q \psi_{-i} + q^{-1} \psi_{1-i} f_i
        \overset{\cref{doze1}}{=} \overline{f_i v_{1-i}} + \overline{k_i^{-1}v_{1-i}} f_i,
    \end{multline*}
    as desired.

    Now suppose that $N=2n$ and $i=1$.  When $j>2$, we have
    \[
        f_1 \psi_j
        = \Psi_1^\dagger \Psi_2^\dagger \psi_j
        \overset{\cref{moca}}{\underset{\cref{notary2},\cref{classicCl}}{=}} \psi_j \Psi_1^\dagger \Psi_2^\dagger
        = \psi_j f_1
        \overset{\cref{doze1}}{\underset{\cref{doze4}}{=}} \overline{f_1 v_j} + \overline{k_1^{-1} v_j} f_1.
    \]
    When $j=2$, we have
    \begin{multline*}
        f_1 \psi_2
        \overset{\cref{moca}}{=} q \Psi_1^\dagger \Psi_2^\dagger \omega_{>2}^{-1} \Psi_2
        \overset{\cref{notary2}}{\underset{\cref{notary1}}{=}} \omega_{>1}^{-1} \Psi_1^\dagger \Psi_2^\dagger \Psi_2
        \overset{\cref{classicCl}}{=} \omega_{>1}^{-1} \Psi_1^\dagger (1 - \Psi_2 \Psi_2^\dagger)
        \\
        \overset{\cref{moca}}{\underset{\cref{classicCl}}{=}} \psi_1^\dagger + \omega_{>1}^{-1} \Psi_2 \Psi_1^\dagger \Psi_2^\dagger
        \overset{\cref{notary1}}{\underset{\cref{moca}}{=}} \psi_1^\dagger + q^{-1} \psi_2 f_1
        \overset{\cref{doze3}}{=} \overline{f_1 v_2} + \overline{k_1^{-1} v_2} f_1.
    \end{multline*}
    When $j=1$, we have
    \begin{multline*}
        f_1 \psi_1
        \overset{\cref{moca}}{=} \Psi_1^\dagger \Psi_2^\dagger \omega_{>1}^{-1} \Psi_1
        \overset{\cref{notary2}}{=} q^{-1} \omega_{>1}^{-1} \Psi_1^\dagger \Psi_2^\dagger \Psi_1
        \overset{\cref{classicCl}}{=} - q^{-1} \omega_{>1}^{-1} \Psi_2^\dagger (1 - \Psi_1 \Psi_1^\dagger)
        \\
        \overset{\cref{classicCl}}{\underset{\cref{notary1}}{=}} - \omega_{>2}^{-1} \Psi_2^\dagger + q^{-1} \omega_{>1}^{-1} \Psi_1 \Psi_1^\dagger \Psi_2^\dagger
        \overset{\cref{moca}}{=} - q \psi_2^\dagger + q^{-1} \psi_1 f_1
        \overset{\cref{doze3}}{\underset{\cref{doze4}}{=}} \overline{f_1 v_1} + \overline{k_1^{-1} v_1} f_1.
    \end{multline*}
    When $2 < j \le n$, we have
    \[
        f_1 \psi_{-j}
        \overset{\cref{moca}}{=} q^{-\rho_j} \Psi_1^\dagger \Psi_2^\dagger \omega_{>j}^{-1} \Psi_j^\dagger
        \overset{\cref{notary2}}{\underset{\cref{classicCl}}{=}} q^{-\rho_j} \omega_{>j}^{-1} \Psi_j^\dagger \Psi_1^\dagger \Psi_2^\dagger
        \overset{\cref{moca}}{=} \psi_{-j} f_1
        \overset{\cref{doze3}}{\underset{\cref{doze4}}{=}} \overline{f_1 v_{-j}} + \overline{k_1^{-1} v_{-j}} f_1.
    \]
    When $j \in \{1,2\}$, we have
    \[
        f_1 \psi_{-j}
        \overset{\cref{moca}}{=} q^{-\rho_j} \Psi_1^\dagger \Psi_2^\dagger \omega_{>j}^{-1} \Psi_j^\dagger
        \overset{\cref{notary2}}{\underset{\cref{classicCl}}{=}} 0
        \overset{\cref{classicCl}}{=} q^{-\rho_j + 1} \omega_{>1}^{-1} \Psi_1^\dagger \Psi_1^\dagger \Psi_2^\dagger
        \overset{\cref{moca}}{=} q \psi_{-1} f_1
        \overset{\cref{doze3}}{\underset{\cref{doze4}}{=}} \overline{f_1 v_{-j}} + \overline{k_1^{-1} v_{-j}} f_1.
    \]

    Now suppose that $N=2n+1$ and $i=1$.  When $j>1$, we have
    \[
        f_1 \psi_j
        \overset{\cref{moca}}{=} q^{\rho_j} \Psi_0 \Psi_1^\dagger \omega_{>j}^{-1} \Psi_j
        \overset{\cref{notary2}}{\underset{\cref{classicCl}}{=}} q^{\rho_j} \omega_{>j}^{-1} \Psi_j \Psi_0 \Psi_1^\dagger
        \overset{\cref{moca}}{=} \psi_j f_1
        \overset{\cref{doze3}}{\underset{\cref{doze4}}{=}} \overline{f_1 v_j} + \overline{k_1^{-1} v_j} f_1.
    \]
    When $j=1$, we have
    \begin{multline*}
        f_1 \psi_1
        \overset{\cref{moca}}{\underset{\cref{rhoodd}}{=}} q^{\frac{1}{2}} \Psi_0 \Psi_1^\dagger \omega_{>1}^{-1} \Psi_1
        \overset{\cref{notary1}}{\underset{\cref{notary2}}{=}} q^{-\frac{1}{2}} \omega_{>0}^{-1} \Psi_0 \Psi_1^\dagger \Psi_1
        \overset{\cref{classicCl}}{=} q^{-\frac{1}{2}} \omega_{>0}^{-1} \Psi_0 (1 - \Psi_1 \Psi_1^\dagger)
        \\
        \overset{\cref{moca}}{\underset{\cref{classicCl}}{=}} q^{-\frac{1}{2}} \psi_0 + q^{-\frac{1}{2}} \omega_{>0}^{-1} \Psi_1 \Psi_0 \Psi_1^\dagger
        \overset{\cref{notary1}}{\underset{\cref{moca}}{=}} q^{-\frac{1}{2}} \psi_0 + q^{-1} \psi_1 f_1
        \overset{\cref{doze3}}{\underset{\cref{doze4}}{=}} \overline{f_1 v_1} + \overline{k_1^{-1} v_1} f_1.
    \end{multline*}
    When $j=0$, we have
    \begin{gather*}
        f_1 \psi_0
        \overset{\cref{moca}}{=} \Psi_0 \Psi_1^\dagger \omega_{>0}^{-1} \Psi_0
        \overset{\cref{notary2}}{=} q^{-1} \omega_{>0}^{-1} \Psi_0 \Psi_1^\dagger \Psi_0
        \overset{\cref{classicCl}}{=} - q^{-1} \omega_{>0}^{-1} \Psi_1^\dagger
        \overset{\cref{notary1}}{\underset{\cref{moca}}{=}} - q^{\frac{1}{2}} \psi_1^\dagger,
        \\
        \overline{f_1 v_0}
        \overset{\cref{doze3}}{=} -q^{\frac{1}{2}}(q+1) \psi_1^\dagger,
        \\
        \overline{k_1^{-1} v_0} f_1
        \overset{\cref{doze4}}{=} \psi_0 f_1
        \overset{\cref{moca}}{=} \omega_{>0}^{-1} \Psi_0 \Psi_0 \Psi_1^\dagger
        \overset{\cref{classicCl}}{=} \omega_{>0}^{-1} \Psi_1^\dagger
        \overset{\cref{notary1}}{=} q \omega_{>1}^{-1} \Psi_1^\dagger
        \overset{\cref{moca}}{=} q^{\frac{3}{2}} \psi_1^\dagger,
    \end{gather*}
    and \cref{frostf} follows.  When $2 \le j \le n$, we have
    \[
        f_1 \psi_{-j}
        \overset{\cref{moca}}{=} q^{-\rho_j} \Psi_0 \Psi_1^\dagger \omega_{>j}^{-1} \Psi_j^\dagger
        \overset{\cref{notary2}}{\underset{\cref{classicCl}}{=}} q^{-\rho_j} \omega_{>j}^{-1} \Psi_j^\dagger \Psi_0 \Psi_1^\dagger
        \overset{\cref{moca}}{=} \psi_{-j} f_1
        \overset{\cref{doze3}}{\underset{\cref{doze4}}{=}} \overline{f_1 v_{-j}} + \overline{k_1^{-1} v_{-j}} f_1.
    \]
    Finally, we have
    \[
        f_1 \psi_{-1}
        \overset{\cref{moca}}{=} q^{-\rho_1} \Psi_0 \Psi_1^\dagger \omega_{>1}^{-1} \Psi_1^\dagger
        \overset{\cref{notary2}}{\underset{\cref{classicCl}}{=}} 0
        \overset{\cref{classicCl}}{=} q^{1 - \rho_1} \omega_{>1}^{-1} \Psi_1^\dagger \Psi_0 \Psi_1^\dagger
        \overset{\cref{moca}}{=} q \psi_{-1} f_1
        \overset{\cref{doze3}}{\underset{\cref{doze4}}{=}} \overline{f_1 v_{-1}} + \overline{k_1^{-1} v_{-1}} f_1.
    \]
\end{proof}

\begin{proof}[{Proof of \cref{frostk}}]
    For $2 \le i \le n$, $0 \le j \le n$, we have
    \begin{multline*}
        k_i^{\pm 1} \psi_j
        \overset{\cref{moca}}{=} q^{\rho_j} \omega_i^{\pm 1} \omega_{i-1}^{\mp 1} \omega_{>j}^{-1} \Psi_j
        \overset{\cref{notary2}}{=} q^{\rho_j \pm \delta_{i,j} \mp \delta_{i-1,j}} \omega_{>j}^{-1} \Psi_j\omega_i^{\pm 1} \omega_{i-1}^{\mp 1}
        \\
        \overset{\cref{moca}}{=} q^{\pm \delta_{i,j} \mp \delta_{i-1,j}} \psi_j \omega_i^{\pm 1}\omega_{i-1}^{\mp 1}
        \overset{\cref{doze4}}{=} \overline{k_i^{\pm 1} v_j} k_i^{\pm 1}.
    \end{multline*}
    The case $2 \le i \le n$, $-n \le j < 0$ is analogous.

    For $N=2n$, $i=1$, and $0 \le j \le n$, we have
    \begin{multline*}
        k_1^{\pm 1} \psi_j
        \overset{\cref{moca}}{=} q^{\rho_j \pm 1} \omega_2^{\pm 1} \omega_1^{\pm 1} \omega_{>j}^{-1} \Psi_j
        \overset{\cref{notary2}}{=} q^{\rho_j \pm 1 \pm \delta_{j,1} \pm \delta_{j,2}} \omega_{>j}^{-1} \Psi_j \omega_2^{\pm 1} \omega_1^{\pm 1}
        \\
        \overset{\cref{moca}}{=} q^{\pm 1 \pm \delta_{j,1} \pm \delta_{j,2}} \psi_j \omega_2^{\pm 1} \omega_1^{\pm 1}
        \overset{\cref{doze4}}{=} \overline{k_1^{\pm 1} v_j} k_1^{\pm 1}.
    \end{multline*}
    The case $N=2n$, $i=1$, and $-n \le j < 0$ is analogous.  The cases $N=2n+1$, $i=1$, and $j \in \Vset$ are also similar.
    \details{
        When $N=2n+1$, $i=1$, and $0 \le j \le n$, we have
        \[
            k_1^{\pm 1} \psi_j
            \overset{\cref{moca}}{=} q^{\rho_j \pm \frac{1}{2}} \omega_1^{\pm 1} \omega_{>j}^{-1} \Psi_j
            \overset{\cref{notary2}}{=} q^{\rho_j \pm \frac{1}{2} \pm \delta_{j,1}} \omega_{>j}^{-1} \Psi_j \omega_1^{\pm 1}
            \overset{\cref{moca}}{=} q^{\pm \frac{1}{2} \pm \delta_{j,1}} \psi_j \omega_1^{\pm 1}
            \overset{\cref{doze4}}{=} \overline{k_1^{\pm 1} v_j} k_1^{\pm 1}.
        \]
    }
\end{proof}

\begin{proof}[{Proof of \cref{frosts}}]
    For $2 \le j \le n$, we have
    \[
        \sigma \psi_j
        \overset{\cref{moca}}{=} \sqrt{-1} q^{\rho_j} (\Psi_1 - \Psi_1^\dagger) \omega_{>j}^{-1} \Psi_j
        \overset{\cref{notary2}}{\underset{\cref{classicCl}}{=}} - \sqrt{-1} q^{\rho_j} \omega_{>j}^{-1} \Psi_j (\Psi_1 - \Psi_1^\dagger)
        \overset{\cref{moca}}{=} -\psi_j \sigma
        \overset{\cref{doze5}}{=} \overline{\sigma v_j} \sigma.
    \]
    The case $-n \le j \le -2$ is analogous.  Finally, we compute
    \begin{multline*}
        \sigma \psi_1
        \overset{\cref{moca}}{=} \sqrt{-1} (\Psi_1 - \Psi_1^\dagger) \omega_{>1}^{-1} \Psi_1
        \overset{\cref{notary2}}{\underset{\cref{classicCl}}{=}} - \sqrt{-1} \omega_{>1}^{-1} \Psi_1^\dagger \Psi_1
        \\
        \overset{\cref{classicCl}}{=} - \sqrt{-1} \omega_{>1}^{-1} \Psi_1^\dagger (\Psi_1 - \Psi_1^\dagger)
        \overset{\cref{moca}}{=} - \psi_1^\dagger \sigma
        \overset{\cref{doze5}}{=} \overline{\sigma v_1} \sigma
    \end{multline*}
    and
    \begin{multline*}
        \sigma \psi_1^\dagger
        \overset{\cref{moca}}{=} \sqrt{-1} (\Psi_1 - \Psi_1^\dagger) \omega_{>1}^{-1} \Psi_1^\dagger
        \overset{\cref{notary2}}{\underset{\cref{classicCl}}{=}} \sqrt{-1} \omega_{>1}^{-1} \Psi_1 \Psi_1^\dagger
        \\
        \overset{\cref{classicCl}}{=} - \sqrt{-1} \omega_{>1}^{-1} \Psi_1 (\Psi_1 - \Psi_1^\dagger)
        \overset{\cref{moca}}{=} - \psi_1 \sigma
        \overset{\cref{doze5}}{=} \overline{\sigma v_{-1}} \sigma.
        \qedhere
    \end{multline*}
\end{proof}

This completes the proof of \cref{whitman}.


\bibliographystyle{alphaurl}
\bibliography{QuantumSpinBrauer}

\newcommand\Salabel{}\newcommand\Sa[2]{Sage}
\begin{thebibliography}{CKM14}

\bibitem[Abo22]{Abo22}
W.~Aboumrad.
\newblock Skew {H}owe duality for types {$BD$} via $q$-{C}lifford algebras.
\newblock 2022.
\newblock \href {https://arxiv.org/abs/2208.09773} {\path{arXiv:2208.09773}}.

\bibitem[BER24]{BER24}
E.~Bodish, B.~Elias, and D.~E.~V. Rose.
\newblock Spin link homology.
\newblock 2024.
\newblock \href {https://arxiv.org/abs/2407.00189} {\path{arXiv:2407.00189}}.

\bibitem[BW25]{BW23}
E.~Bodish and H.~Wu.
\newblock Webs for the quantum orthogonal group.
\newblock {\em Adv. Math.}, 480:Paper No. 110514, 65, 2025.
\newblock \href {https://arxiv.org/abs/2309.03623} {\path{arXiv:2309.03623}},
  \href {https://doi.org/10.1016/j.aim.2025.110514}
  {\path{doi:10.1016/j.aim.2025.110514}}.

\bibitem[CKM14]{CKM14}
Sabin Cautis, Joel Kamnitzer, and Scott Morrison.
\newblock Webs and quantum skew {H}owe duality.
\newblock {\em Math. Ann.}, 360(1-2):351--390, 2014.
\newblock \href {https://arxiv.org/abs/1210.6437} {\path{arXiv:1210.6437}},
  \href {https://doi.org/10.1007/s00208-013-0984-4}
  {\path{doi:10.1007/s00208-013-0984-4}}.

\bibitem[CP95]{CP95}
V.~Chari and A.~Pressley.
\newblock {\em A guide to quantum groups}.
\newblock Cambridge University Press, Cambridge, 1995.
\newblock Corrected reprint of the 1994 original.

\bibitem[Del07]{Del07}
P.~Deligne.
\newblock La cat\'{e}gorie des repr\'{e}sentations du groupe sym\'{e}trique
  {$S_t$}, lorsque {$t$} n'est pas un entier naturel.
\newblock In {\em Algebraic groups and homogeneous spaces}, volume~19 of {\em
  Tata Inst. Fund. Res. Stud. Math.}, pages 209--273. Tata Inst. Fund. Res.,
  Mumbai, 2007.

\bibitem[DF94]{DF94}
J.~T. Ding and Igor~B. Frenkel.
\newblock Spinor and oscillator representations of quantum groups.
\newblock In {\em Lie theory and geometry}, volume 123 of {\em Progr. Math.},
  pages 127--165. Birkh\"{a}user Boston, Boston, MA, 1994.
\newblock \href {https://doi.org/10.1007/978-1-4612-0261-5\_5}
  {\path{doi:10.1007/978-1-4612-0261-5\_5}}.

\bibitem[DHS13]{DHS13}
R.~Dipper, J.~Hu, and F.~Stoll.
\newblock Symmetrizers and antisymmetrizers for the {BMW}-algebra.
\newblock {\em J. Algebra Appl.}, 12(7):1350032, 22, 2013.
\newblock \href {https://arxiv.org/abs/1109.0342} {\path{arXiv:1109.0342}},
  \href {https://doi.org/10.1142/S0219498813500321}
  {\path{doi:10.1142/S0219498813500321}}.

\bibitem[GK91]{GK91}
A.~M. Gavrilik and A.~U. Klimyk.
\newblock {$q$}-deformed orthogonal and pseudo-orthogonal algebras and their
  representations.
\newblock {\em Lett. Math. Phys.}, 21(3):215--220, 1991.
\newblock \href {https://arxiv.org/abs/math/0203201}
  {\path{arXiv:math/0203201}}, \href {https://doi.org/10.1007/BF00420371}
  {\path{doi:10.1007/BF00420371}}.

\bibitem[GRS22]{GRS22}
M.~Gao, H.~Rui, and L.~Song.
\newblock A basis theorem for the affine {K}auffman category and its cyclotomic
  quotients.
\newblock {\em J. Algebra}, 608:774--846, 2022.
\newblock \href {https://arxiv.org/abs/2006.09626} {\path{arXiv:2006.09626}},
  \href {https://doi.org/10.1016/j.jalgebra.2022.07.005}
  {\path{doi:10.1016/j.jalgebra.2022.07.005}}.

\bibitem[Hay90]{Hay90}
T.~Hayashi.
\newblock {$q$}-analogues of {C}lifford and {W}eyl algebras---spinor and
  oscillator representations of quantum enveloping algebras.
\newblock {\em Comm. Math. Phys.}, 127(1):129--144, 1990.
\newblock URL: \url{http://projecteuclid.org/euclid.cmp/1104180043}.

\bibitem[HS99]{HS99}
I.~Heckenberger and A.~Sch\"{u}ler.
\newblock Symmetrizer and antisymmetrizer of the {B}irman-{W}enzl-{M}urakami
  algebras.
\newblock {\em Lett. Math. Phys.}, 50(1):45--51, 1999.
\newblock \href {https://arxiv.org/abs/math/0002170}
  {\path{arXiv:math/0002170}}, \href {https://doi.org/10.1023/A:1007675821808}
  {\path{doi:10.1023/A:1007675821808}}.

\bibitem[KM23]{KM23}
C.~Keller and L.~M\"uller.
\newblock Finite symmetries of quantum character stacks.
\newblock {\em Theory Appl. Categ.}, 39:Paper No. 3, 51--97, 2023.
\newblock \href {https://arxiv.org/abs/2107.12348} {\path{arXiv:2107.12348}},
  \href {https://doi.org/10.70930/tac/975tnnjc}
  {\path{doi:10.70930/tac/975tnnjc}}.

\bibitem[KS97]{KS97}
A.~Klimyk and K.~Schm\"udgen.
\newblock {\em Quantum groups and their representations}.
\newblock Texts and Monographs in Physics. Springer-Verlag, Berlin, 1997.
\newblock \href {https://doi.org/10.1007/978-3-642-60896-4}
  {\path{doi:10.1007/978-3-642-60896-4}}.

\bibitem[Let97]{Let97}
G.~Letzter.
\newblock Subalgebras which appear in quantum {I}wasawa decompositions.
\newblock {\em Canad. J. Math.}, 49(6):1206--1223, 1997.
\newblock \href {https://doi.org/10.4153/CJM-1997-059-4}
  {\path{doi:10.4153/CJM-1997-059-4}}.

\bibitem[Lus10]{Lus10}
G.~Lusztig.
\newblock {\em Introduction to quantum groups}.
\newblock Modern Birkh\"{a}user Classics. Birkh\"{a}user/Springer, New York,
  2010.
\newblock Reprint of the 1994 edition.
\newblock \href {https://doi.org/10.1007/978-0-8176-4717-9}
  {\path{doi:10.1007/978-0-8176-4717-9}}.

\bibitem[LZ15]{LZ15}
G.~I. Lehrer and R.~B. Zhang.
\newblock The {B}rauer category and invariant theory.
\newblock {\em J. Eur. Math. Soc. (JEMS)}, 17(9):2311--2351, 2015.
\newblock \href {https://arxiv.org/abs/1207.5889} {\path{arXiv:1207.5889}},
  \href {https://doi.org/10.4171/JEMS/558} {\path{doi:10.4171/JEMS/558}}.

\bibitem[MS21]{MS21}
Y.~Mousaaid and A.~Savage.
\newblock Affinization of monoidal categories.
\newblock {\em J. \'{E}c. polytech. Math.}, 8:791--829, 2021.
\newblock \href {https://arxiv.org/abs/2010.13598} {\path{arXiv:2010.13598}},
  \href {https://doi.org/10.5802/jep.158} {\path{doi:10.5802/jep.158}}.

\bibitem[MS24]{MS24}
P.~J. McNamara and A.~Savage.
\newblock The spin {B}rauer category.
\newblock {\em Forum Math. Sigma}, 12:Paper No. e98, 2024.
\newblock \href {https://arxiv.org/abs/2312.11766} {\path{arXiv:2312.11766}},
  \href {https://doi.org/10.1017/fms.2024.102}
  {\path{doi:10.1017/fms.2024.102}}.

\bibitem[NS95]{NS95}
M.~Noumi and T.~Sugitani.
\newblock Quantum symmetric spaces and related {$q$}-orthogonal polynomials.
\newblock In {\em Group theoretical methods in physics ({T}oyonaka, 1994)},
  pages 28--40. World Sci. Publ., River Edge, NJ, 1995.
\newblock \href {https://arxiv.org/abs/math/9503225}
  {\path{arXiv:math/9503225}}.

\bibitem[OW02]{OW02}
R.~C. Orellana and H.~G. Wenzl.
\newblock {$q$}-centralizer algebras for spin groups.
\newblock {\em J. Algebra}, 253(2):237--275, 2002.
\newblock \href {https://doi.org/10.1016/S0021-8693(02)00069-8}
  {\path{doi:10.1016/S0021-8693(02)00069-8}}.

\bibitem[\Sa22]{sagemath}
\Salabel{{The Sage Developers}}.
\newblock {\em {S}ageMath, the {S}age {M}athematics {S}oftware {S}ystem
  ({V}ersion 9.5)}, 2022.
\newblock URL: \url{https://www.sagemath.org}, \href
  {https://doi.org/10.5281/zenodo.6259615} {\path{doi:10.5281/zenodo.6259615}}.

\bibitem[Sel11]{Sel11}
P.~Selinger.
\newblock A survey of graphical languages for monoidal categories.
\newblock In {\em New structures for physics}, volume 813 of {\em Lecture Notes
  in Phys.}, pages 289--355. Springer, Heidelberg, 2011.
\newblock \href {https://arxiv.org/abs/0908.3347} {\path{arXiv:0908.3347}},
  \href {https://doi.org/10.1007/978-3-642-12821-9\_4}
  {\path{doi:10.1007/978-3-642-12821-9\_4}}.

\bibitem[ST19]{ST19}
A.~Sartori and D.~Tubbenhauer.
\newblock Webs and {$q$}-{H}owe dualities in types {BCD}.
\newblock {\em Trans. Amer. Math. Soc.}, 371(10):7387--7431, 2019.
\newblock \href {https://arxiv.org/abs/1701.02932} {\path{arXiv:1701.02932}},
  \href {https://doi.org/10.1090/tran/7583} {\path{doi:10.1090/tran/7583}}.

\bibitem[SW24]{SW22}
A.~Savage and B.~W. Westbury.
\newblock Quantum diagrammatics for {$F_4$}.
\newblock {\em J. Pure Appl. Algebra}, 228(11):Paper No. 107731, 35, 2024.
\newblock \href {https://arxiv.org/abs/2204.11976} {\path{arXiv:2204.11976}},
  \href {https://doi.org/10.1016/j.jpaa.2024.107731}
  {\path{doi:10.1016/j.jpaa.2024.107731}}.

\bibitem[Tur89]{Tur89}
V.~G. Turaev.
\newblock Operator invariants of tangles, and {$R$}-matrices.
\newblock {\em Izv. Akad. Nauk SSSR Ser. Mat.}, 53(5):1073--1107, 1135, 1989.
\newblock \href {https://doi.org/10.1070/IM1990v035n02ABEH000711}
  {\path{doi:10.1070/IM1990v035n02ABEH000711}}.

\bibitem[TW05]{TW05}
I.~Tuba and H.~Wenzl.
\newblock On braided tensor categories of type {$BCD$}.
\newblock {\em J. Reine Angew. Math.}, 581:31--69, 2005.
\newblock \href {https://arxiv.org/abs/math/0301142}
  {\path{arXiv:math/0301142}}, \href
  {https://doi.org/10.1515/crll.2005.2005.581.31}
  {\path{doi:10.1515/crll.2005.2005.581.31}}.

\bibitem[Wen12]{Wen12}
H.~Wenzl.
\newblock On centralizer algebras for spin representations.
\newblock {\em Comm. Math. Phys.}, 314(1):243--263, 2012.
\newblock \href {https://arxiv.org/abs/1107.4183} {\path{arXiv:1107.4183}},
  \href {https://doi.org/10.1007/s00220-012-1494-z}
  {\path{doi:10.1007/s00220-012-1494-z}}.

\bibitem[Wen20]{Wen20}
H.~Wenzl.
\newblock Dualities for spin representations.
\newblock 2020.
\newblock \href {https://arxiv.org/abs/2005.11299} {\path{arXiv:2005.11299}}.

\bibitem[Wes08]{Wes08}
B.~W. Westbury.
\newblock Invariant tensors for the spin representation of
  {$\mathfrak{so}(7)$}.
\newblock {\em Math. Proc. Cambridge Philos. Soc.}, 144(1):217--240, 2008.
\newblock \href {https://arxiv.org/abs/math/0601209}
  {\path{arXiv:math/0601209}}, \href
  {https://doi.org/10.1017/S0305004107000722}
  {\path{doi:10.1017/S0305004107000722}}.

\end{thebibliography}

\end{document}